\documentclass[12pt]{iopart}

\usepackage{iopams}
\usepackage{graphicx,cite}
\usepackage{multirow}
\usepackage{theorem}
\usepackage{algorithm}
\usepackage{enumitem}
\usepackage[usenames]{color}

\newtheorem{Def}{Definition}
\newtheorem{Prop}{Proposition}
\newtheorem{Rem}{Remark}
\newtheorem{Thm}{Theorem}
\newtheorem{Lemma}{Lemma}
\newtheorem{Cor}{Corollary}
\newtheorem{assumptions}{Assumption}
\newcommand{\N}{\mathbb{N}}
\newcommand{\R}{\mathbb{R}}
\newcommand{\sub}{\partial}
\newenvironment{proof}{{\it Proof.}}{\hfill$\square$\vspace{0.3cm}\\}
\def\hk{h^{(k)}}
\def\thk{h^{(k)}_\gamma}
\newcommand{\projalg}[1]{\ifcase #1\or SGP\or PDHG\or PDHG-D\or CBGP\or SGP\or SS\or WSS2}%
\newcommand{\eqref}[1]{(\ref{#1})}%

\usepackage{enumitem}

\newlist{algorithmsteps}{enumerate}{1}
\setlist[algorithmsteps,1]{
  label={\textsc{Step}~\arabic*},
  leftmargin=*,
  align=left,
  labelsep=2mm,
}

\newcommand{\ve}[1]{ #1}
\def\x{ {\ve x}}  

\def\y{ {\ve y}}
\def\u{ {\ve u}}\def\lamk{{\lambda_k}}
\def\w{{\ve w}}
\def\z{ {\ve z}}
\def\g{ {\ve g}}
\def\d{{\ve d}}
\def\xk{ \x^{(k)}}
\def\yk{ \y^{(k)}}

\def\zk{ \z^{(k)}}
\def\yk{ \y^{(k)}}
\def\dk{ \d^{(k)}}
\def\xkk{ \x^{(k+1)}}

\def\tykm{ \ty^{(k-1)}}

\def\ak{{\alpha_k}}

\def\kinN{{k\in\mathbb N}}

\def\dom{\mathrm{dom}}

\def\R{\mathbb R}

\def\kinN{{k\in\N}}
\def\prox{{\mathrm{prox}}}
\def\k{^{(k)}}

\def\tyk{\tilde \y^{(k)}}

\def\v{\ve v}  \def\vk{\v^{(k)}}
\def\u{\ve u}

\def\w{\ve w}

\def\ty{\tilde \y}

\def\bR{\bar {\R}} \def\dist{\mathrm{dist}}
\def\wk{\w^{(k)}}
  
\def\c{\ve c}
\def\uz{{\ve u}^0}
\def\mytilde{\raise.17ex\hbox{$\scriptstyle\sim$}} 
\def\Psik{\Psi^{(k)}}
\newcommand{\silviacorr}[1]{\textcolor{black}{#1}}
\DeclareGraphicsExtensions{.pdf,.eps}

\begin{document}

\title[The VMILAn method]{On the convergence of a linesearch based proximal-gradient method for nonconvex optimization}

\author{S Bonettini$^1$, I Loris$^2$, F Porta$^1$, M Prato$^3$ and S Rebegoldi$^3$}
\address{$^1$ Dipartimento di Matematica e Informatica, Universit\`{a} degli Studi di Ferrara, Via Saragat 1, 44121 Ferrara, Italy}
\address{$^2$ D\'epartement de Math\'ematique, Universit\'e libre de Bruxelles, Boulevard du Triomphe, 1050 Bruxelles, Belgium}
\address{$^3$ Dipartimento di Scienze Fisiche, Informatiche e Matematiche, Universit\`{a} degli Studi di Modena e Reggio Emilia, Via Campi 213/b, 41125 Modena, Italy}
\eads{\mailto{silvia.bonettini@unife.it}}

\begin{abstract}
We consider a variable metric linesearch based proximal gradient method for the minimization of the sum of a smooth, possibly nonconvex function plus a convex, possibly nonsmooth term. We prove convergence of this iterative algorithm to a critical point if the objective function satisfies the Kurdyka-{\L}ojasiewicz property at each point of its domain, under the assumption that a limit point exists. The proposed method is applied to a wide collection of image processing problems and our numerical tests show that our algorithm results to be flexible, robust and competitive when compared to recently proposed approaches able to address the optimization problems arising in the considered applications.
\end{abstract}

%
\vspace{2pc}

\submitto{\IP}
%

\section{Introduction}\label{sec1}

In inverse problems, direct inversion formulas and associated fast reconstruction algorithms (such as filtered backprojection in computed tomography) are available only for a restricted set of problems. In many cases the solution to an inverse problem is reformulated in terms of an optimization problem, in which the objective function includes a distance-like term $f_0(x)$ describing the relation between the unknown object $x$ and the measured data and possibly an additional function $f_1(x)$ aimed at restricting the search of the object to desirable properties specified by $f_1(x)$. The resulting minimization problem has the form
\begin{equation}\label{minprob}
\min_{x\in \R^n} f(x)\equiv f_0(x) + f_1(x).
\end{equation}
The development of efficient numerical optimization algorithms for problems of this type is of great importance for the practical resolution of inverse problems.

Two of the most widely used assumptions in the formulation of inverse problems are the Gaussian nature of the noise on the data, and the linearity of the relationship between measured data and unknown object. Together, they lead to convenient linear least squares terms which, when combined with popular regularizer(s), make problem \eqref{minprob} quadratic or, at least, convex. However, nonlinearity of the relationship between measurements and unknowns or non-Gaussianity of the noise can lead to nonconvex optimization problems. These problems are more difficult to solve than their convex counterparts and algorithms for their numerical solution are less developed.

Examples of nonlinearity can be found in many problems. Blind deconvolution \cite{ayers1988iterative}, where both object and point spread function need to be recovered, is a prime example. More generally, the variational formulation of non-negative matrix factorization \cite{Lee1999} leads to nonconvex optimization problems.

The maximum-likelihood formulation for the \emph{simultaneous} recovery of the activity and the attenuation correction factors in time-of-flight positron emission tomography \cite{6264102} is another example of a nonconvex optimization problem encountered in inverse problems.
In a similar vein, quantitative photoacoustic tomography \cite{0266-5611-29-7-075006} deals with the problem of reconstructing not only the distribution of initial pressure from measurements of propagated acoustic waves, but also seeks to determine chromophore concentration distributions, a nonlinear ill-posed problem.

In global seismic tomography \cite{Nolet2008} scientists typically try to image the seismic wave speed in the Earth's mantle, itself a proxy for temperature, based on Earthquake arrival times and an approximate linear relationship between both. In local seismic tomography scientists use seismic arrays for imaging the Crust and Upper Mantle. In particular, reflection tomography utilizes artificial tremors for the reconstruction of shallow subsurface features and is inherently nonlinear \cite{Tarantola1984}.

Optical flow, i.e. the recovery of motion from images, is also a nonlinear ill-posed problem. Horn and Schunck \cite{Horn1981} studied a variational formulation of the problem while Brox et al. \cite{Brox2004} avoid linearisation in the data term to allow for large displacements and obtain smaller errors than a series of other methods.

Several inverse problems in magnetic resonance imaging and tomography also give rise to nonlinear equations. Few examples are velocity-encoded MRI \cite{Valkonen2014,Benning2014} where the complex phases of images are related to the velocity of the imaged fluid, the treatment of the Stejskal--Tanner equation in diffusion tensor imaging \cite{JMRI:JMRI1076,Wang2004} and the reconstruction problem arising from phase contrast tomography \cite{Repetti_etal_2014}.

Gaussian noise leads to least squares, but often Gaussianity cannot be assumed and therefore least squares is not appropriate. Examples of non-Gaussianity are also abundant in inverse problems. One example is data obtained from photon counting; this naturally leads to Poisson noise and the associated  Kullback-Leibler divergence, which is a convex function. However, other noise models lead to nonconvex problems, even when the relation between noiseless data and unknown object is linear. Examples that will be discussed in the applications section are Cauchy noise and signal dependent noise.

If the data misfit term $f_0$ is differentiable and the penalty term $f_1$ is sufficiently simple, the structure of the objective function in \eqref{minprob} can be exploited by the class of proximal-gradient (or forward-backward) algorithms \cite{Combettes-Vu-2014}, which are described by the generic first order iteration
\begin{equation}\label{iter}
x^{(k+1)} =  x^{(k)} +\lambda_k (y^{(k)} - x^{(k)}),
\end{equation}
where $y^{(k)}$ is given by
\begin{equation}\label{yk}
y^{(k)} = \prox_{\alpha_k f_1}^{D_k} (\xk-\alpha_kD_k^{-1}\nabla f_0(\xk)),
\end{equation}
$\alpha_k$ is a scalar steplength parameter, $D_k$ is a symmetric positive definite matrix and $\prox_{\alpha f_1}^D(\z)$ is the proximal (or resolvent) operator of $f_1$ associated to the parameter $\alpha$ and to the matrix $D$, i.e.
\begin{equation}\label{proj}
\prox_{\alpha f_1}^D(\z) = \arg\min_{\x\in\R^n} f_1(\x) + \frac{1}{2\alpha} (\x-\z)^TD(\x-\z).
\end{equation}
Formally, several recently proposed methods are described by iteration \eqref{iter}--\eqref{proj} (see for example \cite{Attouch-etal-2013,Bonettini-Loris-Porta-Prato-2015,Chouzenoux-etal-2014,Combettes-Vu-2014,Frankel-etal-2015}), even if the meaning of the parameters $\lamk$, $\ak$, $D_k$ substantially differs from one method to the other. In this paper we further develop the approach proposed in \cite{Bonettini-Loris-Porta-Prato-2015}, called VMILA (Variable Metric Inexact Line--search Algorithm), where $\lamk$ is determined by a backtracking loop in order to satisfy an Armijo--type inequality, while $\ak$ and $D_k$ should be considered as ``free'' parameters which can be tuned for improving the algorithmic performances.

The main strength of VMILA consists in allowing the use of well performing, adaptive strategies to choose $\ak, D_k$, originally proposed in the context of smooth optimization, for improving the practical convergence behaviour. When $f_1$ reduces to the indicator function of a convex set, VMILA reduces to the scaled gradient projection method (SGP) which has been used in the last years to address effectively several image reconstruction problems in astronomy and microscopy \cite{Benvenuto-etal-2010,Bonettini-Prato-2010a,Bonettini-etal-2013b,Bonettini-Prato-2015a,Loris-etal-2009,Prato-etal-2012,Zanella-etal-2009,Zanella-etal-2013}. An alternative approach for accelerating forward--backward methods consists in adding an extrapolation step; this idea has been first proposed in \cite{Nesterov-2005} and recently developed in \cite{Beck-Teboulle-2009a,Ochs-etal-2014,Villa-etal-2013}.

From the theoretical point of view, the general convergence result \cite{Bonettini-Loris-Porta-Prato-2015} on the VMILA sequence states that all its limit points are stationary for problem \eqref{minprob}, provided that both the steplength $\alpha_k$ and the eigenvalues of $D_k$ are chosen in prefixed positive intervals $[\alpha_{\min},\alpha_{\max}]$ and $[\frac{1}{\mu},\mu]$, respectively. Convergence of the sequence to a minimum point of \eqref{minprob} has been proved for convex objective functions by choosing suitable scaling matrices sequences.

In the following sections we give better insight into the theoretical convergence properties when the objective function in \eqref{minprob} is not necessarily convex but satisfies the Kurdyka--{\L}ojasiewicz (KL) property \cite{Lojasiewicz-1963,Lojasiewicz-1993,Kurdyka-1998}. In particular, we propose a new variant of VMILA and we demonstrate that under the KL assumption, when the gradient of $f_0$ is Lipschitz continuous and there exists a limit point of the iterates sequence, this point is stationary for \eqref{minprob} and the whole sequence converges to it. In our analysis we also address the case when the proximal point \eqref{proj} is computed inexactly. Finally several applications of the algorithm for the solution of imaging inverse problems are presented.

\section{The problem and some preliminaries}\label{sec2}

\subsection{Notations}

We denote the extended real numbers set as $\bR = \R\cup \{-\infty,+\infty\}$ and by $\R_{\geq0}$, $\R_{>0}$ the set of non-negative and positive real numbers, respectively. The norm of a vector $\x\in\R^n$, induced by a symmetric positive definite matrix $D$ is $\|\x\|_D= \sqrt{\x^TD\x}$. Given $\mu\geq 1$, we denote by ${\mathcal M}_\mu$ the set of all symmetric positive definite matrices with all eigenvalues contained in the interval $[\frac 1 \mu,\mu]$. For any $D\in {\mathcal M}_\mu$ we have that $D^{-1}$ belongs to ${\mathcal M}_\mu$ and
\begin{equation}\label{ine_norm}
\frac 1\mu\|\x\|^2 \leq \|\x\|^2_{D}\leq \mu\|x\|^2
\end{equation}
for any $\x\in\R^n$. The indicator function of a convex set $\Omega\subset\R^n$ is defined as
\begin{equation*}
\iota_\Omega(\x) = \left\{\begin{array}{cl} 0&\mbox{ if } \x\in\Omega \\ +\infty & \mbox{ if } \x\not\in\Omega\end{array}\right..
\end{equation*}
For $-\infty < \upsilon_1 < \upsilon_2 \leq +\infty$, we set
$[\upsilon_1 < f < \upsilon_2] = \{z \in \R^n: \upsilon_1 < f(z) < \upsilon_2 \}$. Moreover, we denote with ${\rm{dist}}(z,\Omega)$ the distance between a point $z$ and a set $\Omega \subset \R^n$, i.e. $\dist (\z,\Omega) = \inf_{\x\in \Omega}\|\x-\z\|$.

\subsection{Problem formulation}

In this paper we address the optimization problem
\begin{equation}\label{minprob_statement}
\min_{x\in \R^n} f(x)\equiv f_0(x) + f_1(x)
\end{equation}
under the following assumptions on the involved functions:
\begin{assumptions}\label{ass1}
\silviacorr{
(i) $f_1:\R^n\to\bR$ is proper, convex and lower semicontinuous.\\
(ii) $f_0:\R^n\to\R$ is smooth, i.e. continuously differentiable, on an open set $\Omega_0\supset \dom(f_1)$.\\
(iii) $f_0$ has an $L-$Lipschitz continuous gradient on $\dom(f_1)$ with $L>0$, i.e.
\begin{equation*}
\|\nabla f_0(x)-\nabla f_0(y)\|\leq L\|x-y\|, \ \forall \ x,y\in \dom(f_1).
\end{equation*}
(iv) $f = f_0+f_1$ is bounded from below.}
\end{assumptions}

\subsection{Notions of subdifferential calculus}
\begin{Def}\label{def:nonconvex_subdiff}\cite[Definition 8.3]{Rockafellar-Wets-1998}
Let $\Psi$ be a function from $\R^n$ to $\bR$, and let $\x\in\dom(\Psi)$. The \emph{Fr\'{e}chet subdifferential} of $\Psi$ at $x$ is the set
\begin{equation*}
\hat\partial \Psi(\x) = \left\{\v\in\R^n : \liminf_{\y\to \x,\y\neq\x} \frac 1{\|\x-\y\|}(\Psi(\y)-\Psi(\x)-(\y-\x)^T\v)\geq 0\right\}.
\end{equation*}
Moreover, the \emph{subdifferential} of $\Psi$ at $\x$ is defined as
\begin{eqnarray*}
\partial \Psi(\x) = \{ &\v\in\R^n : \exists \{\yk\}_{k\in\N} \subset \R^n,\vk \in \hat\partial \Psi(\yk) \ \forall k\in \N \ {\rm{such}} \ {\rm{that}} \\
                       &\yk\to\x, \ \Psi(\yk)\to \Psi(\x) \ {\rm{and}} \ \vk\to\v \}.
\end{eqnarray*}
Finally, we define $\dom(\partial \Psi)=\{\x\in\dom(\Psi):\partial \Psi(\x) \neq \emptyset\}$.
\end{Def}
\begin{Rem}\label{rem2}
(i) If $\x\in \R^n$ is a local minimizer of $\Psi$, then $0\in \partial \Psi(\x)$. If $\Psi$ is convex, this condition is also sufficient.\\
(ii) A point $\x\in \R^n$ is \emph{stationary} for $\Psi$ if $\x\in\dom(\Psi)$ and $0\in \partial\Psi(\x)$.
\end{Rem}
\begin{Lemma}\label{fenchelsub}
\cite[Proposition 8.12]{Rockafellar-Wets-1998} Let $\Psi:\R^n\rightarrow\bR$ be a proper, convex function. Then for any $\x\in \dom(\Psi)$
\begin{equation*}
\partial \Psi(\x)=\{\v\in \R^n: \Psi(\y) \geq \Psi(\x) + (\y-\x)^T\v \ \ \forall \y\in \R^n \}=\hat \partial \Psi(\x).
\end{equation*}
\end{Lemma}
\begin{Lemma}\label{sumrule}
\cite[Exercise 8.8(c)]{Rockafellar-Wets-1998} Let $f=f_0+f_1$ be any function satisfying Assumption \ref{ass1}. Then, for any $\x\in \dom(f)$
\begin{equation*}
\partial f(\x) = \{\nabla f_0(\x)\} + \partial f_1(\x).
\end{equation*}
\end{Lemma}
\begin{Def}\label{def:convex_subdiff}\cite[p. 82]{Zalinescu-2002}
Let $\Psi:\R^n\rightarrow\bR$ be a proper, convex function. Given $\epsilon\in\R_{\geq 0}$, the $\epsilon$-subdifferential of $\Psi$ at a point $\x\in\R^n$ is the set
\begin{equation}\label{eps-subdiff}
\partial_\epsilon \Psi(\x) = \{\v\in \R^n: \Psi(\y) \geq \Psi(\x) + (\y-\x)^T\v -\epsilon \ \ \forall \y\in \R^n\}.
\end{equation}
\end{Def}
\begin{Rem}\label{rem3}
(i) If $\x\notin \dom(\Psi)$, then $\partial_\epsilon \Psi(\x)= \emptyset$ for any $\epsilon\in \R_{\geq0}$. Conversely, if $x\in \dom(\Psi)$ and $\Psi$ is lower semicontinuous at $x$, then $\partial_\epsilon \Psi(\x)\neq \emptyset$ for any $\epsilon\in \R_{>0}$ \cite[Theorem 2.4.4]{Zalinescu-2002}.\\
(ii) For $\epsilon = 0$, thanks to Lemma \ref{fenchelsub}, the usual subdifferential set $\partial \Psi(\x)$ is recovered. In this case, it might happen that $\partial \Psi(\x)= \emptyset$ even if $\x\in \dom(\Psi)$ (see \cite[p. 215]{Rockafellar-1970} for a counterexample). However, if $x\in {\rm{int}} \ \dom(\Psi)$, then $\partial \Psi(\x)\neq \emptyset$ \cite[Theorem 2.4.12]{Zalinescu-2002}.\\
\silviacorr{(iii) Let $\Psi$ be lower semicontinuous, $\{\epsilon_k\}_{k\in \N}\subset \R_{\geq 0}$, $\epsilon\in \R_{\geq 0}$, and $\{(\xk,v^{(k)})\}_{k\in \N}$ a sequence such that $(\xk,v^{(k)})\in{\rm{graph}} \ \partial_{\epsilon_k} \Psi=\{(\x,x^*)\in \R^n\times\R^n: x^*\in \partial_{\epsilon_k} \Psi(\x)\}$. If $(\xk,v^{(k)})\rightarrow(\x,v)$ and $\epsilon_k\rightarrow\epsilon$ as $k\rightarrow +\infty$, then $(\x,v)\in {\rm{graph}} \ \partial_{\epsilon} \Psi$ \cite[Theorem 2.4.2(ix)]{Zalinescu-2002}.}\\
\end{Rem}

\subsection{Kurdyka--{\L}ojasiewicz functions}

In this section we define the Kurdyka--{\L}ojasiewicz (KL) property. We adopt the definition employed also in \cite{Attouch-etal-2010,Attouch-etal-2013,Ochs-etal-2014}, but we remark that other versions of this property are studied in the literature \cite{Attouch-Bolte-2009,Bolte-etal-2010b,Chouzenoux-etal-2014}.
\begin{Def}\label{KL}
Let $f:\R^n \longrightarrow \bR$  be a proper, lower semicontinuous function. The function $f$ is said to have the KL property at $\overline{z} \in \dom(\partial f)$ if there exist $\upsilon \in (0,+\infty]$, a neighborhood $U$ of  $\overline{z}$ and a continuous concave function $\phi:[0,\upsilon)\longrightarrow [0,+\infty)$ such that:
\begin{itemize}
\item $\phi(0) = 0$;
\item $\phi$ is $C^1$ on $(0,\upsilon)$;
\item $\phi'(s) > 0$ for all $s \in (0,\upsilon)$;
\item the KL inequality
\begin{equation*}
\phi'(f(z)-f(\overline{z})) {\rm{dist}}(0,\partial f(z)) \geq 1
\end{equation*}
holds for all $z \in U \cap [f(\overline{z}) < f < f(\overline{z}) + \upsilon]$.
\end{itemize}
If $f$ satisfies the KL property at each point of $\dom(\partial f)$, then $f$ is called a KL function.
\end{Def}
Examples of KL functions are the indicator functions of semi-algebraic sets, real polynomials, $p$-norms and, in general,  semi-algebraic functions or real analytic functions \cite{Attouch-etal-2010,Bolte-etal-2014,Xu-Yin-2013}.
As a result of this, a large variety of optimization problems frequently addressed in signal and image processing is included, as those exploiting a $p$-norm or the Kullback-Leibler divergence as fit-to-data term, and box plus equality constraints as feasible set.

\section{Algorithm and convergence analysis}\label{sec3}

\subsection{The proposed algorithm}\label{sec:algo}

The proposed approach, which is outlined in Algorithm \ref{algo:nuSGP}, is denoted with the name VMILAn (where ``n'' stands for ``new version'') and aims at finding a stationary point for problem \eqref{minprob_statement}.\\
Before presenting Algorithm \ref{algo:nuSGP}, we recall the following definitions. Given the point $\xk\in\R^n$, the parameters $\ak\in \R_{>0}$, $\gamma\in[0,1]$ and a symmetric positive definite matrix $D_k$, we define the function
\begin{equation}\label{Hk}
\thk(\x) = \nabla f_0(\xk)^T(\x-\xk) + \frac \gamma{2\alpha_k} \|\x-\xk\|^2_{D_k} + f_1(\x)-f_1(\xk).
\end{equation}
We observe that $\thk$ is strongly convex for any $\gamma\in (0,1]$. Thus, by setting $\hk = \hk_1$, we define the (unique) proximal point
\begin{equation}\label{minhk}
\yk = \prox_{\alpha_k f_1}^{D_k} (\zk) = \arg\min_{\x\in\R^n} \hk(\x),
\end{equation}
where $\zk = \xk-\alpha_kD_k^{-1}\nabla f_0(\xk)$.\\
\begin{algorithm}[H]\caption{Variable metric inexact line--search based algorithm - new version (VMILAn)}\label{algo:nuSGP}
Choose $0<\alpha_{\min}\leq\alpha_{\max}$, $\mu \geq 1$, $\delta, \beta \in(0,1)$, $\gamma \in [0,1]$, $\tau\in \R_{> 0}$, $\x^{(0)}\in\dom(f_1)$. \\
For $k=0,1,2,...$
\begin{algorithmsteps}
\item\label{step1} Choose $\ak\in[\alpha_{\min},\alpha_{\max}]$, $D_k\in {\mathcal M}_{\mu}$.
\item\label{step2} Let $\thk$, $\hk$ and $\yk$ be defined as in \eqref{Hk}-\eqref{minhk}.\\
Compute $\tyk$ such that
\begin{eqnarray}
&\hk(\tyk) - \hk(\yk) \leq - \frac \tau 2\thk(\tyk).\label{inexact}
\end{eqnarray}
\item\label{step3} Set $\dk = \tyk-\xk$.
\item\label{step4} Compute the smallest non-negative integer $i_k$ such that
\begin{equation}\label{Armijo}
f(\xk + \delta^{i_k}\dk) \leq f(\xk) + \beta \delta^{i_k}\thk(\tyk)
\end{equation}
and set $\lamk = \delta^{i_k}$.
\item\label{step5} Compute the new point as
\begin{equation}\label{SGPmodified}
x^{(k+1)}=
\cases{\tyk&if $f(\tyk) < f(x^{(k)}+\lamk d^{(k)})$ \\
x^{(k)}+\lambda_k d^{(k)}&otherwise \\}.
\end{equation}
\end{algorithmsteps}
\end{algorithm}
\silviacorr{Algorithm \ref{algo:nuSGP} is a slightly modified version of VMILA \cite{Bonettini-Loris-Porta-Prato-2015} and belongs to the class of proximal-gradient algorithms. Its main features are described below.\\
\emph{\ref{step1} - Variable metric.}\\ In our approach, the steplength parameter $\ak$ and the scaling matrix $D_k$ should be considered as almost free parameters which can be tuned to better capture the local features of the objective function and constraints, with the aim to accelerate the progress towards the solution. Indeed, in the following convergence analysis we will make the only assumption that they are bounded as required at \ref{step1}. Concerning the practical choice of the steplength parameter $\ak$, general rules have been proposed in the literature, e.g. the Barzilai and Borwein rules \cite{Barzilai-Borwein-1988} and the more recent advancement in this field employing the Ritz values \cite{Fletcher-2012}, and their practical effectiveness is well established. Unlike the steplength selection, choosing an appropriate scaling matrix $D_k$ is strictly related to the problem features, i.e. the specific shape of the objective function to be minimized and the constraints. Some guidelines about this choice can be found in the literature on imaging inverse problems in variational form, for example the Majorize-Minimize principle \cite{Chouzenoux-etal-2014} or the Split Gradient strategy \cite{Lanteri-etal-2002}. More details on these aspects are given in Section \ref{sec_exp1}.\\
\emph{\ref{step2} - Inexact computation of the proximal point}\\
Condition \eqref{inexact} at \ref{step2} expresses an inexact computation of the proximal point which is similar to the one introduced in \cite[Definition 2.1]{Salzo-Villa-2012} and is weaker than condition \cite[Equation 31]{Bonettini-Loris-Porta-Prato-2015} used in VMILA. We observe first that, since $\hk(\xk)=0$, then $\hk(\yk)\leq 0$ and, by further recalling that $\thk(y)\leq\hk(y)$ for all $\y\in\R^n$, it follows from condition \eqref{inexact} that
\begin{equation}\label{hgamma}
\thk(\tyk)\leq 0
\end{equation}
with the equality holding if and only if $\tyk$ is stationary \cite[Proposition 2.3]{Bonettini-Loris-Porta-Prato-2015}. Thus, \ref{step2} is well-posed. Furthermore, condition \eqref{hgamma}, together with the fact that $f_1$ is convex, implies that a point $\tyk$ satisfying \eqref{inexact} belongs to $\dom(f_1)$.\\
The more appealing feature of condition \eqref{inexact} is that a point $\tyk$ satisfying it can be computed in practice, even if $\hk(\yk)$ is not known, in the quite general case when $f_1(\x) = g(A\x)$, where $g:\R^m\to\R$ is a proper, convex, lower semicontinuous function and $A\in\R^{m\times n}$. In this case, the dual problem of \eqref{minhk} is
\begin{equation}\label{duale}
\max_{\v\in \R^m} \Psik(\v)
\end{equation}
with
\begin{equation*}
\fl \Psik(\v) = -\frac{1}{2\ak}\|\ak D_k^{-1}A^T\v-\zk\|^2_{D_k} - g^*(\v)-f_1(\xk)-\frac{\ak}{2} \|\nabla f_0(\xk)\|^2_{D_k^{-1}} + \frac{1}{2\ak}\|\zk\|^2_{D_k}.
\end{equation*}
Condition \eqref{inexact} is fulfilled by any point $\tyk = \zk-\ak D_k^{-1} A^T v$ with $\v\in\R^m$ satisfying
\begin{equation}\label{stopcond}
\hk(\tyk)\leq \eta \Psik(\v)
\end{equation}
where $\eta=1/(1+\tau/2)$. Indeed, if inequality \eqref{stopcond} holds we have
\begin{equation}\label{rev1}
\fl\hk(\tyk)-\hk(\yk)\leq \hk(\tyk)-\Psik(v)\leq -\frac \tau 2 \hk(\tyk) \leq -\frac \tau 2 \thk(\tyk),
\end{equation}
where the leftmost inequality follows from the fact that the dual function satisfies $\Psik(\v)\leq \hk(\y)$ for any $\y\in\R^n$, $\v\in\R^m$, while the last inequality is a consequence of $0\leq \gamma\leq 1$.\\
A point satisfying \eqref{stopcond} can be computed by applying an iterative method to \eqref{duale}, using \eqref{stopcond} as stopping criterion. More in detail, let $\{\v^{(k,\ell)}\}_{\ell\in\N}\subset \dom(\Psi^{(k)})$ be a sequence converging to the solution of the dual problem \eqref{duale} as $\ell\to \infty$, which can be generated by applying for example a forward--backward method to \eqref{duale}. If $\dom(f_1)$ is closed, we can define $\tilde y^{(k,\ell)}= P_{\dom(f_1)}(\zk-\ak D_k^{-1} A^T \v^{(k,\ell)})$, where $P_{\dom(f_1)}(\cdot)$ denotes the Euclidean projection onto $\dom(f_1)$. Then, the requested point is $\tyk = \tilde y^{(k,\ell)}$ where $\ell$ is the smallest integer such that $\hk(\ty^{(k,\ell)})\leq \eta \Psik(\v^{(k,\ell)}) $. Notice that the projection onto $\dom(f_1)$ avoids infinite values on the left hand side of \eqref{stopcond}.\\
We point out that condition \eqref{inexact} is equivalent to $0\in\partial_{\epsilon_k} \hk(\tyk)$, where $\epsilon_k = -\frac \tau 2 \thk(\tyk)$, which is a relaxed version of the inclusion characterizing the exact proximal point, i.e. $0\in\partial \hk(\yk)$. In the next section we will prove that the $\epsilon_k$ defined above are summable and then, in particular, $\lim_{k\to\infty}\epsilon_k = \lim_{k\to\infty}\thk(\tyk) = 0$ (see Lemma \ref{keylemma} to follow). Thus, the approximate computation of the proximal point through inequality \eqref{inexact} becomes automatically more accurate as the iterations proceed.\\
\emph{\ref{step4} - Armijo-like backtracking loop}\\
The steplength (or overrelaxation) parameter $\lamk$ is adaptively computed by means of a backtracking loop at \ref{step4}, which terminates when the Armijo-like condition \eqref{Armijo} is satisfied. The aim of \eqref{Armijo} is to accept only the steplength which produces a sufficient decrease of the objective function and this is crucial for the convergence of the whole method. Setting $\gamma=0$ allows to recover the standard Armijo condition and, indeed, $\gamma$ can be considered as an on/off parameter to include or not the quadratic term $\|\tyk-\xk\|^2_{D_k}$ on the right-hand-side of \eqref{Armijo}; in general, taking $\gamma = 1$ may produce larger steplengths (see Figure \ref{fig:fig0}).\\
Since $\tyk$ satisfying \eqref{inexact} necessarily belongs to the domain of $f_1$, $f_1$ is convex and $\xk\in\dom(f_1)$, any point on the line $\xk+\lambda(\tyk-\xk)$, $\lambda\in[0,1]$ belongs to $\dom(f_1)$. Then, by Assumption \ref{ass1} (ii), $ f(\xk+\lambda \dk) < +\infty$ for all $\lambda\in[0,1]$ and, as a consequence, the two sides of \eqref{Armijo} only involve finite quantities.\\
Inequality \eqref{hgamma} implies that the linesearch procedure at \ref{step4} terminates in a finite number of steps, i.e., for all $k\in \N$ there exists $i_k<\infty$ such that \eqref{Armijo} holds \cite[Proposition 3.1]{Bonettini-Loris-Porta-Prato-2015}.\\
\emph{\ref{step5} - Overrelaxation}\\
We observe that \eqref{Armijo} does not necessarily imply that $f(\xk+\lamk\dk) \leq f(\tyk)$ (see Figure \ref{fig:fig0}). Then, we force this inequality to hold by an extra step, \ref{step5}, which guarantees that $f(\xkk) \leq f(\tyk)$ and $f(\xkk) \leq f(\xk+\lamk\dk)$, where $\lamk$ is computed via the backtracking loop at \ref{step4}. \ref{step5} is the main difference between VMILA \cite{Bonettini-Loris-Porta-Prato-2015} and Algorithm \ref{algo:nuSGP} and it is crucial for proving the convergence of the sequence $\{\xk\}_{k\in \N}$ in Theorem \ref{thm:abstract_conv}. It could also allow, in general, to take a point corresponding to a smaller value of the objective function instead of simply setting $\xkk = \xk+\lamk\dk$.}

\begin{figure}
\begin{center}
  \includegraphics[scale = 0.6]{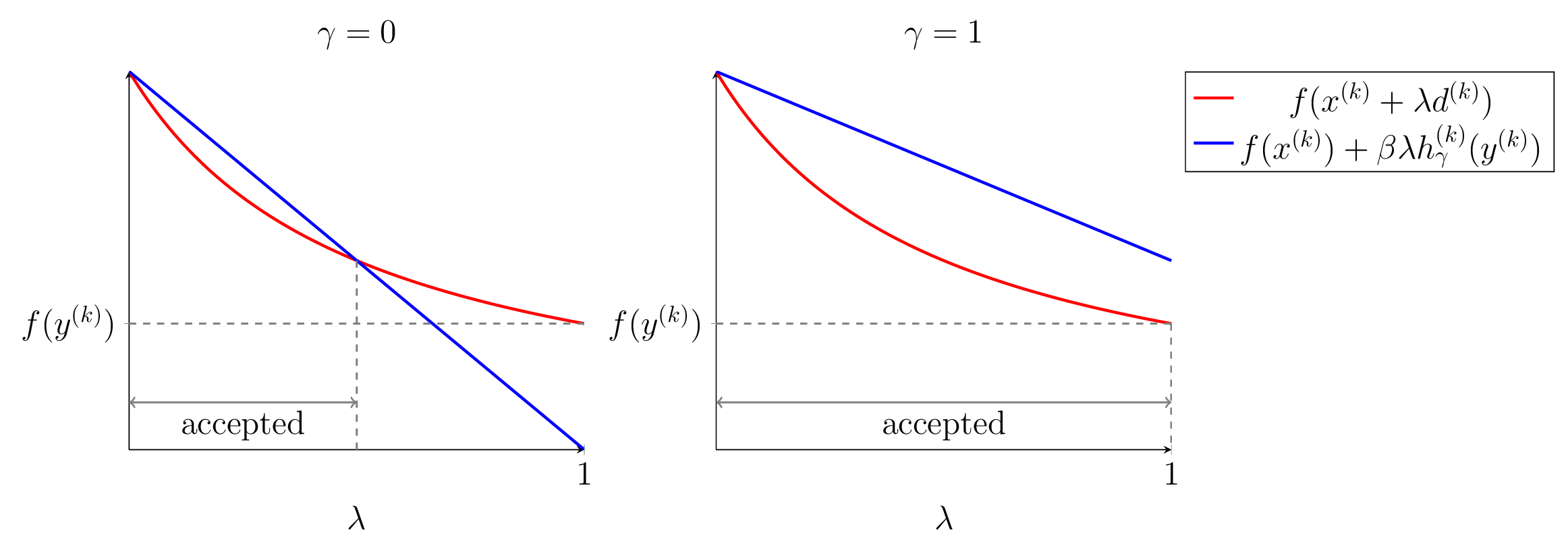}\\
  \caption{Linesearch example: $f_0(x) = \frac 2{x+1}$, $f_1(x) = \iota_{[0,10]}(x)$, $\xk= 0$, $\beta=\frac 1 2$, $\alpha_k = 1$, $D_k=1$. In general, the points satisfying the Armijo condition could not improve the function value at $\yk$.}\label{fig:fig0}
\end{center}
\end{figure}

\subsection{Convergence analysis}
We collect in the following lemma some properties of Algorithm \ref{algo:nuSGP}, which will be fundamental for the subsequent analysis. Here and in the following we denote by $\{\xk\}_\kinN$, $\{\tyk\}_{k\in \N}$ and $\{\lambda_k\}_{k\in \N}$ the sequences generated by Algorithm \ref{algo:nuSGP}.
\begin{Lemma}\label{keylemma}
For all $k\geq 0$, the following relations hold
\begin{eqnarray}
& & \lambda_k\geq\lambda_{\min}>0\label{lambdastar}\\
& &  0\leq-\sum_{k=0}^{\infty} \thk(\tyk) <\infty\label{new_work5}\\
& & f(\xkk) + a\|\xkk-\xk\|^2 \leq f(\xk)\label{HH1}\\
& & f(\xkk)\leq f(\tyk) \leq f(\xk) +\eta_k, \ \ \lim_k\eta_k = 0\label{HH2}
\end{eqnarray}
and there exist $\bar{\epsilon}_k,\hat{\epsilon}_k\in \R_{\geq 0}$, $0\leq \bar{\epsilon}_k+\hat{\epsilon}_k\leq -\frac\tau 2 \thk(\tyk)$, $v^{(k)}\in \{\nabla f_0(\tyk)\}+\partial_{\bar{\epsilon}_k} f_1(\tyk)$ such that
\begin{equation}\label{HH3bis}
\|v^{(k)}\|\leq b\|\xkk-\xk\|+ \zeta_{k+1}, \ \ \lim_{k\rightarrow \infty}\zeta_k = 0
\end{equation}
for some $\lambda_{\min},a,b\in \R_{>0}$, $\eta_k,\zeta_k\in \R_{\geq 0}$, $\zeta_k=\mathcal{O}(\sqrt{\hat{\epsilon}_{k-1}})$.
\end{Lemma}
\begin{proof}
See \ref{appendixA}.
\end{proof}

Properties similar to \eqref{HH1}--\eqref{HH3bis} hold for several proximal gradient methods developed for nonconvex, nonsmooth problems (see e.g. \cite{Attouch-etal-2013,Chouzenoux-etal-2014,Ochs-etal-2014}).\\
Based on these properties, we state the following proposition, which claims a continuity property of the objective function $f$ with respect to the sequence $\{\xk\}_\kinN$ and its limit points (if $f_1$ is continuous in its domain, the conclusion is straightforward).
\begin{Prop}\label{real_lemma5}
Suppose that the sequence $\{\xk\}_\kinN$ admits a limit point $\bar\x$. Then,
\begin{equation}\label{H3}
\lim_{k\to\infty} f(\xk) = f(\bar\x).
\end{equation}
Moreover, $\bar \x$ is stationary.
\end{Prop}
\begin{proof}
See \ref{appendixB}.
\end{proof}

\begin{Rem}\label{rem6}

Thanks to \ref{step5} of Algorithm \ref{algo:nuSGP}, condition
\begin{equation*}
f(x^{(k+1)})\leq f(x^{(k)}+\lambda_k d^{(k)})
\end{equation*}
is satisfied for all $k\in\N$. This inequality, together with \eqref{inexact}, allows to prove the stationarity of any limit point of the sequence generated by VMILAn also by means of Theorem 3.1 in \cite{Bonettini-Loris-Porta-Prato-2015}, under the assumption that ${\rm{dom}}(f_1)$ is a closed set. \silviacorr{Proposition \ref{real_lemma5} is an alternative to Theorem 3.1 in \cite{Bonettini-Loris-Porta-Prato-2015} since it does not require the closedness of ${\rm{dom}}(f_1)$ but, unlike Theorem 3.1 in \cite{Bonettini-Loris-Porta-Prato-2015}, exploits Lipschitz continuity of $\nabla f_0$.}
\end{Rem}
We have now set the basis for our main convergence result, which will be stated in the following. The proof is similar but not identical to Lemma 2.6 in \cite{Attouch-etal-2013} (see also \cite{Frankel-etal-2015}), since \silviacorr{here we have to take into account of the overrelaxation at \ref{step5}}.
\begin{Thm}\label{thm:abstract_conv}
Suppose that $f$ is a KL function and assume that the sequence $\{\xk\}_\kinN$ generated by Algorithm \ref{algo:nuSGP} satisfies the following condition
\begin{equation}\label{HH3}
\exists \ \vk\in \partial f(\tyk) : \|\vk\|\leq b \|\xkk-\xk\| + \zeta_{k+1}, \ \ \sum_{k=1}^{\infty}\zeta_k < \infty,
\end{equation}
for some $b>0$, $\zeta_k\in \R_{\geq 0}$, and admits a limit point $\bar \x$. Then,
\begin{equation}
\sum\limits_{k=0}^{+ \infty}\|x^{(k+1)}-x^{(k)}\| < +\infty
\end{equation}
and, therefore, the whole sequence converges to $\bar\x$, which is stationary for problem \eqref{minprob_statement}.
\end{Thm}
\begin{proof}
The stationarity of the limit points of $\{\xk\}_{k\in \N}$ is ensured by Proposition \ref{real_lemma5}. It remains to show that the sequence has finite length and, thus, converges.
Let $\upsilon$, $\phi$ and $U$ be as in Definition \ref{KL}. These objects exist since the KL inequality holds, in particular, at $\bar{x}$. From Proposition \ref{real_lemma5} we have $\lim_{k\to\infty} f(\xk) = f(\bar\x)$ and, from \eqref{HH2}, it also follows that $\lim_{k\to\infty}f(\tyk)=f(\bar{x})$. Consequently, the following inequality
\begin{equation}\label{condKL}
f(\bar{x})\leq f(\xk)\leq f(\tykm)< f(\bar{x})+\upsilon
\end{equation}
holds for all sufficiently large $k$. Furthermore, let $\rho > 0$ be such that $B(\bar{x},\rho)\subset U$. Then, using the continuity of $\phi$, the fact that $\bar \x$ is a limit point of $\{\xk\}_\kinN$ and $\sum_{k}\zeta_k<\infty$, one can choose $k_0 \in \N$ sufficiently large such that \eqref{condKL} holds for all $k> k_0$ and the following inequalities are satisfied:
\begin{eqnarray*}
\|\bar{x}-x^{(k_0)}\| \leq \frac\rho 4 \ & \ ; \ & \ 3\sqrt{\frac{f(x^{(k_0)})-f(\bar{x})}{a\lambda_{\min}^2}}\leq \frac\rho 4\\
\frac{b}{a}\phi(f(x^{(k_0)})-f(\bar{x}))\leq \frac\rho 4 \ & \ ; \ & \ \frac 1 b\sum_{i=k_0+1}^\infty \zeta_i\leq \frac \rho 4,
\end{eqnarray*}
$a,b$ being the positive constants in inequalities \eqref{HH1} and \eqref{HH3}.
With a little abuse of notation, we will now use $\{x^{(k)}\}_{k\in \N}$ to denote the sequence $\{x^{(k+k_0)}\}_{k\in \N}$ (and $\{\zeta_k\}_{k\in \N}$ instead of $\{\zeta_{k+k_0}\}_{k\in \N}$), so that \eqref{condKL} and the following inequality hold
\begin{equation}\label{technicalcond}
\|\bar{x}-x^{(0)}\| + 3\sqrt{\frac{f(x^{(0)})-f(\bar{x})}{a\lambda_{\min}^2}}
                        +\frac{b}{a}\phi(f(x^{(0)})-f(\bar{x})) +\frac 1 b\sum_{i=1}^\infty \zeta_i \leq \rho,
\end{equation}
for all $k\geq 1$.
Before we proceed with the core of the proof, let us rewrite \eqref{HH1} as
\begin{equation}\label{H1bis}
\|x^{(k+1)}-x^{(k)}\|\leq \sqrt{\frac{f(x^{(k)})-f(x^{(k+1)})}{a}},
\end{equation}
which, by using \ref{step5} of Algorithm \ref{algo:nuSGP} and \eqref{lambdastar}, writes also as
\begin{equation}\label{H1tris}
\|\tyk-x^{(k)}\|\leq \sqrt{\frac{f(x^{(k)})-f(x^{(k+1)})}{a\lambda_{\min}^2}}.
\end{equation}
Fix $k\geq 1$. We show that if $x^{(k)}, \tykm\in B(\bar{x},\rho)$, then
\begin{equation}\label{fundeq}
2\|x^{(k+1)}-x^{(k)}\|\leq \|x^{(k)}-x^{(k-1)}\|+\phi_k + \frac 1 b \zeta_{k},
\end{equation}
where $\phi_k=\frac{b}{a}[\phi(f(x^{(k)})-f(\bar{x}))-\phi(f(x^{(k+1)})-f(\bar{x}))]$. First we observe that, because of \eqref{condKL}, the quantity $\phi(f(x^{(k)})-f(\bar{x}))$ makes sense for all $k\in\N$, and thus $\phi_k$ is well defined.

If $x^{(k+1)}=x^{(k)}$, inequality \eqref{fundeq} holds trivially. Then we assume $x^{(k+1)} \neq x^{(k)}$ which, thanks to \eqref{H1bis}, implies $f(\xk)> f(\xkk) \geq f(\bar\x)$. Hence, from \eqref{HH2} we obtain $f(\bar\x) < f(\xk)\leq f(\tykm)$ which together with \eqref{condKL}, gives
\begin{equation*}
x^{(k)},\tykm\in B(\bar{x},\rho)\cap [ f(\bar{x})<f<f(\bar{x})+\upsilon].
\end{equation*}
Therefore, we can use the KL inequality in both $x^{(k)}$ and $\tykm$.

Combining the KL inequality at $\tykm$ with \eqref{HH3} shows that $v^{(k-1)}\neq 0$ and $b\|x^{(k)}-x^{(k-1)}\|+\zeta_{k}\neq 0$. Since $v^{(k-1)}\in \sub f(\tykm)$, using again the KL inequality with \eqref{HH3} we obtain
\begin{equation}\label{KL+H2}
\phi'(f(\tykm)-f(\bar{x}))\geq \frac{1}{\|v^{(k-1)}\|}\geq \frac{1}{b\|x^{(k)}-x^{(k-1)}\|+\zeta_{k}}.
\end{equation}
Since $\phi$ is concave, its derivative is non increasing, thus $f(\tykm)-f(\bar{x})\geq f(x^{(k)})-f(\bar{x})$ implies
\begin{equation*}
\phi'(f(x^{(k)})-f(\bar{x}))\geq \phi'(f(\tykm)-f(\bar{x})).
\end{equation*}
Applying this fact to inequality \eqref{KL+H2} leads to
\begin{equation}\label{KL+H2bis}
\phi'(f(x^{(k)})-f(\bar{x}))\geq \frac{1}{b\|x^{(k)}-x^{(k-1)}\|+\zeta_{k}}.
\end{equation}
Using the concavity of $\phi$, \eqref{HH1} and \eqref{KL+H2bis}, we obtain
\begin{eqnarray*}
\fl \phi(f(\xk)-f(\bar\x))-\phi(f(\xkk)-f(\bar\x)) &\geq \phi'(f(\xk)-f(\bar\x))(f(\xk)-f(\xkk))\\
&\geq \phi'(f(\xk)-f(\bar\x))a\|\xkk-\xk\|^2\\
&\geq \frac{a\|\xkk-\xk\|^2}{b\|x^{(k)}-x^{(k-1)}\|+\zeta_{k}}.
\end{eqnarray*}
Rearranging terms in the last inequality yields
\begin{equation*}
\|x^{(k+1)}-x^{(k)}\|^2\leq \phi_k\left(\|x^{(k)}-x^{(k-1)}\|+\frac 1 b \zeta_{k}\right),
\end{equation*}
which, by applying the inequality $2\sqrt{uv}\leq u+v$, gives relation \eqref{fundeq}.

We are now going to establish that for $k=1,2,\ldots$
\begin{eqnarray}
& &x^{(k)},\ty^{(k-1)}\in B(\bar{x},\rho),\label{xj}\\
& &\sum\limits_{i=1}^{k}\|x^{(i+1)}-x^{(i)}\|+\|x^{(k+1)}-x^{(k)}\|\leq \|x^{(1)}-x^{(0)}\|+\chi_k+\frac 1 b \sum_{i=1}^k\zeta_i,\label{cauchy}
\end{eqnarray}
where $\chi_k=\frac{b}{a}[\phi(f(x^{(1)})-f(\bar{x}))-\phi(f(x^{(k+1)})-f(\bar{x}))]$.

Let us prove \eqref{xj} and \eqref{cauchy} by induction. Using \eqref{H1bis} with $k=0$ we have
\begin{equation}\label{H10}
\|x^{(1)}-x^{(0)}\|\leq \sqrt{\frac{f(x^{(0)})-f(x^{(1)})}{a}}\leq \sqrt{\frac{f(x^{(0)})-f(\bar{x})}{a}}.
\end{equation}
Combining the above equation with \eqref{technicalcond} and using the triangle inequality, we obtain
\begin{eqnarray*}
\|\bar{x}-x^{(1)}\|&\leq \|\bar{x}-x^{(0)}\|+\|x^{(0)}-x^{(1)}\|\\
                   &\leq \|\bar{x}-x^{(0)}\|+\sqrt{\frac{f(x^{(0)})-f(\bar{x})}{a}} < \rho,
\end{eqnarray*}
namely $x^{(1)}\in B(\bar{x},\rho)$. Using \eqref{H1tris} with $k=0$ and applying the same arguments as before, we also have $\ty^{(0)}\in B(\bar{x},\rho)$. Finally, direct use of \eqref{fundeq} shows that \eqref{cauchy} holds with $k=1$.

By induction, suppose that \eqref{xj} and \eqref{cauchy} hold for some $k=j\geq 1$. First we prove that $x^{(j+1)}\in B(\bar{x},\rho)$. We have
\begin{eqnarray*}
\fl \|\x^{(j+1)}-\bar\x\|&\leq \|\x^{(0)} -\bar\x\|+\|\x^{(0)}-\x^{(1)}\|+\sum_{i=1}^j \|\x^{(i+1)}-\x^{(i)}\|\\
&\leq \|\x^{(0)} -\bar\x\| + 2\|\x^{(0)}-\x^{(1)}\|+\chi_j +\frac{1}{b}\sum_{i=1}^j\zeta_i\\
&\leq \|\x^{(0)} -\bar\x\| + 2\sqrt{\frac{f(\x^{(0)})-f(\bar\x)}{a}}+\frac b a \phi(f(\x^{(0)})-f(\bar\x)) +\frac{1}{b}\sum_{i=1}^j\zeta_i\\
&<\rho,
\end{eqnarray*}
where the first inequality follows from the triangle inequality, the second one from \eqref{cauchy} with $k=j$, the third one from \eqref{H10} and the monotonicity of $\phi$ and the last one from \eqref{technicalcond}.
Similarly, we can prove that $y^{(j)}\in B(\bar{x},\rho)$. Noticing that $f(\bar{x})\leq f(x^{(k+1)})\leq f(x^{(k)})\leq f(x^{(0)})$, \eqref{H1tris} yields
\begin{equation*}
\|\ty^{(j)}-x^{(j)}\|\leq \sqrt{\frac{f(x^{(0)})-f(\bar{x})}{a\lambda_{\min}^2}}.
\end{equation*}
By using the above relation, the triangle inequality, \eqref{cauchy} with $k=j$, the monotonicity of $\phi$ and \eqref{technicalcond}, we have
\begin{eqnarray*}
\fl \|\bar{x}-\ty^{(j)}\| \leq \ \|\bar{x}-x^{(0)}\|+\|x^{(0)}-x^{(1)}\| +\sum\limits_{i=1}^{j}\|x^{(i+1)}-x^{(i)}\|+\|x^{(j+1)}-x^{(j)}\| +\|x^{(j)}-\ty^{(j)}\|\\
\leq \ \|\bar{x}-x^{(0)}\|+2\|x^{(0)}-x^{(1)}\|+\chi_j+\frac{1}{b}\sum_{i=1}^j\zeta_i+\|x^{(j)}-\ty^{(j)}\|\\
\leq \ \|\bar{x}-x^{(0)}\|+3\sqrt{\frac{f(x^{(0)})-f(\bar{x})}{a\lambda_{\min}^2}}+\frac{b}{a}\phi(f(x^{(0)})-f(\bar{x}))+\frac{1}{b}\sum_{i=1}^j \zeta_i\\
\leq \ \rho,
\end{eqnarray*}
or equivalently $\ty^{(j)}\in B(\bar{x},\rho)$. Now we observe that \eqref{fundeq} with $k=j+1$ writes as
\begin{equation*}
2\|x^{(j+2)}-x^{(j+1)}\|\leq \|x^{(j+1)}-x^{(j)}\|+\phi_{j+1}+ \frac 1 b \zeta_{j+1}.
\end{equation*}
Adding the above inequality with \eqref{cauchy} (with $k=j$) yields \eqref{cauchy} with $k=j+1$, which completes the induction proof.

By directly using \eqref{cauchy}, we get
\begin{equation*}
\sum\limits_{i=1}^{k}\|x^{(i+1)}-x^{(i)}\|\leq \|x^{(1)}-x^{(0)}\|+\frac{b}{a}\phi(f(x^{(1)})-f(\bar{x}))+\frac{1}{b}\sum_{i=1}^{k}\zeta_i
\end{equation*}
and (on account of \eqref{HH3}) therefore
\begin{equation*}
\sum\limits_{i=1}^{+\infty}\|x^{(i+1)}-x^{(i)}\|<+\infty,
\end{equation*}
which implies that the sequence $\{x^{(k)}\}_{k\in \N}$ converges to some $x^{*}$. Considering that $\bar{x}$ is a limit point of the sequence, it must be $x^{*}=\bar{x}$.
\end{proof}

\silviacorr{When $\tyk = \yk=\prox_{\ak f_1}^{D_k} (\zk)$, we have $\thk(\tyk)=0$ and, thanks to Lemma \ref{keylemma}, \eqref{HH3} is automatically guaranteed with $\zeta_k\equiv 0$.} When this choice is made, Algorithm \ref{algo:nuSGP} becomes an exact proximal--gradient method, whose convergence properties are stated in the following corollary, which is a direct consequence of Lemma \ref{keylemma} and Theorem \ref{thm:abstract_conv}.
\begin{Cor}\label{cor:exactSGPconverge}
Suppose that $f$ is a KL function. Let $\{\xk\}_\kinN$ and $\{\lamk\}_\kinN$ be the sequences generated by Algorithm \ref{algo:nuSGP} with $\tyk=\yk$ for all $k\geq 0$. If there exists a limit point $\bar \x$ of $\{\xk\}_\kinN$, then
\begin{itemize}
\item[(i)] $\lim\limits_{k\to\infty} f(\xk) = f(\bar\x)$;
\item[(ii)] $\bar \x$ is a stationary point for problem \eqref{minprob_statement};
\item[(iii)] the sequence $\{\xk\}_\kinN$ converges to $\bar{x}$ and has finite length.
\end{itemize}
\end{Cor}

\subsection{Convergence rate analysis}\label{sec:rate}

We now investigate the convergence rate of \silviacorr{Algorithm \ref{algo:nuSGP}}. In particular, we follow the same outline given in \cite{Frankel-etal-2015}, in which three convergence results are proved for a similar abstract descent method when the function $\phi$ in Definition \ref{KL} is of the form $\phi(t)=\frac{C}{\theta}t^{\theta}$, with $C>0$ and $\theta\in (0,1]$. For instance, this assumption holds for continuous subanalytic functions on a closed domain \cite{Bolte-etal-2007a}, real analytic functions, semialgebraic functions and the sum of a real analytic function and a semialgebraic function (see \cite{Xu-Yin-2013} and references therein). Unlike in \cite{Frankel-etal-2015}, we do not restrict to the case where $\zeta_k\equiv 0$, but we only require that the convergence of the sequence $\{\zeta_k\}_{k\in \N}$ is controlled by the quantity $\thk(\tyk)$.\\
The following theorem expresses the distance of the sequence $\{x^{(k)}\}_{k\in \N}$ to the limit in terms of the function gap and is an adaption of \cite[Theorem 3]{Frankel-etal-2015}.
\begin{Thm}\label{thm:conv_rate_1}
Suppose that $f$ is a KL function and that the sequence $\{x^{(k)}\}_{k\in \N}$ satisfies \eqref{HH3} with
\begin{equation}\label{o_epsilonk}
\zeta_k = {\mathcal O}(\thk(\tyk)).
\end{equation}
Assume in addition that $\{\xk\}$ admits a limit point $\bar{x}$. Let $\phi$ be as in Definition \ref{KL} for the point $\bar{x}$ and set $\bar{\phi}(t)=\max \{\phi(t),\sqrt{t}\}$.
Then, there exists $M\in \R_{>0}$ such that
\begin{equation}
\|x^{(k)}-\bar{x}\|\leq\left(\frac{1}{\sqrt{a}}+\frac{M}{b}+\frac{b}{a}\right)\left(\bar{\phi}(f(x^{(k-1)})-f(\bar{x}))\right).
\end{equation}
\end{Thm}
\begin{proof}
By combining \eqref{Armijo}, \eqref{inexact} and \eqref{lambdastar}, one can show that
\begin{eqnarray*}
 -\frac{\tau}{2}h_{\gamma}^{(k-1)}(\tilde{y}^{(k-1)}) &\leq \frac{\tau}{2}\left(\frac{f(x^{(k-1)})-f(x^{(k)})}{\beta \lambda_{k-1}}\right)\\
&\leq \frac{\tau}{2\beta \lambda_{\min}}\left(f(x^{(k-1)})-f(x^{(k)})\right).
\end{eqnarray*}
From \eqref{o_epsilonk} and the above inequality, there exists $M\in \R_{>0}$ such that
\begin{equation}\label{asymp_eps}
\zeta_k\leq M\left(f(x^{(k-1)})-f(x^{(k)})\right),
\end{equation}
for all $k\in \N$.\\
Let $s^{(k)}:=f(x^{(k)})-f(\bar{x})\geq 0$. If there exists $k\in \N$ such that $s^{(k)}=0$, then the algorithm terminates in a finite number of steps. Then we assume that $s^{(k)}>0$ for all $k\in \N$. As previously shown in the proof of Theorem \ref{thm:abstract_conv}, there exists $k_0\in \N$ such that \eqref{fundeq} holds for all $k\geq k_0$. Summing \eqref{fundeq} for $k=k_0,\ldots,N$, we get
\begin{equation}\label{ineq_eps}
\sum_{k=k_0}^{N}\|x^{(k+1)}-x^{(k)}\|\leq \|x^{(k_0)}-x^{(k_0-1)}\|+\frac{b}{a}\phi(s^{(k_0)})+\frac{1}{b}\sum_{k=k_0}^{N}\zeta_k.
\end{equation}
By using \eqref{asymp_eps}, summing it for $k=k_0,\ldots,N$ and observing that $f(x^{(N)})\geq f(\bar{x})$, \eqref{ineq_eps} yields the following inequality
\begin{equation}\label{ineq_eps_2}
\sum_{k=k_0}^{N}\|x^{(k+1)}-x^{(k)}\|\leq \|x^{(k_0)}-x^{(k_0-1)}\|+\frac{b}{a}\phi(s^{(k_0)})+\frac{M}{b} s^{(k_0-1)}.
\end{equation}
Applying the triangle inequality and passing to the limit, we obtain
\begin{eqnarray*}
\fl \|x^{(k_0)}-\bar{x}\|\leq \sum_{k=k_0}^{\infty}\|x^{(k+1)}-x^{(k)}\|\leq \|x^{(k_0)}-x^{(k_0-1)}\|+\frac{b}{a}\phi(s^{(k_0)})+\frac{M}{b} s^{(k_0-1)}\\
\leq \frac{1}{\sqrt{a}}\sqrt{f(x^{(k_0-1)})-f(x^{(k_0)})}+\frac{b}{a}\phi(s^{(k_0)})+\frac{M}{b} s^{(k_0-1)},
\end{eqnarray*}
where the last inequality follows from \eqref{H1bis}. Finally, recalling that $f(x^{(k_0)})\geq f(\bar{x})$, $\phi$ is an increasing function and $\{s^{(k)}\}_{k\in \N}$ is nonincreasing, we can write
\begin{equation}
\|x^{(k_0)}-\bar{x}\|\leq \frac{1}{\sqrt{a}}\sqrt{s^{(k_0-1)}}+\frac{b}{a}\phi(s^{(k_0-1)})+\frac{M}{b} s^{(k_{0}-1)}.
\end{equation}
Since $s^{(k_0-1)}\leq \sqrt{s^{(k_0-1)}}$ for a sufficiently large $k_0\in \N$, we conclude that $\|x^{(k_0)}-\bar{x}\|\leq \left(\frac{1}{\sqrt{a}}+\frac{M}{b}+\frac{b}{a}\right)\bar{\phi}(s^{(k_0-1)})$.
\end{proof}
The next result directly follows from the previous theorem and provides explicit rates of convergence, for both the function values and the iterates.
\begin{Thm}\label{thm:conv_rate_2}
Suppose that $f$ satisfies the KL property in $\bar{x}$ (a limit point of $\{x^{(k)}\}_{k\in \N}$) with $\phi(t)=\frac{C}{\theta}t^{\theta}$, where $C>0$ and $\theta\in (0,1]$, and that conditions \eqref{HH3} and \eqref{o_epsilonk} hold.
\begin{itemize}
\item[(i)] If $\theta=1$, then $\{x^{(k)}\}_{k\in \N}$ converges in a finite number of steps.
\item[(ii)] If $\theta\in [\frac{1}{2},1)$, then there exist $d>0$ and $\bar{k}\in \N$ such that
\begin{enumerate}
\item $f(x^{(k)})-f(\bar{x})=\mathcal{O}\left(e^{-d\left(k-\bar{k}\right)}\right)$
\item $\|x^{(k)}-\bar{x}\|=\mathcal{O}\left(e^{-\frac{d}{2}\left(k-\bar{k}+1\right)}\right)$.
\end{enumerate}
\item[(iii)] If $\theta\in (0,\frac{1}{2})$, then there exists $\bar{k}\in \N$ such that
\begin{enumerate}
\item $f(x^{(k)})-f(\bar{x})=\mathcal{O}\left(\left(k-\bar{k}\right)^{-\frac{1}{1-2\theta}}\right)$
\item $\|x^{(k)}-\bar{x}\|=\mathcal{O}\left(\left(k-\bar{k}+1\right)^{-\frac{\theta}{1-2\theta}}\right)$.
\end{enumerate}
\end{itemize}
\end{Thm}
\begin{proof}
First we can assume that $s^{(k)}=f(x^{(k)})-f(\bar{x})>0$ for all $k\in \N$, since otherwise the algorithm would terminate in a finite number of steps. \\
Let $U$ be as in Definition \ref{KL} for the point $\bar{x}$. From Theorem 1 we know that $\{x^{(k)}\}_{k\in \N}$ converges to $\bar{x}$ and, because of \eqref{ykconv0}, also $\{\tyk\}_{k\in \N}$ does. Therefore there exists $\bar{k}\in \N$ such that
\begin{equation*}
x^{(k+1)},\tyk \in U\cap [f(\bar{x})<f<f(\bar{x})+v]
\end{equation*}
for all $k\geq \bar{k}$, thus allowing to apply the KL inequality in $\tyk$.\\
Let us take the squares of both sides of condition \eqref{HH3}, divide and multiply them by $b^2$ and $a$ respectively, thus obtaining
\begin{equation*}
\frac{a}{b^2}\|v^{(k)}\|^2\leq a\|x^{(k+1)}-x^{(k)}\|^2+\frac{a}{b^2}\zeta_{k+1}^2+\frac{2a}{b}\zeta_{k+1}\|x^{(k+1)}-x^{(k)}\|.
\end{equation*}
By applying condition \eqref{HH1} to the previous inequality, we get the following relation
\begin{equation*}
\frac{a}{b^2}\|v^{(k)}\|^2\leq (s^{(k)}-s^{(k+1)})+\frac{a}{b^2}\zeta_{k+1}^2+\frac{2\sqrt{a}}{b}\zeta_{k+1}\sqrt{s^{(k)}-s^{(k+1)}}.
\end{equation*}
Since $\lim_{k\to \infty}\zeta_k=0$, it is possible to choose $\bar{k}\in \N$ such that $\zeta_{k+1}^2\leq \zeta_{k+1} \leq \sqrt{\zeta_{k+1}}$ holds for all $k\geq \bar{k}$. Recalling that thanks to \eqref{o_epsilonk} there exists $M>0$ such that $\zeta_{k+1}\leq M (s^{(k)}-s^{(k+1)})$ (see \eqref{asymp_eps}), we obtain
\begin{equation*}
\frac{a}{b^2}\|v^{(k)}\|^2\leq m(s^{(k)}-s^{(k+1)}),
\end{equation*}
where $m=1+\frac{a}{b^2}M+\frac{2\sqrt{aM}}{b}$.\\
Set $t^{(k)}=f(\tyk)-f(\bar{x})$. Then, by multiplying each side of the inequality by $\phi'(t^{(k)})^2$, we have
\begin{equation*}
\fl m\phi'(s^{(k+1)})^2(s^{(k)}-s^{(k+1)}) \geq m \phi'(t^{(k)})^2(s^{(k)}-s^{(k+1)}) \geq \frac{a}{b^2}\phi'(t^{(k)})^2\|v^{(k)}\|^2 \geq \frac{a}{b^2},
\end{equation*}
where the extreme left inequality has been derived using condition \eqref{HH2}, whereas the extreme right one has been obtained by applying the KL inequality in $\tyk$. Therefore, we have come to the following relation
\begin{equation}\label{eq_rate}
\phi'(s^{(k+1)})^2(s^{(k)}-s^{(k+1)})\geq \frac{a}{mb^2}.
\end{equation}
Equation \eqref{eq_rate} is identical to \cite[Theorem 3.4, Equation 6]{Frankel-etal-2015}, from which {\it (i)}, the rates on the function values in part 1 of {\it (ii)} and in part 1 of {\it (iii)} follow immediately, whereas the rates on the iterates contained in part 2 of {\it (ii)} and part 2 of {\it (iii)} are obtained by combining the rates on the function values and Theorem \ref{thm:conv_rate_1}.
\end{proof}

\silviacorr{Since choosing $\tyk = \yk$ at \ref{step2} implies that \eqref{HH3} is satisfied with $\zeta_k \equiv 0$, the convergence rates of Theorem \ref{thm:conv_rate_1} and \ref{thm:conv_rate_2} hold for the exact version of Algorithm \ref{algo:nuSGP}}.

\section{Numerical experience}\label{sec4}

In order to confirm the efficiency of the suggested algorithm we carry out different numerical experiments on realistic optimization problems arising from imaging applications. We compare the obtained results with those provided by some recent methods already applied in such a framework. \silviacorr{All the numerical results in the following sections have been obtained on a PC equipped with an INTEL Core i7 processor 2.70GHz with 8GB of RAM running Matlab ver 7 R2010b.}

\subsection{Image deconvolution in presence of signal dependent Gaussian noise}\label{sec_exp1}

In this section we consider the image restoration problem described in \cite{Chouzenoux-etal-2014}, where the observed data $\g\in\R^n$ are assumed to be acquired according to the model
\begin{equation*}
g_i = (H\x_{\rm{true}})_i+\sigma_i((H\x_{\rm{true}})_i)w_i,
\end{equation*}
where $\x_{\rm{true}} \in \R^n$ denotes the original image to be reconstructed, $H\in\R^{n\times n}$ is a matrix with non-negative entries representing the acquisition system, $\w = (w_1,\cdots,w_n)^T$ is a realization of Gaussian random vector with zero mean and covariance matrix $I_n$ and $\sigma_i:\R\to\R_{>0}$ is defined as
\begin{equation*}
\sigma_i(u) = \sqrt{a_iu+b_i},
\end{equation*}
with $a_i\in\R_{\geq 0}$, $b_i\in\R_{>0}$, for all $i=1,...,n$.

Following the Bayesian paradigm \cite{Geman-Geman-1984}, an estimate of the true image $\x_{\rm{true}}$ can be computed by solving the minimization problem \eqref{minprob} where $f_0$ is a data discrepancy function corresponding to the negative log--likelihood of the data, and $f_1$ is a regularization term chosen to induce some desired properties on the computed solution.

In this case, the negative log-likelihood function is given by
\begin{equation}\label{f0_CPR}
f_0(\x) = \frac 1 2\sum_{i=1}^n \frac{((H\x)_i-g_i)^2}{a_i(H\x)_i + b_i} + \log(a_i(H\x)_i+b_i),
\end{equation}
which is nonconvex and smooth in $\dom(f_0)=\{\x\in\R^n:a_i(H\x)_i + b_i>0 \ \forall i=1,...,n\}$.

If one wants to preserve the edges in the reconstruction and also the non-negativity of the pixel values, the regularization term can be chosen as the sum of the total variation functional \cite{Rudin-Osher-Fatemi-1992} and the indicator function of the set $\R^n_{\geq 0}$, i.e.
\begin{equation}\label{eq:TV+Ind}
f_1(x) = \rho \sum_{i=1}^n \|\nabla_i\x\| + \iota_{\R_{\geq 0}^n}(x),
\end{equation}
where $\rho\in\R_{>0}$ is a regularization parameter and $\nabla_i\in \R^{2\times n}$ represents the discrete gradient of the two dimensional object $\x$ at pixel $i$.

Since $b_i> 0$ for all $i=1,\ldots,n$ and $H$ has non--negative entries, we have $\dom(f_0)\supset \dom(f_1)$ and, in addition, $\nabla f_0$ is Lipschitz continuous in $\dom(f_1)$. Moreover, the graph of $f$ lies in the o-minimal structure containing the graph of the exponential function \cite{Bolte-etal-2007b}, thus $f$ is a KL function.

In order to validate the effectiveness of the proposed method, we consider the test problem ``jetplane'', which can be downloaded from \cite{Repetti2013} (see figure \ref{fig:fig1}). Here, the operator $H$ corresponds to a convolution  with a truncated Gaussian function of size $7\times 7$, $a_i=b_i = 1$ for all $i=1,...,n$ and $\rho = 0.03$.

Since the proximal operator of $f_1$ is not available in a closed form, it has to be approximated via an iterative solution. We observe that the nonsmooth regularization term has the form $f_1(\x) = \phi(A\x)$ where $A^T = (\nabla_1^T,\ldots,\nabla_n^T,I_n)\in \R^{n\times 3n}$ and $\phi:\R^{3n}\to\bar\R$ is defined as
\begin{equation*}
\phi(t) = \sum_{i=1}^n\left\|\left(
\begin{array}{c}
t_{2i-1} \\
t_{2i}
\end{array}
\right)\right\| + \iota_{\R_{\geq 0}^n}\left(
\begin{array}{c}
t_{2n+1} \\
\vdots \\
t_{3n}
\end{array}\right).
\end{equation*}
We implement an inexact version of Algorithm \ref{algo:nuSGP}, where the approximate proximal point $\tyk$ satisfying \eqref{inexact} is computed as described in Section \ref{sec:algo}; in particular, as inner solver for the dual problem \eqref{duale} we adopt the algorithm proposed in \cite{Chambolle-Dossal-2014}. We remark that, if \eqref{HH3} is not ensured, we could not invoke Theorem \ref{thm:abstract_conv} to guarantee the convergence of the whole sequence. However, the stationarity of the limit points is guaranteed by Proposition \ref{real_lemma5}, which holds independently of \eqref{HH3}.

We implement Algorithm \ref{algo:nuSGP} in Matlab environment, setting $\alpha_{\min}=10^{-5}$, $\alpha_{\max}=10^2$, $\delta = 0.5$, $\beta = 10^{-4}$, $\gamma = 1$, $\tau = 10^6-1$, which are the same parameter choices used for VMILA in \cite{Bonettini-Loris-Porta-Prato-2015}. Moreover, as in \cite{Bonettini-Loris-Porta-Prato-2015} we choose the variant of the FISTA algorithm \cite{Beck-Teboulle-2009a} proposed in \cite{Chambolle-Dossal-2014} with $a = 2.1$ as inner solver.

As concerns the choice of the metric, we consider three different choices for $D_k$, all leading to a diagonal matrix whose entries are defined as follows:
\begin{itemize}
\item[MM] $(D_k)_{ii}^{-1} = \max\{\min\{(A_k)_{ii},\mu\},\frac{1}{\mu}\}$, where $A_k$ is defined in \cite[formula (36)]{Chouzenoux-etal-2014} with $\varepsilon = 0$. This matrix $A_k$ is introduced in \cite{Chouzenoux-etal-2014}, following the Majorization-Minimization (MM) approach, where the authors show that the quadratic function $Q(\x,\xk) = f_0(\xk) + \nabla f_0(\xk)^T(\x-\xk) + \frac 1 2 \|\x-\xk\|^2_{A_k}$ is a majorant function for $f_0$, i.e. $f_0(\x)\leq Q(\x,\xk)$ for all $\x\in\dom(f_1)$.
\item[SG] $(D_k)_{ii}^{-1} = \max\{\min\{\frac{\xk_i}{V_i(\xk)+\epsilon},\mu\},\frac{1}{\mu}\}$, where $\epsilon$ is set to the machine precision and $V(\xk)$ is defined as $V(\xk) = H^T s\k$ with
\begin{equation}\nonumber
 s\k_i = (H\x)_i \frac{a_i ( (H\x)_i + \g_i) + 2b_i}{2 (a_i(H\x)_i + b_i)^2} + \frac{a_i}{2 (a_i(H\x)_i + b_i)}.
\end{equation}
This choice can be explained in the context of the split gradient (SG) methods \cite{Lanteri-etal-2002}, which are scaled gradient methods based on a splitting of the gradient $\nabla f_0(\xk)$ in the difference between a positive part $V(\xk)$ and a non-negative one $U(\xk)$.
\item[I] $D_k = I_n$.
\end{itemize}
The parameter $\mu$ bounding the diagonal entries of $D_k$ is set to $10^{10}$. Once computed the matrix $D_k$, the stepsize parameter $\alpha_k$ is chosen using a recent strategy proposed in \cite{Porta-Prato-Zanni-2015} and based on the approximation of the eigenvalues of the Hessian matrix of the objective function by means of a Lanczos--like process (see also \cite{Fletcher-2012} for more details in the unconstrained case). In our problem, for a fixed positive integer $m$ (in our experiments we consider $m=3$), one has to:
\begin{itemize}
\item[a)] Define the matrices
\begin{equation*}
\widetilde{G} = \left[D_{k-m}^{1/2}\widetilde{\ve{g}}^{(k-m)}, \ldots, D_{k-1}^{1/2}\widetilde{\ve{g}}^{(k-1)}\right] \ , \
\Gamma=
\left[
\begin{array}{ccc}
\alpha^{-1}_{k-m}			&	&\\
-\alpha^{-1}_{k-m}	&\ddots		&\\
	&\ddots	&\alpha^{-1}_{k-1}\\
	&	&-\alpha^{-1}_{k-1}\\
\end{array}
\right],
\end{equation*}
by collecting $m$ consecutive steplengths and reduced gradients
\begin{equation}\label{redgrad}
\widetilde{g}^{(k)}_j = \cases{ 0&if $x^{(k)}_j=0$, \\ \left[\nabla f_0(\ve{x}^{(k)})\right]_j&if $x^{(k)}_j > 0$}.
\end{equation}
\item[b)] Compute the Cholesky factorization $R^TR$ of the $m \times m$ matrix $\widetilde{G}^T\widetilde{G}$, the solution $\ve{r}$ of the linear system $R^T\ve{r}=\widetilde{G}^TD_k^{1/2}\widetilde{\ve{g}}^{(k)}$ and the $m \times m$ matrix $\Phi = [R \quad \ve{r}]\Gamma R^{-1}$.
\item[c)] Compute the eigenvalues of the symmetric and tridiagonal approximation $\widetilde{\Phi}$ of $\Phi$ defined as
\begin{equation*}
\widetilde{\Phi} = {\rm{diag}}(\Phi) + {\rm{tril}}(\Phi,-1) + {\rm{tril}}(\Phi,-1)^T,
\end{equation*}
being ${\rm{diag}}(\cdot)$ and ${\rm{tril}}(\cdot,-1)$ the diagonal and the strictly lower triangular parts of a matrix, and use the reciprocal of the positive eigenvalues obtained as steplengths for the next iterations.
\end{itemize}
We compare the performances of our method with the variable metric forward backward (VMFB) algorithm \cite{Chouzenoux-etal-2014}, in the implementation provided by the authors which can be downloaded from \cite{Repetti2013}. We observed that both methods achieve the same value of the objective function in the limit, denoted by $f^*$, which is in general not guaranteed for nonconvex problems. Thus in this case we can compare the optimization properties of the algorithms by measuring the progress toward this value, which has been numerically approximated first by running 5000 iterations of all methods and retaining the smallest value.

\def\subFigScale{0.3}
\begin{figure}
\begin{center}
\begin{tabular}{ccc}
\includegraphics[scale=\subFigScale]{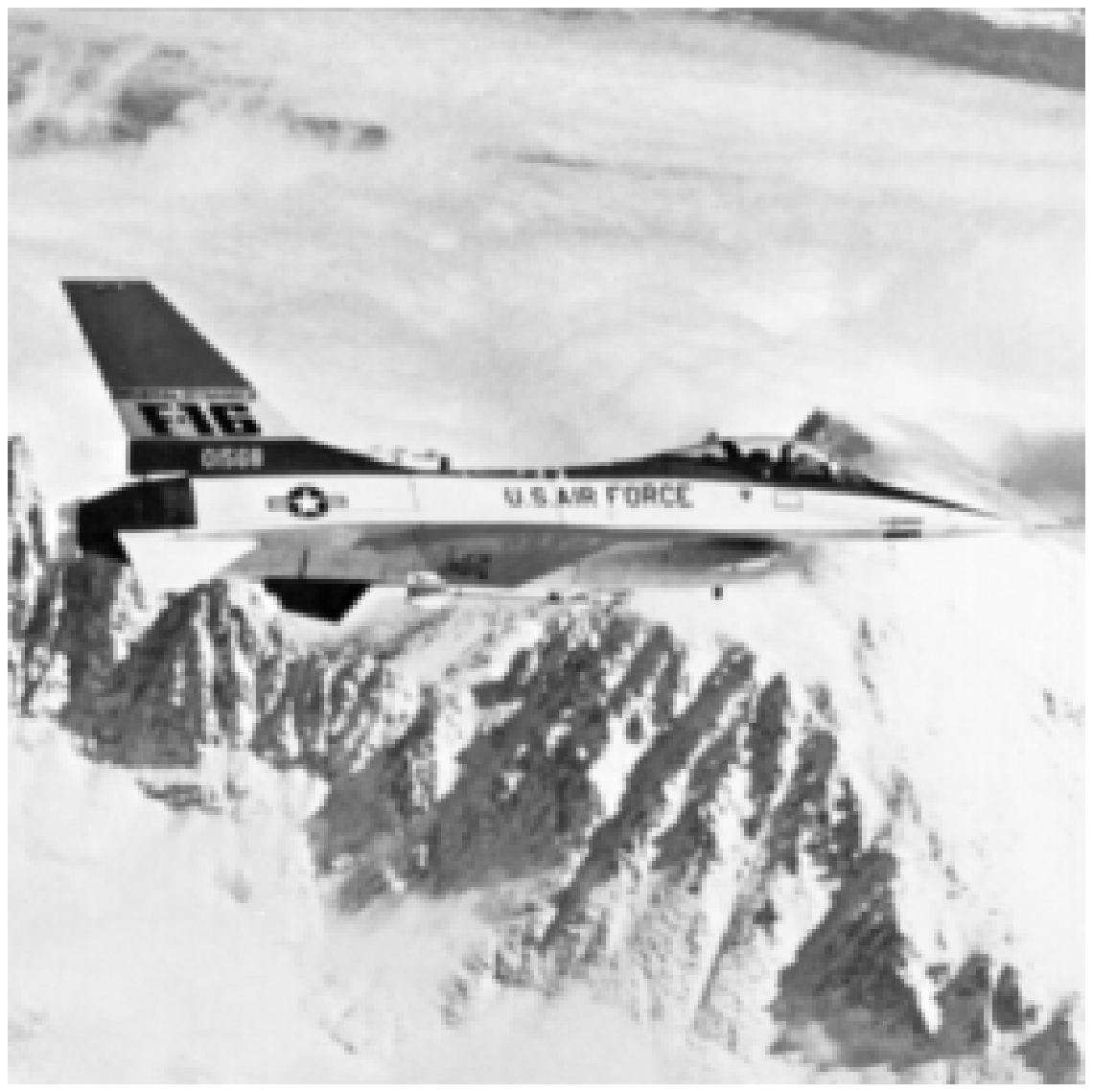}&\includegraphics[scale=\subFigScale]{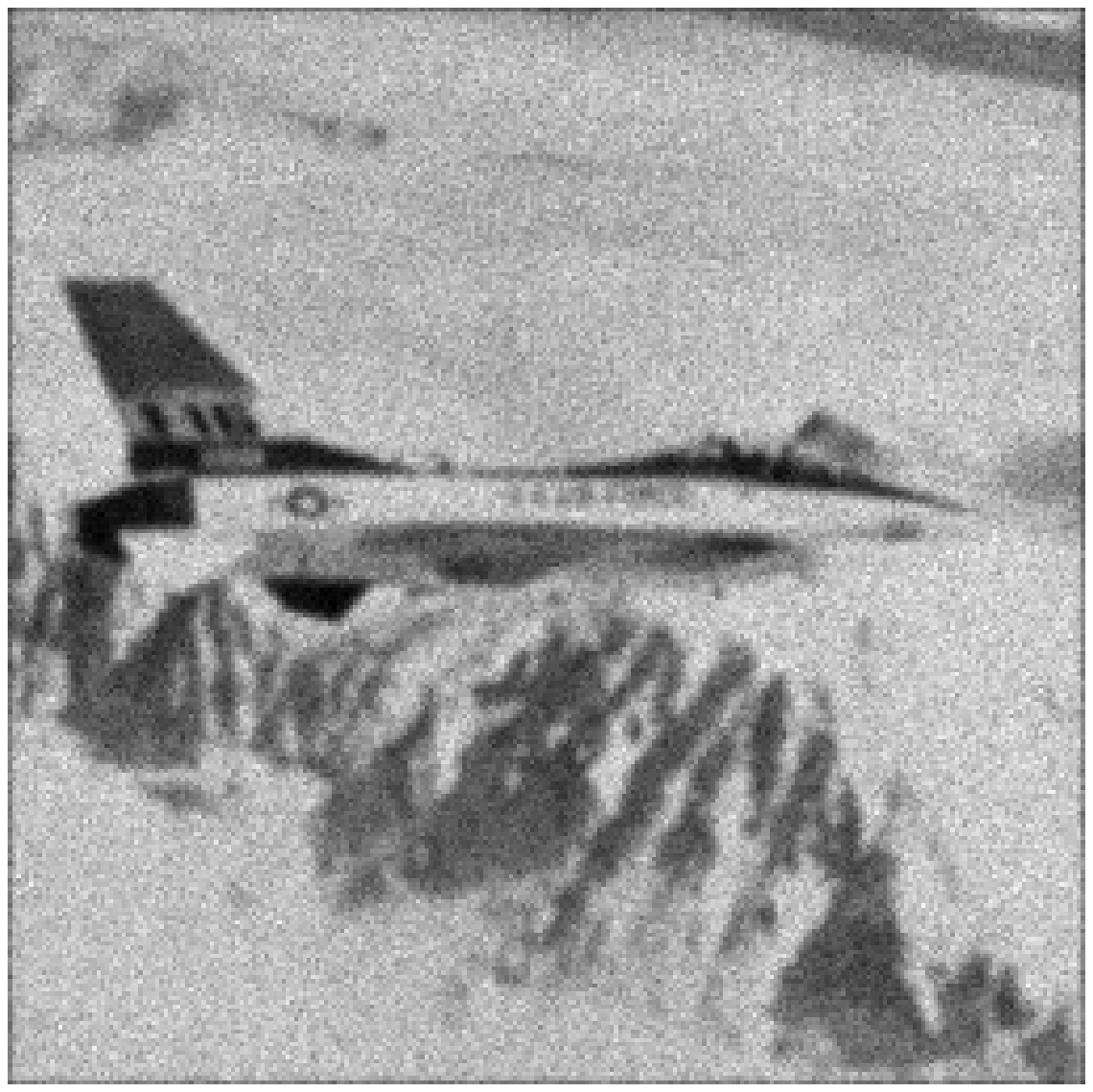}&\includegraphics[scale=\subFigScale]{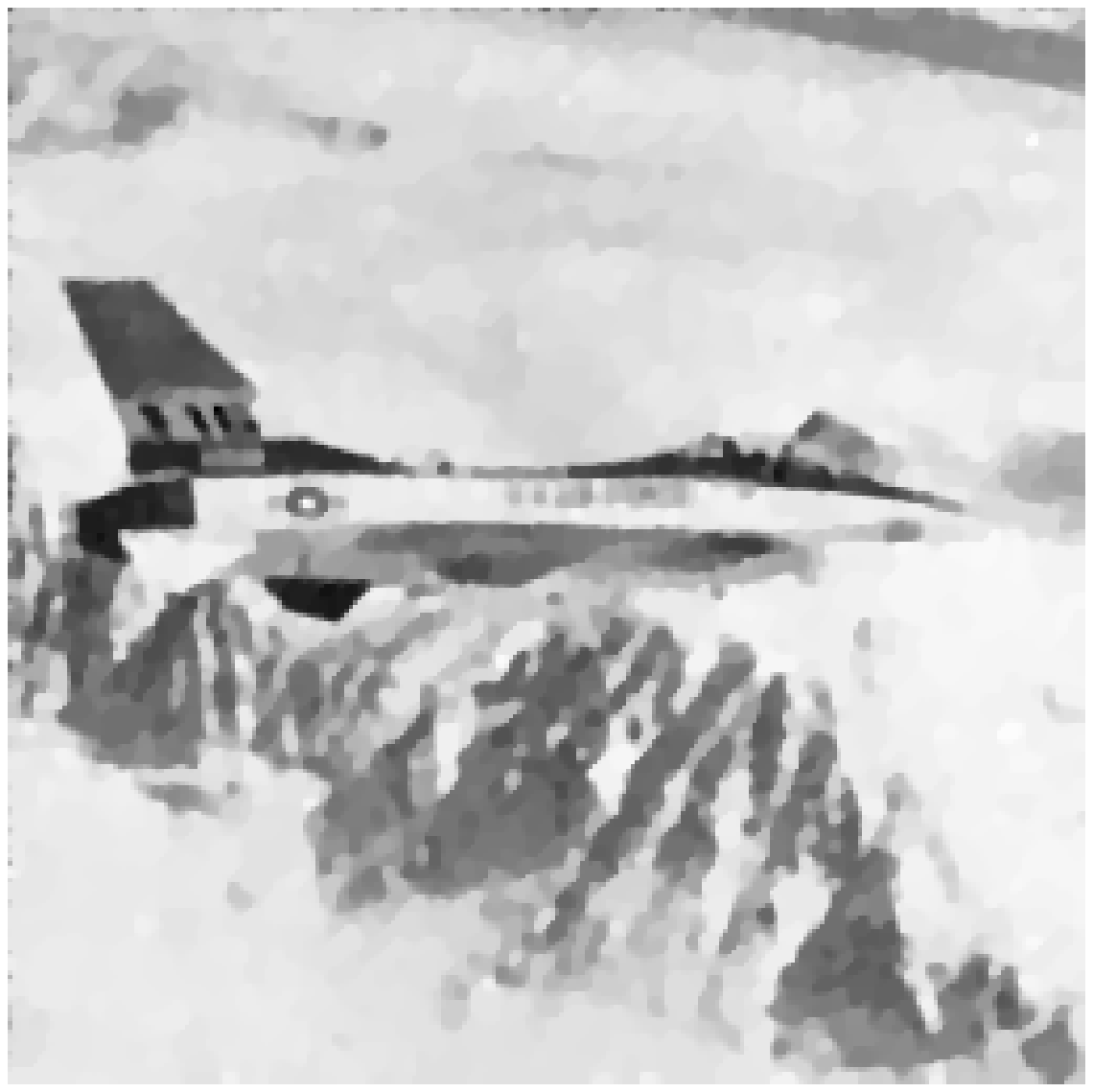}
\end{tabular}
\caption{Jetplane test problem: original object (left), blurred and noisy image (middle), and VMILAn reconstruction (right).}
\label{fig:fig1}
\end{center}
\end{figure}

\begin{figure}
\begin{center}
\begin{tabular}{cc}
\includegraphics[scale=\subFigScale]{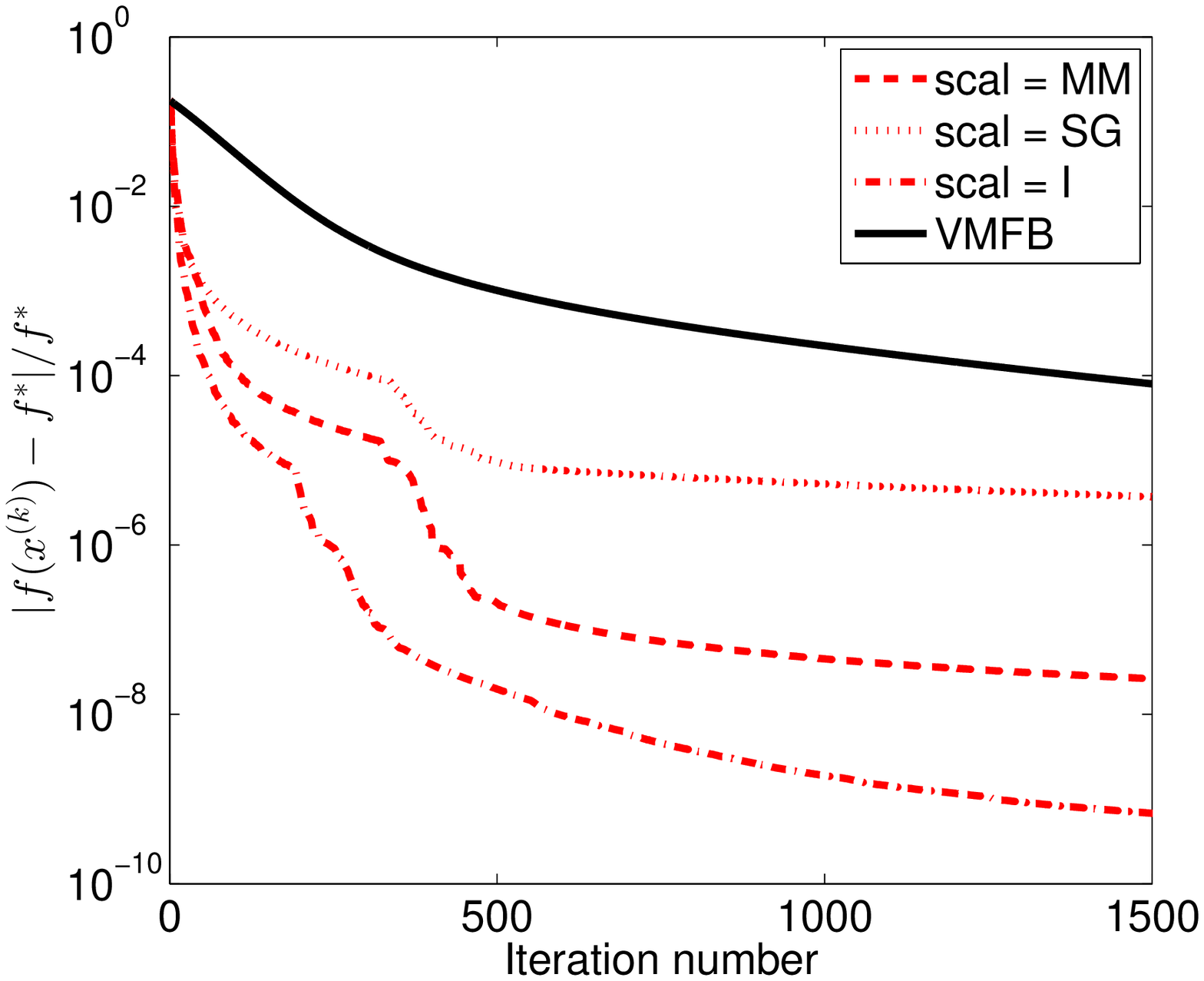}&\includegraphics[scale=\subFigScale]{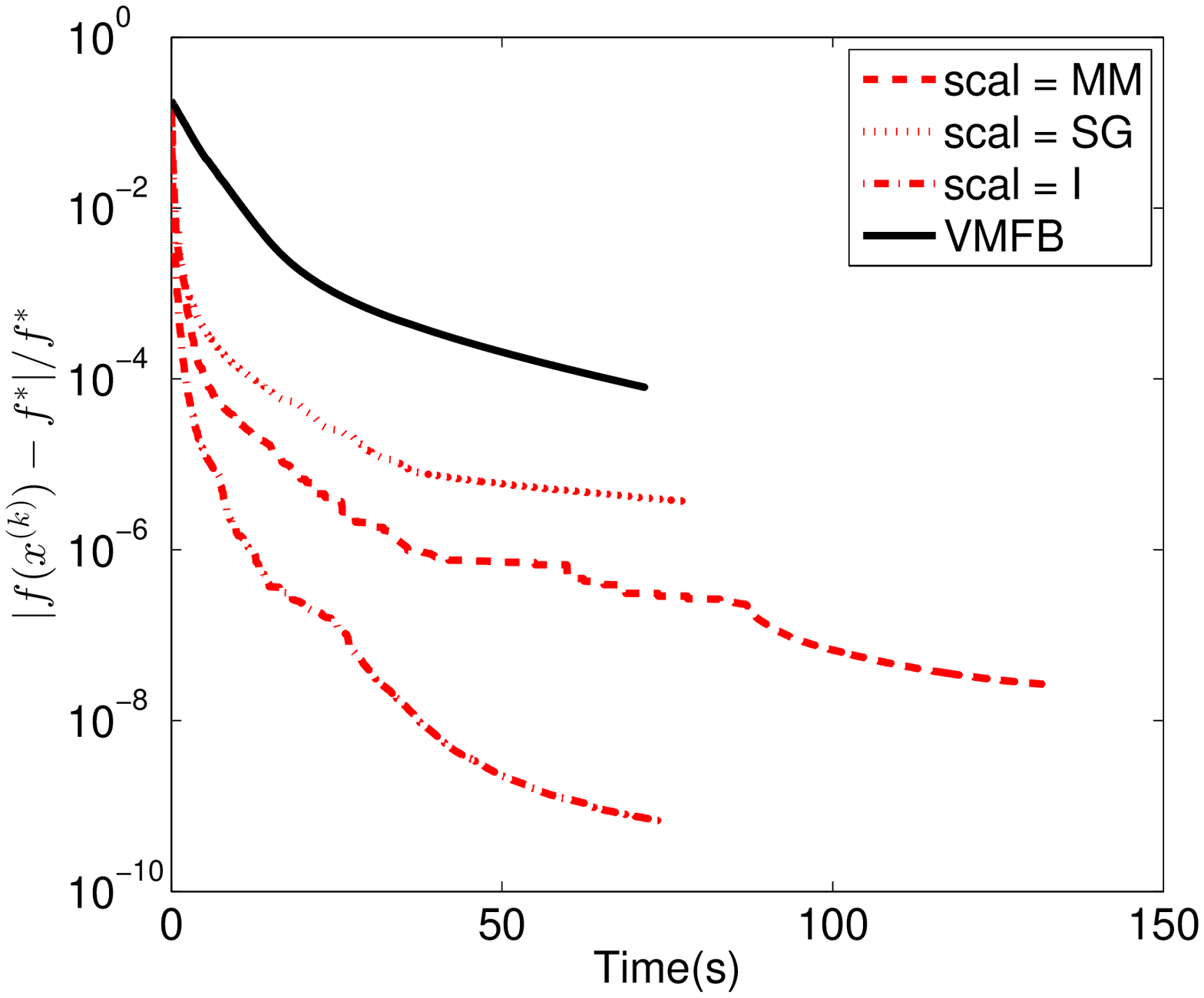}
\end{tabular}
\caption{Image deconvolution in presence of signal dependent Gaussian noise. Relative decrease of the objective function toward the minimum value with respect to the iteration number (left) and computational time in seconds (right).}
\label{fig:fig2}
\end{center}
\end{figure}

Figure \ref{fig:fig2} reports the relative decrease of the objective function with respect to the minimum value $f^*$ as a function of the iteration number and of the computational time. We can observe a faster decrease of the objective function for Algorithm \ref{algo:nuSGP}. In this case, the best performances are achieved by choosing $D_k = I_n$ but, in general, Algorithm \ref{algo:nuSGP} significantly benefits of the variable choice of the stepsize $\alpha_k$. The inner solver for computing an approximation of the proximal point requires about 2--3 iterations per outer iteration, except for the choice SG of the matrix $D_k$. \silviacorr{In all experiments the first option in \eqref{SGPmodified} never occurred.} The reconstructed image obtained with VMILAn is shown in the right panel of figure \ref{fig:fig1}.\\

\subsection{Linear diffusion based image compression}\label{sec_lindiffus}

The second image processing application we consider is the linear diffusion based image compression considered in \cite{Ochs-etal-2014} and consists in finding the optimal interpolation points for the compression procedure (see also \cite{Galic-etal-2008,Hoeltgen-etal-2013}). In particular, the problem can be described by means of the minimization problem
\begin{equation}\label{compr}
\min_{\c\in \R^n} \frac{1}{2} \|A^{-1}C\uz - \uz\|_2^2 + \lambda \|c\|_1 + \iota_{\mathcal{C}}(\c),
\end{equation}
where $\uz \in \R^n$ denotes the original image, $\c \in \R^n$ is the so-called inpainting mask and represents the unknown weights to be assigned to each pixel in the compression step, $C={\rm{diag}}(\c) \in \R^{n \times n}$ and $A=C + (C-I_n)L_n$, being $L_n \in \R^{n \times n}$ the Laplacian operator. As concerns the choice of the feasible set $\mathcal{C}$, although the natural choice would be the cartesian product $[0,1]^n$, in our experiments we observed that better results can be obtained by allowing the inpainting mask to assume values greater than 1, and therefore we chose $\mathcal{C}=[0,1.5]^n$.\\
The presence of the non-negativity constraint allows to apply VMILAn by including the term $\lambda \|c\|_1$ in the differentiable part $f_0$ and setting $f_1(\c) = \iota_{\mathcal{C}}(\c)$. The proximal operator of $f_1$ reduces to the projection over the set $\mathcal{C}$ and thus it is computed exactly. Moreover, $f$ is a KL function, being the sum of semi-algebraic functions, and $\nabla f_0$ is Lipschitz-continuous. Finally, the boundedness of the feasible set $\mathcal{C}$ guarantees the existence of a limit point. All these facts allow to apply Corollary \ref{cor:exactSGPconverge} and to state the convergence of the sequence to a stationary point of $f$.\\
Since the gradient of $f_0$ does not suggest any natural decomposition, we consider the nonscaled version of VMILAn by setting $D_k=I_n$ for all $k$. As concerns the steplength parameter $\alpha_k$, we used the same strategy
described in the previous section by replacing \eqref{redgrad} with
\begin{equation*}
\widetilde{g}^{(k)}_j = \cases{0&if $c^{(k)}_j \in \{0,1.5\}$, \\ \left[\nabla f_0(\c^{(k)})\right]_j&if $c^{(k)}_j \in (0,1.5)$}
\end{equation*}
and setting $\alpha_{\max} = 10^5$.\\ 
We compare VMILAn with the iPiano algorithm \cite[Algorithm 4]{Ochs-etal-2014}, which is a forward--backward method with extrapolation whose sequence generated converges to a critical point of \eqref{compr} thanks to the KL property of the objective function. Unlike the choice made for VMILAn, here we followed the implementation of the authors and left the term $\lambda \|c\|_1$ in the $f_1$ part of the objective function (we tried also the other splitting but we always obtained worse results). All the other parameters defining iPiano have been chosen as suggested in \cite{Ochs-etal-2014}. The test problems are the same used in \cite[\S 5.2.2]{Ochs-etal-2014} and named ``trui'', ``peppers'' and  ``walter'' (see figure \ref{fig:fig3}). In table \ref{table} we report the iteration numbers performed by the two methods together with the corresponding values of the objective function, density and mean squared error (MSE) computed by
\begin{equation*}
{\rm{MSE}}(\u,\uz)=\frac{1}{n}\sum_{i=1}^n (u_i - u_i^0)^2,
\end{equation*}
where $\u=A^{-1}C\uz$ is the reconstructed image. Moreover, since in this case it seems that the two algorithms do not converge to the same minima, in figure \ref{fig:fig4} we do not plot the relative distance between the objective function and the minimum but we show the decrease of the objective function with respect to the iteration number and the computational time in seconds. The behaviour of the steplength $\alpha_k$ and the linesearch parameter $\lambda_k$ is also shown in the right column of figure \ref{fig:fig4}. Finally, in the right column of figure \ref{fig:fig3} the reconstructions obtained with VMILAn are given.

\begin{table}
\centering
\begin{tabular}{cccccc}
\hline
Test image & Algorithm & Iterations & Obj. func. & Density & MSE \\
\hline
\multirow{2}{*}{trui}    & iPiano & 1000 & 21.58 & 4.97\% & 17.27 \\
                         & VMILAn &  599 & 21.50 & 4.80\% & 17.95 \\
\hline
\multirow{2}{*}{peppers} & iPiano & 1000 & 23.10 & 5.95\% & 19.64 \\
                         & VMILAn &  655 & 23.01 & 5.81\% & 19.99 \\
\hline
\multirow{2}{*}{walter}  & iPiano & 1000 & 10.32 & 5.10\% & 8.27 \\
                         & VMILAn &  699 & 10.23 & 4.66\% & 8.55 \\
\hline
\end{tabular}
\caption{Summary of two algorithms for three test images.}
\label{table}
\end{table}

\def\subFigScale{0.3}
\begin{figure}
\hspace{-10mm}
\begin{center}
\begin{tabular}{ccc}
\includegraphics[scale=\subFigScale]{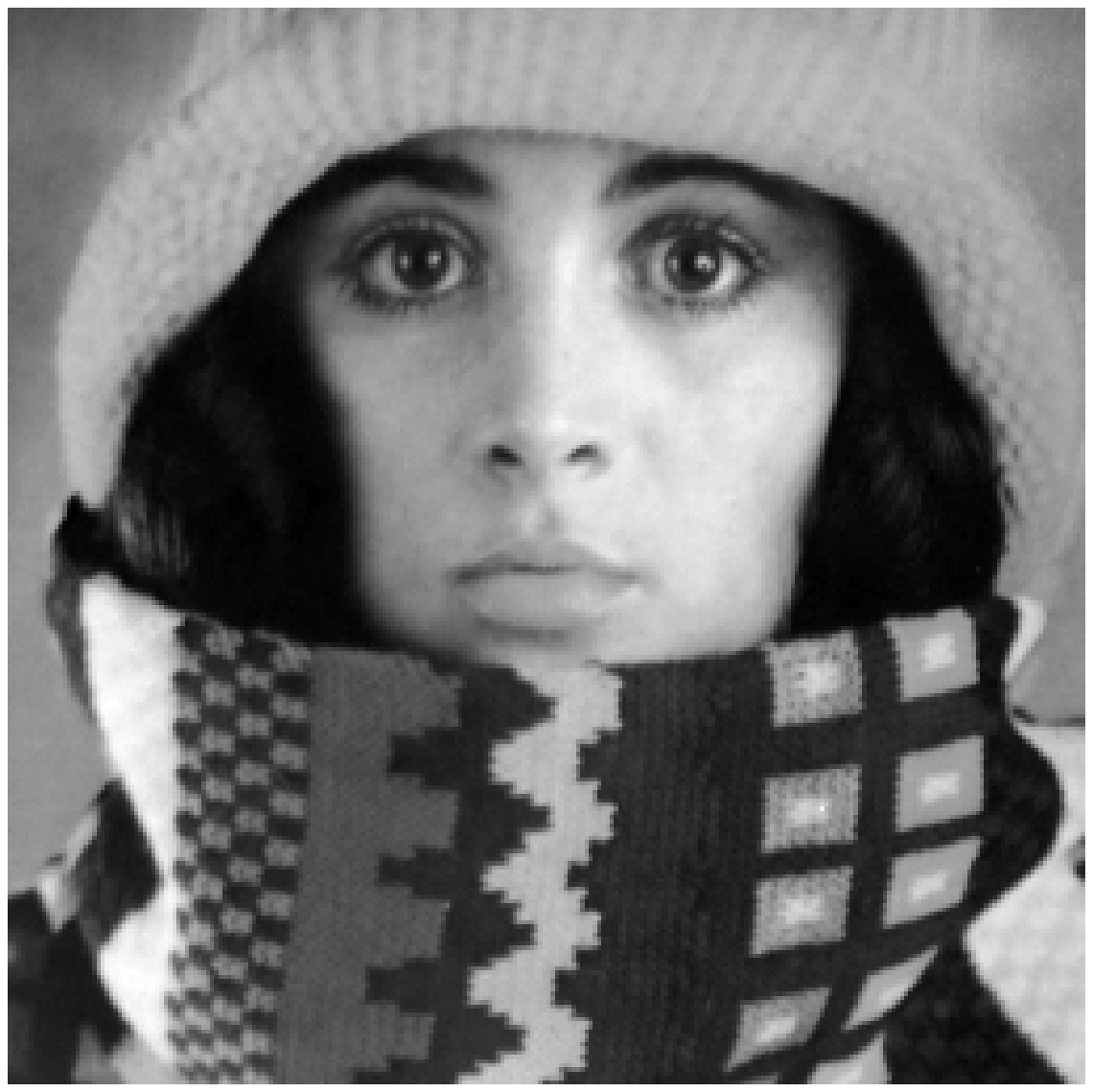}&\includegraphics[scale=\subFigScale]{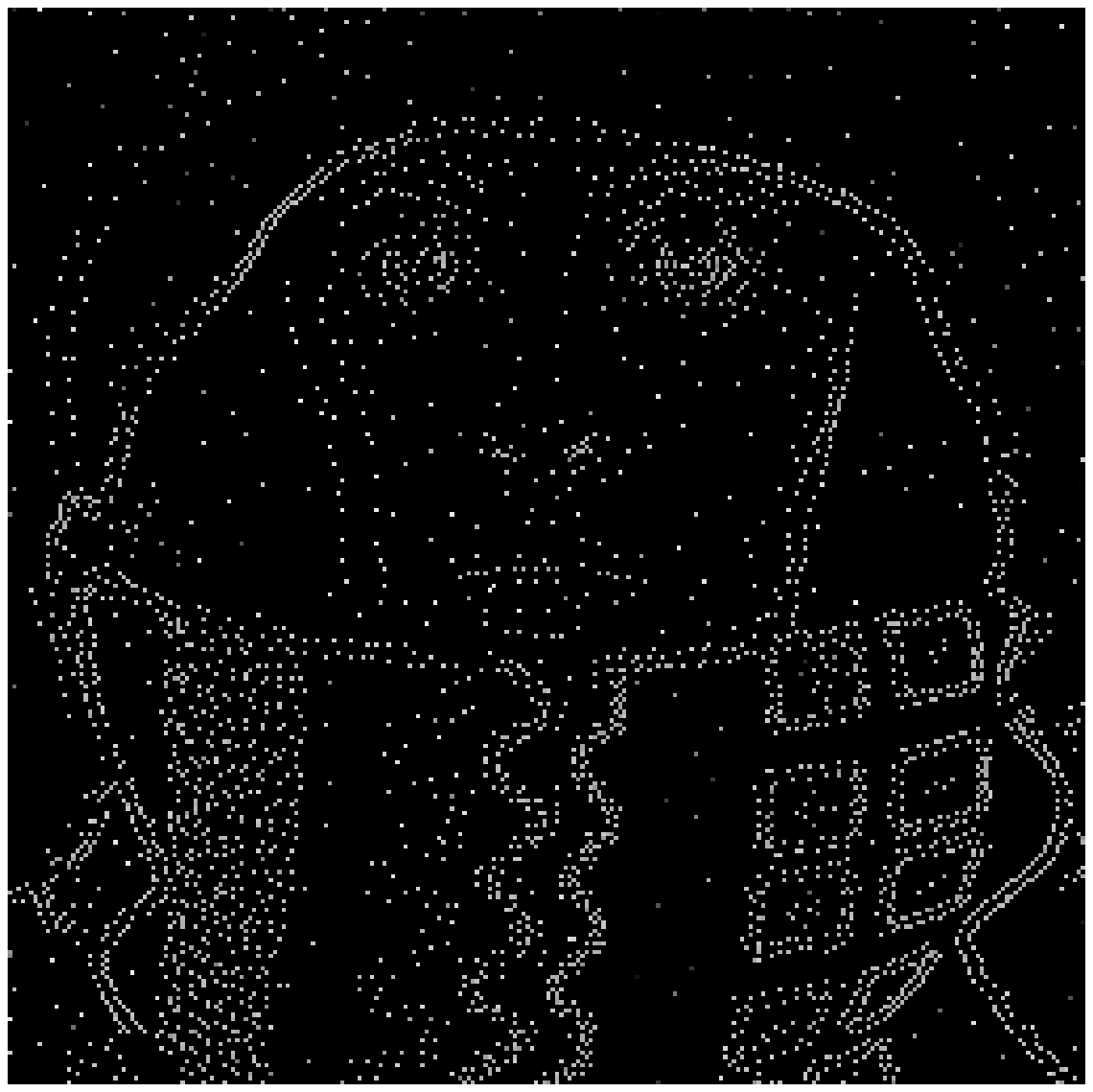}&\includegraphics[scale=\subFigScale]{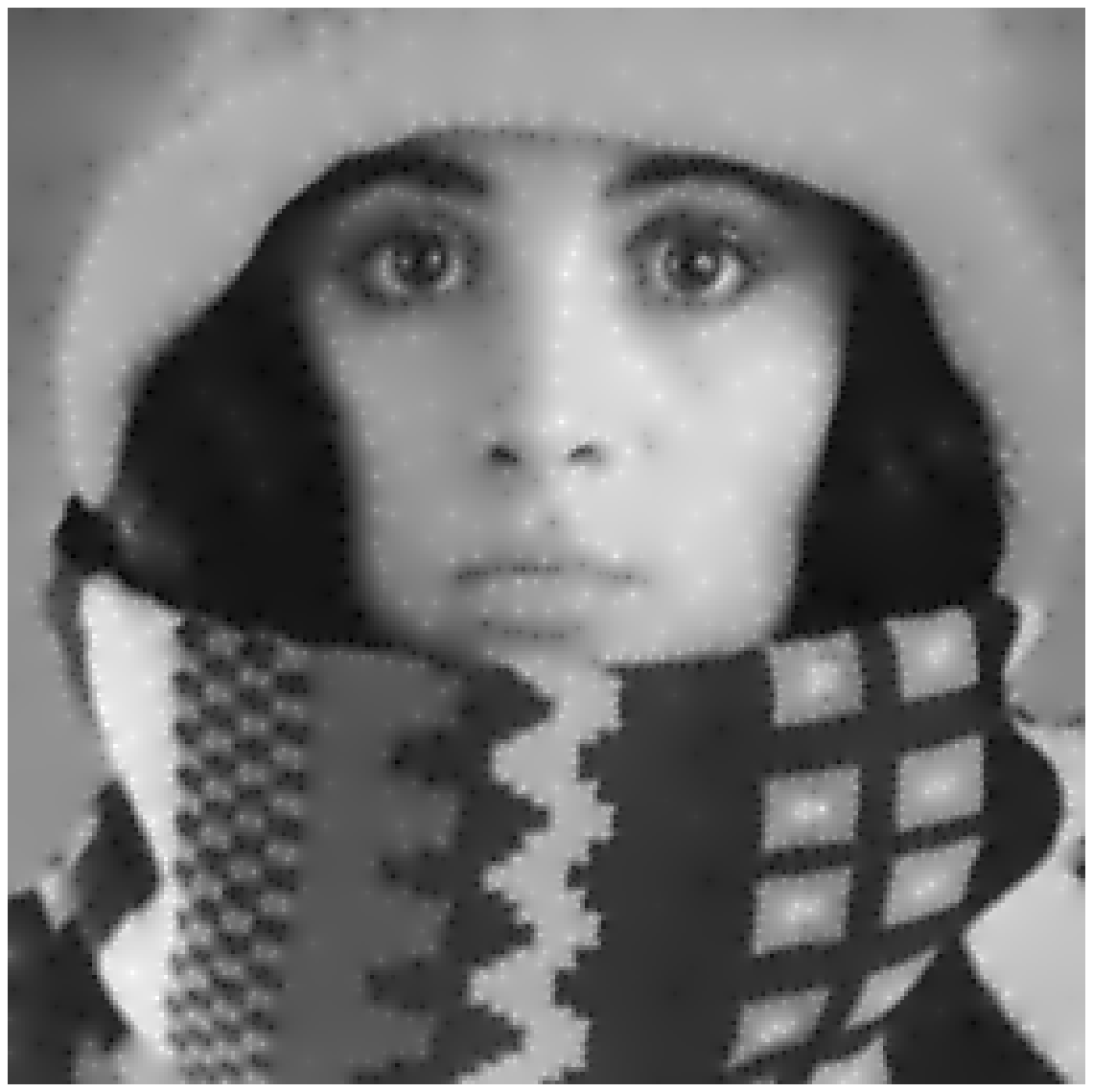}\\
\includegraphics[scale=\subFigScale]{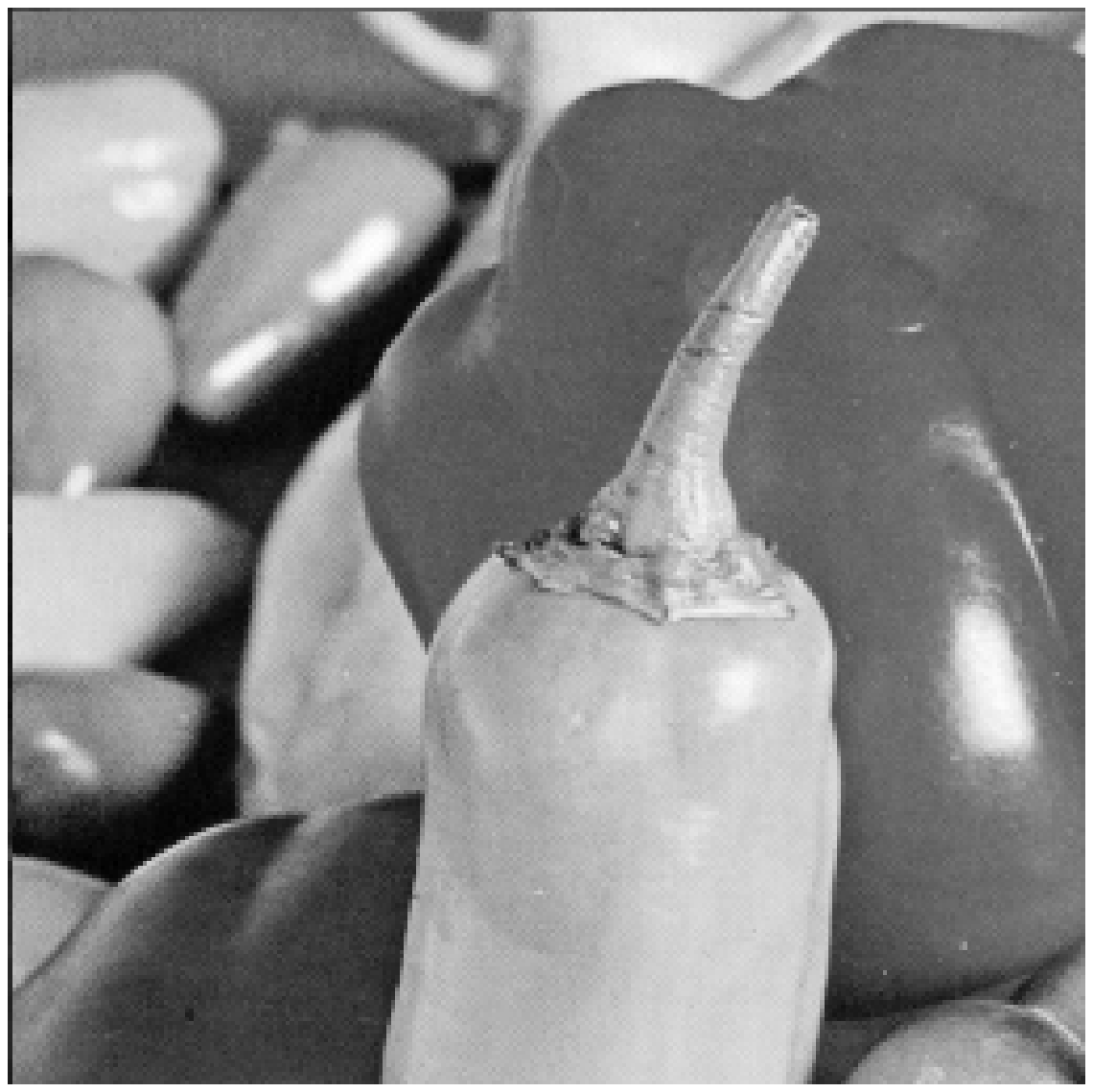}&\includegraphics[scale=\subFigScale]{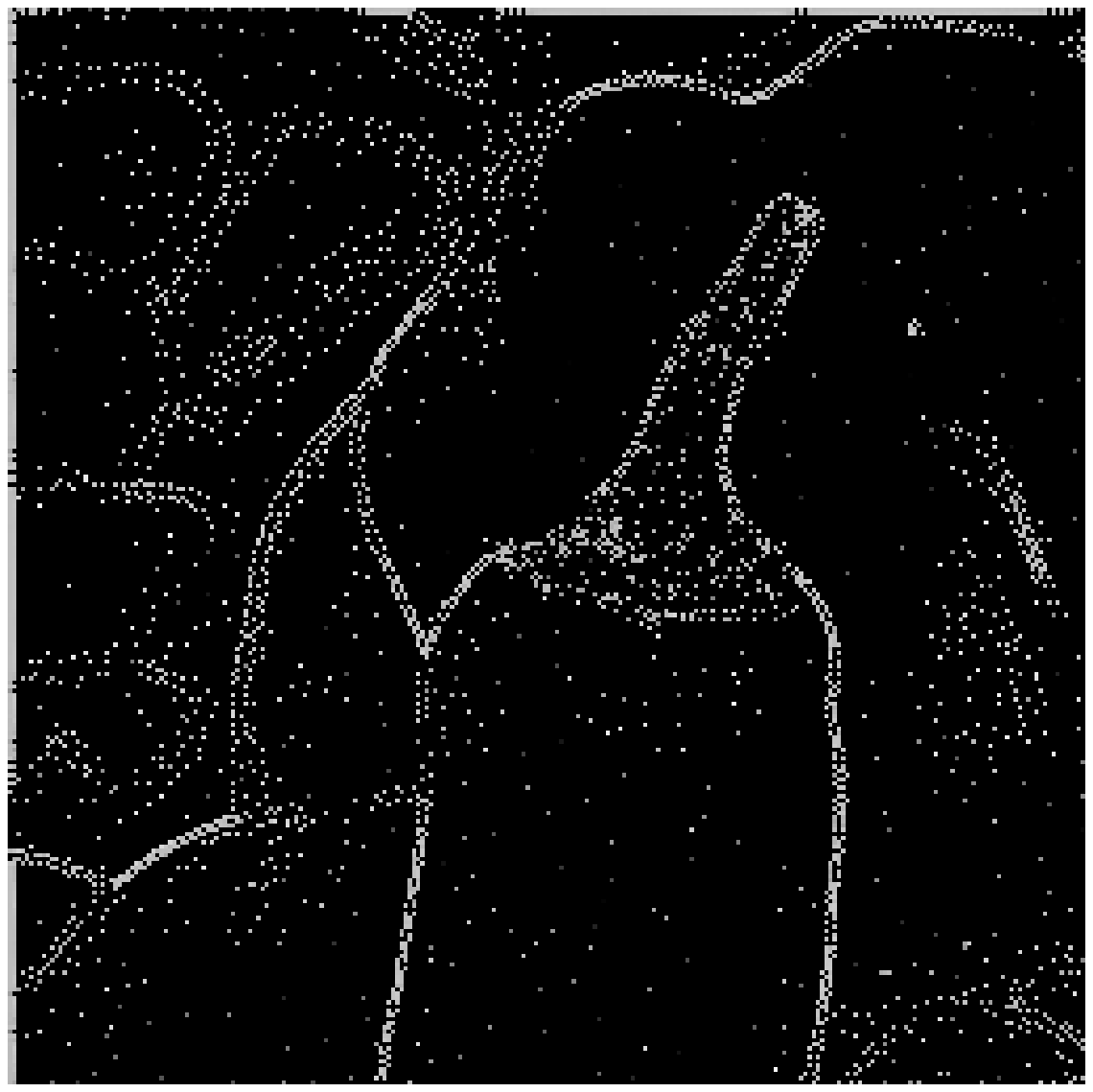}&\includegraphics[scale=\subFigScale]{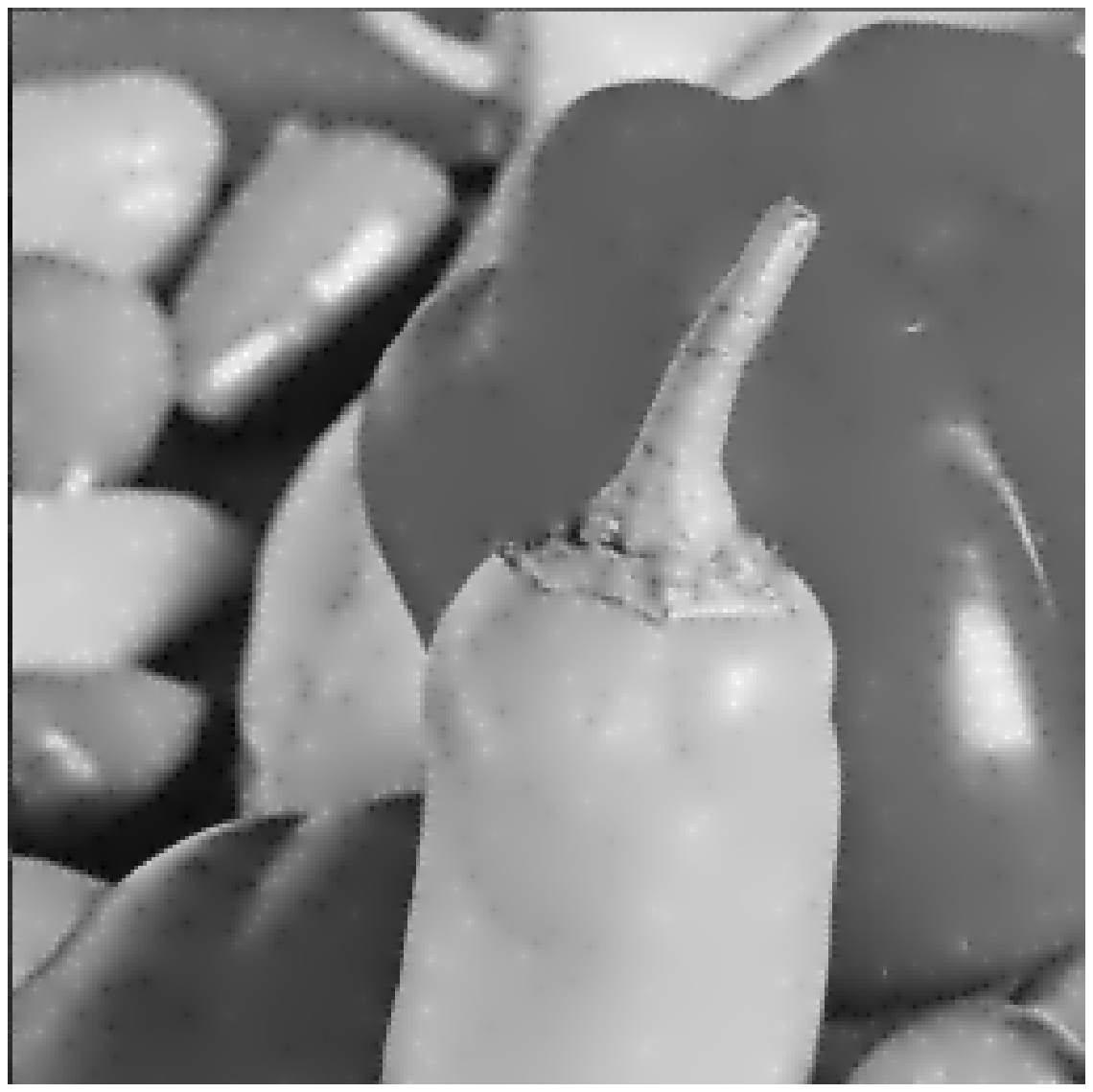}\\
\includegraphics[scale=\subFigScale]{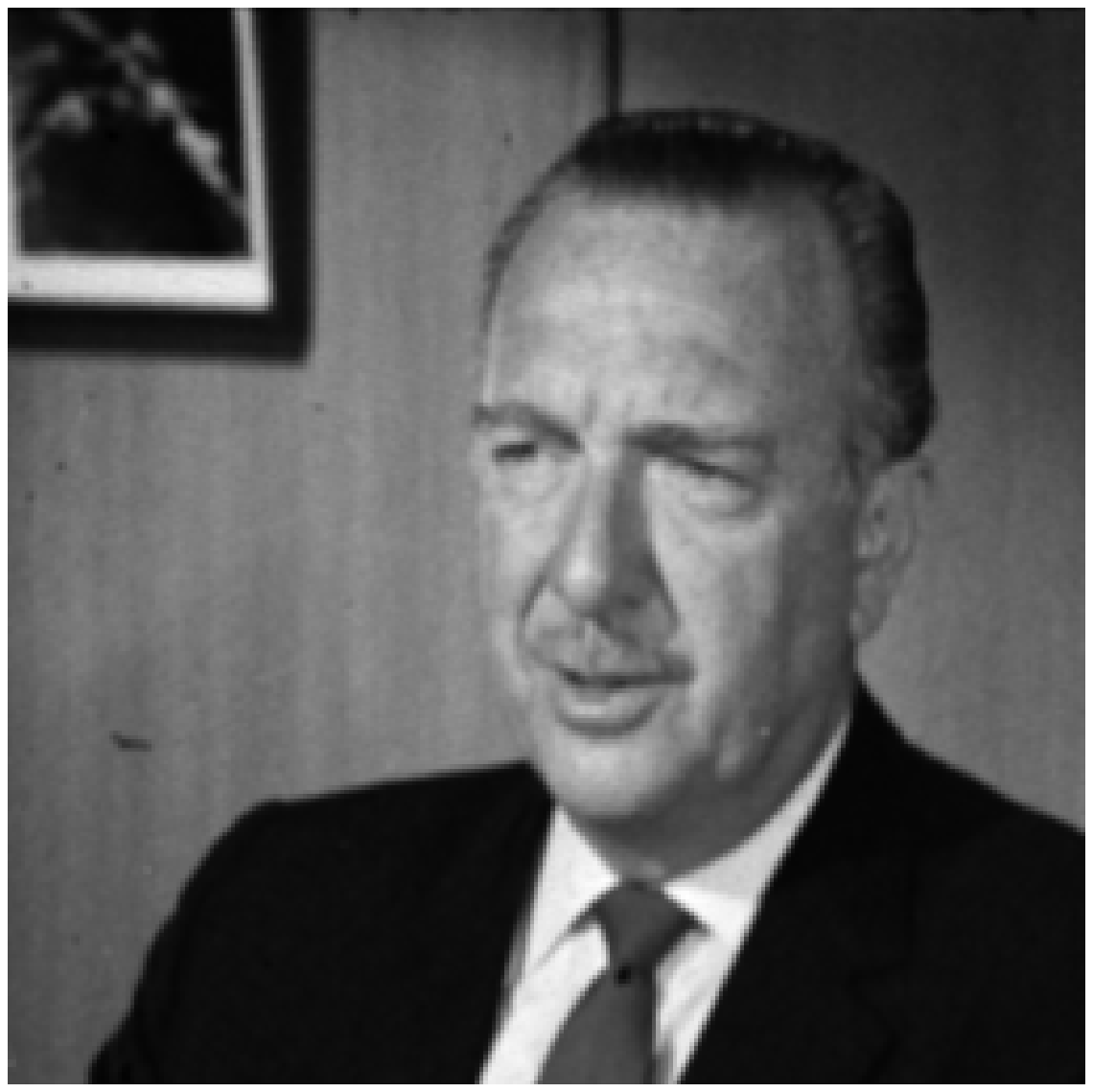}&\includegraphics[scale=\subFigScale]{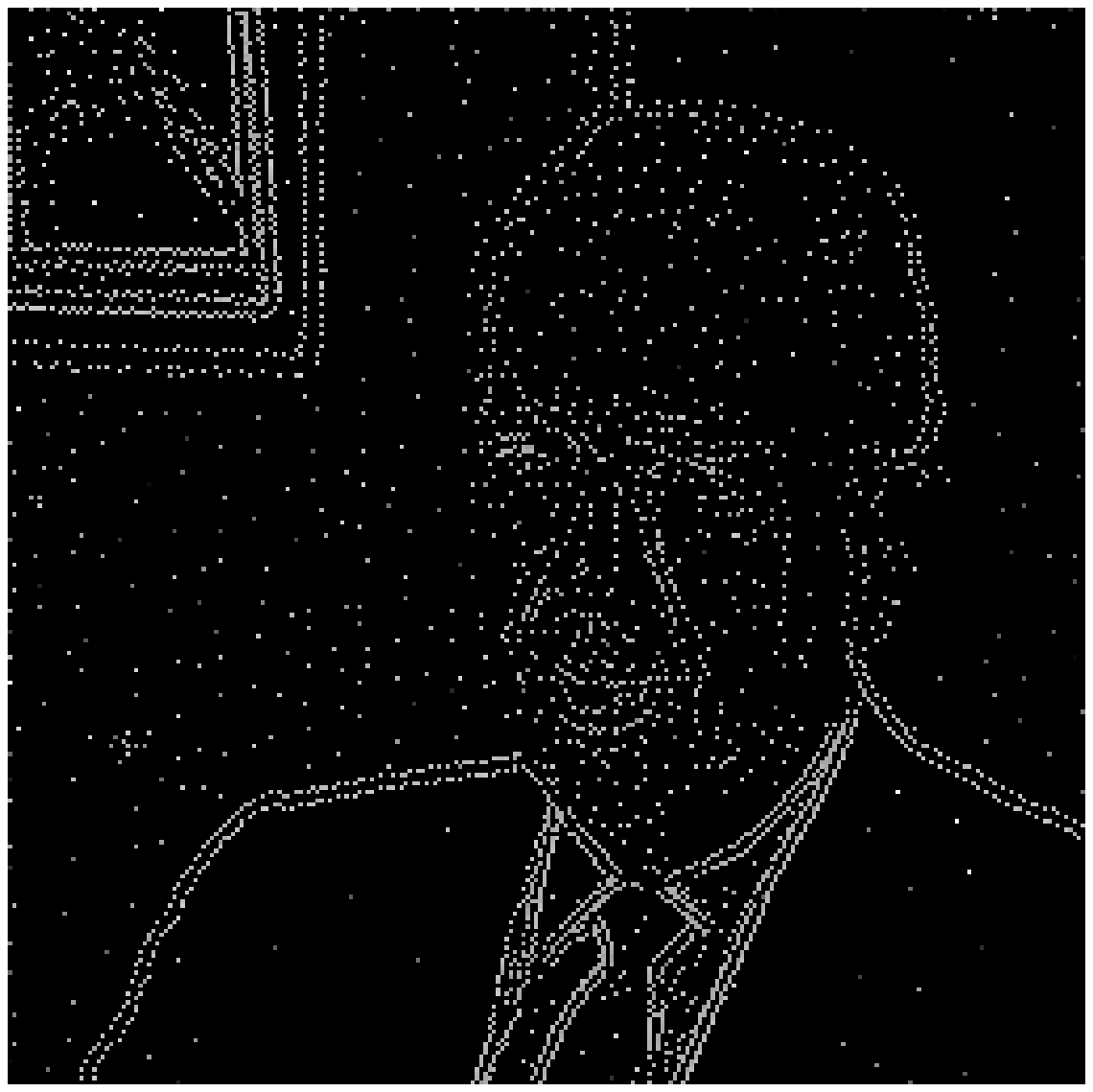}&\includegraphics[scale=\subFigScale]{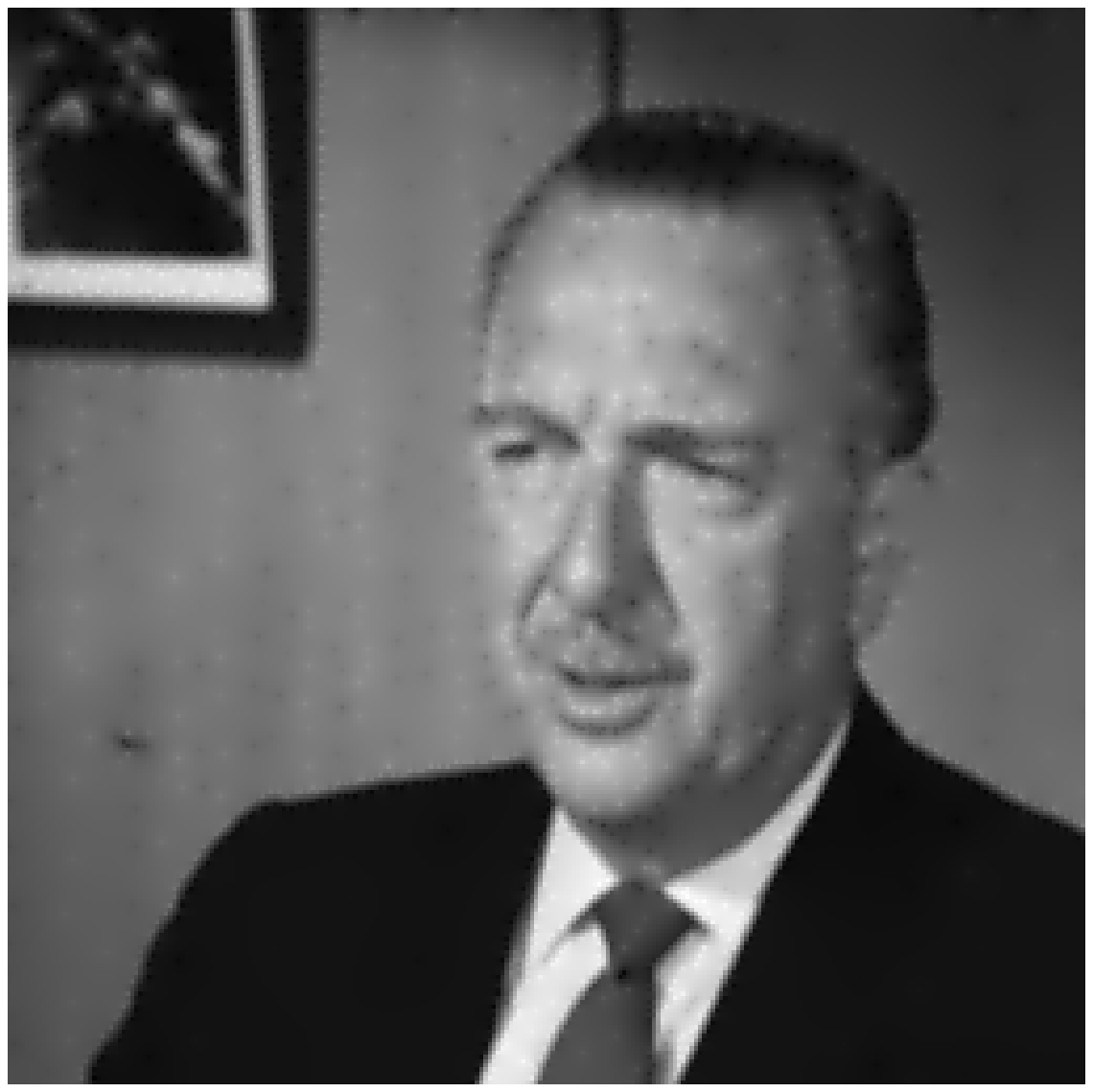}
\end{tabular}
\caption{Trui (top row), peppers (central row) and walter (bottom row) datasets. Original image (left), inpainting mask (middle) and VMILAn reconstruction (right).}\label{fig:fig3}
\end{center}
\end{figure}

\begin{figure}
\begin{center}
\makebox[\linewidth]{
\begin{tabular}{ccc}
\includegraphics[scale=\subFigScale]{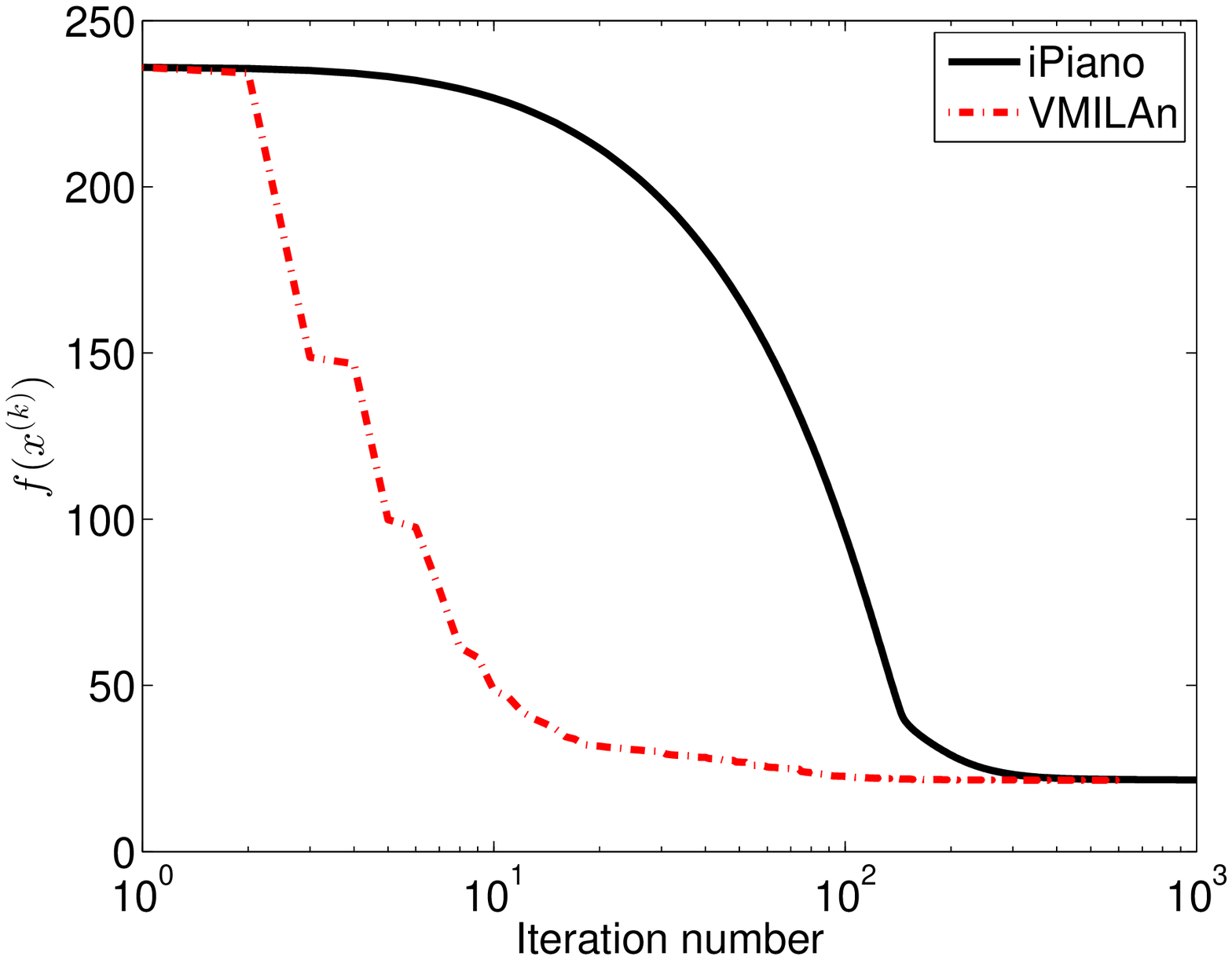}&\includegraphics[scale=\subFigScale]{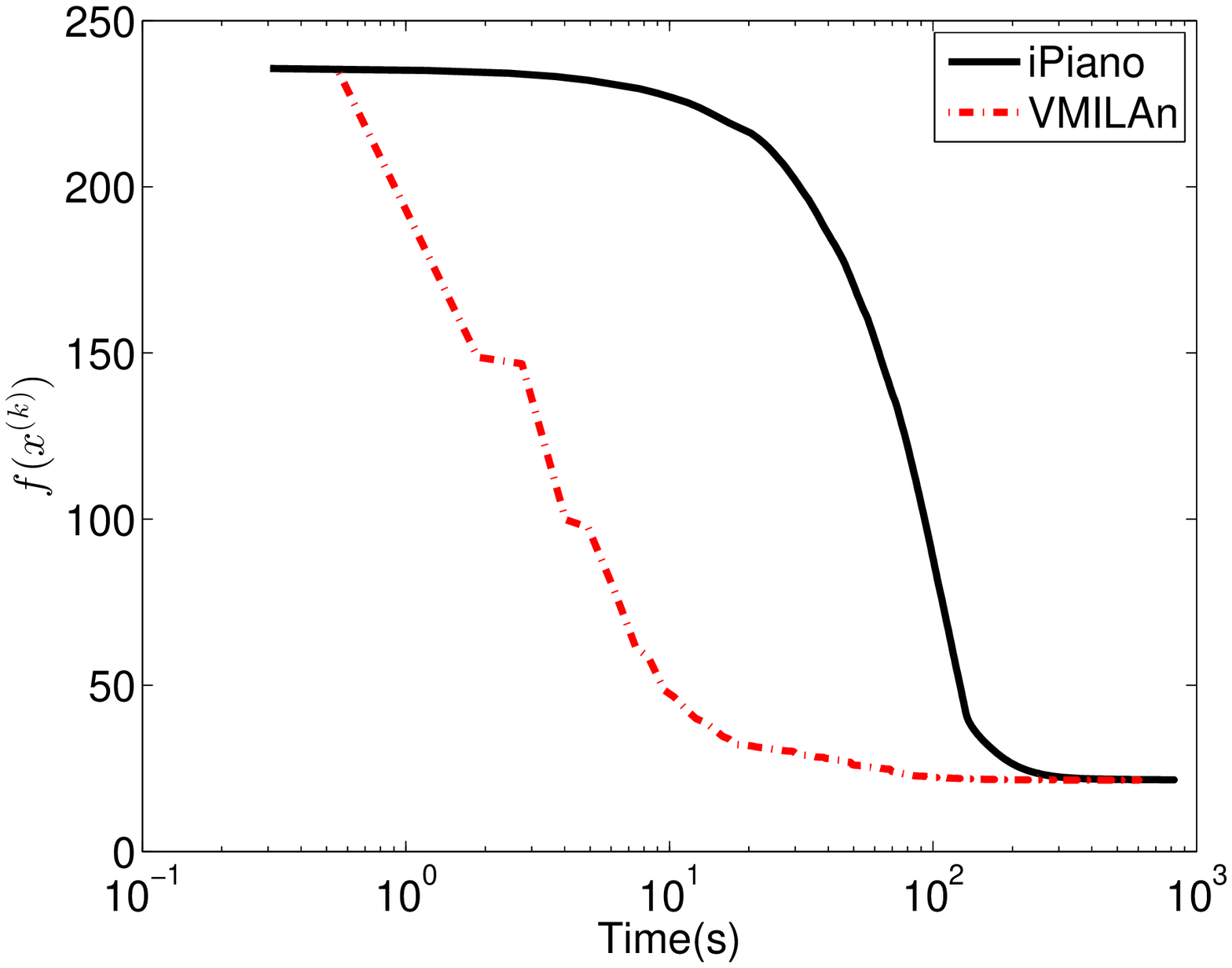}&\includegraphics[scale=0.41]{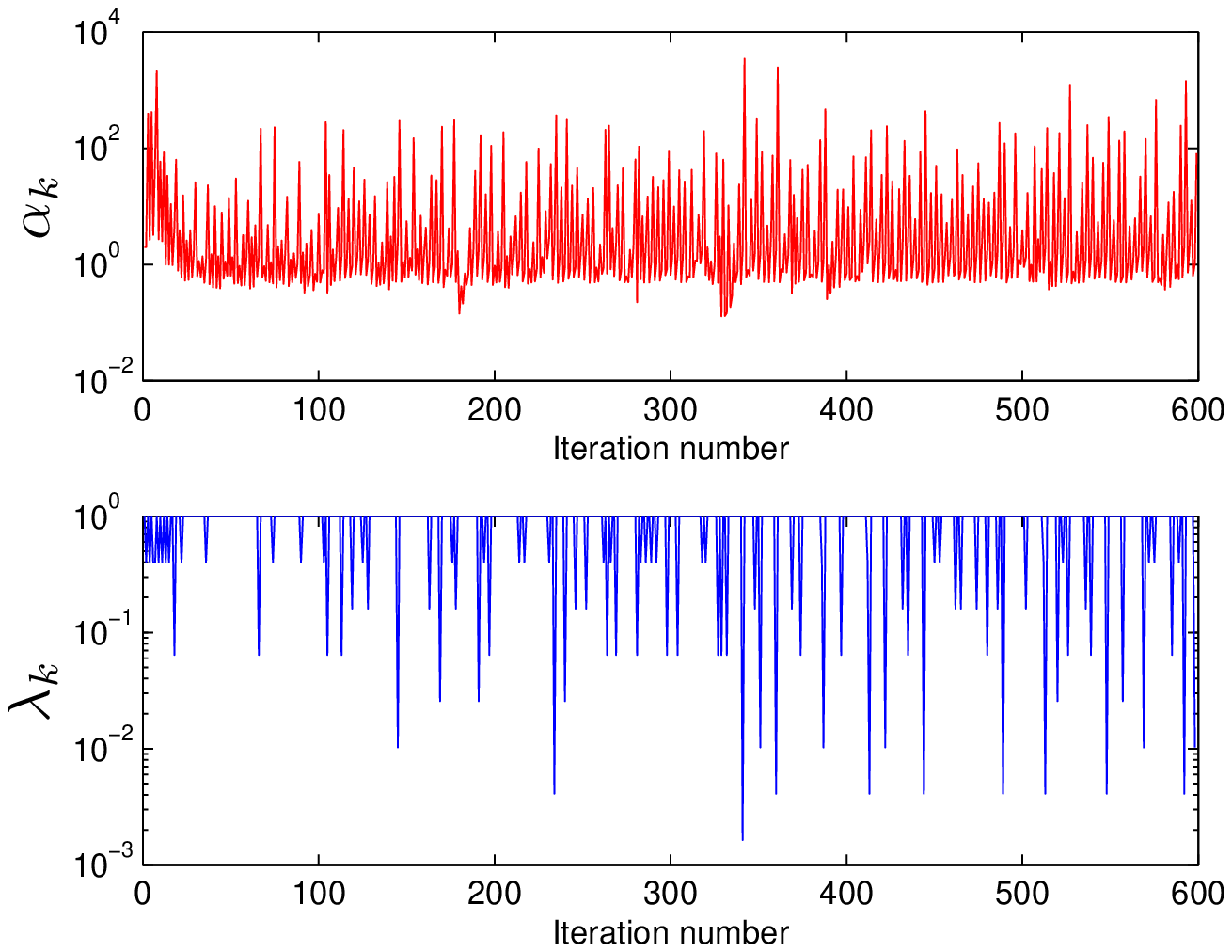}\\
\includegraphics[scale=\subFigScale]{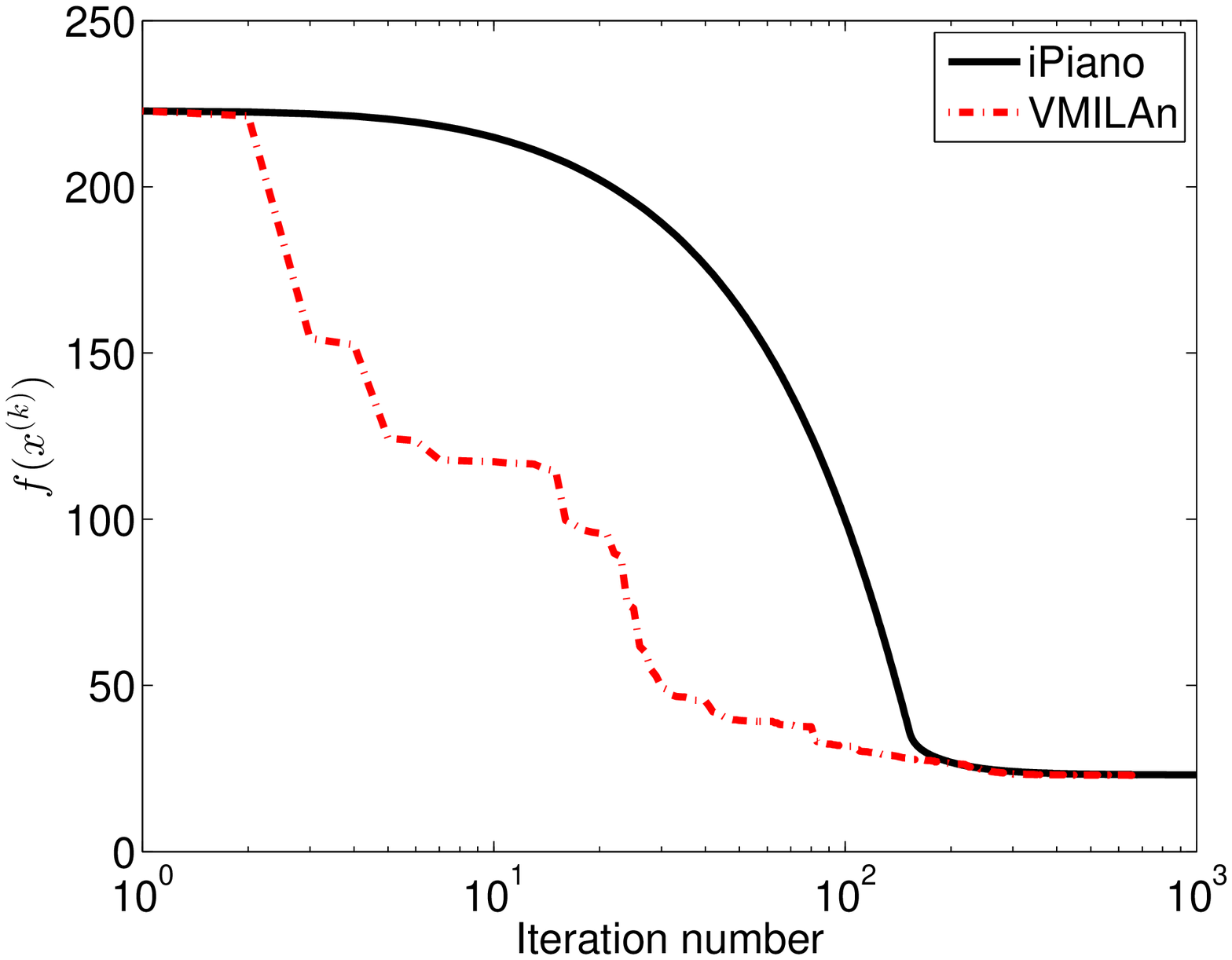}&\includegraphics[scale=\subFigScale]{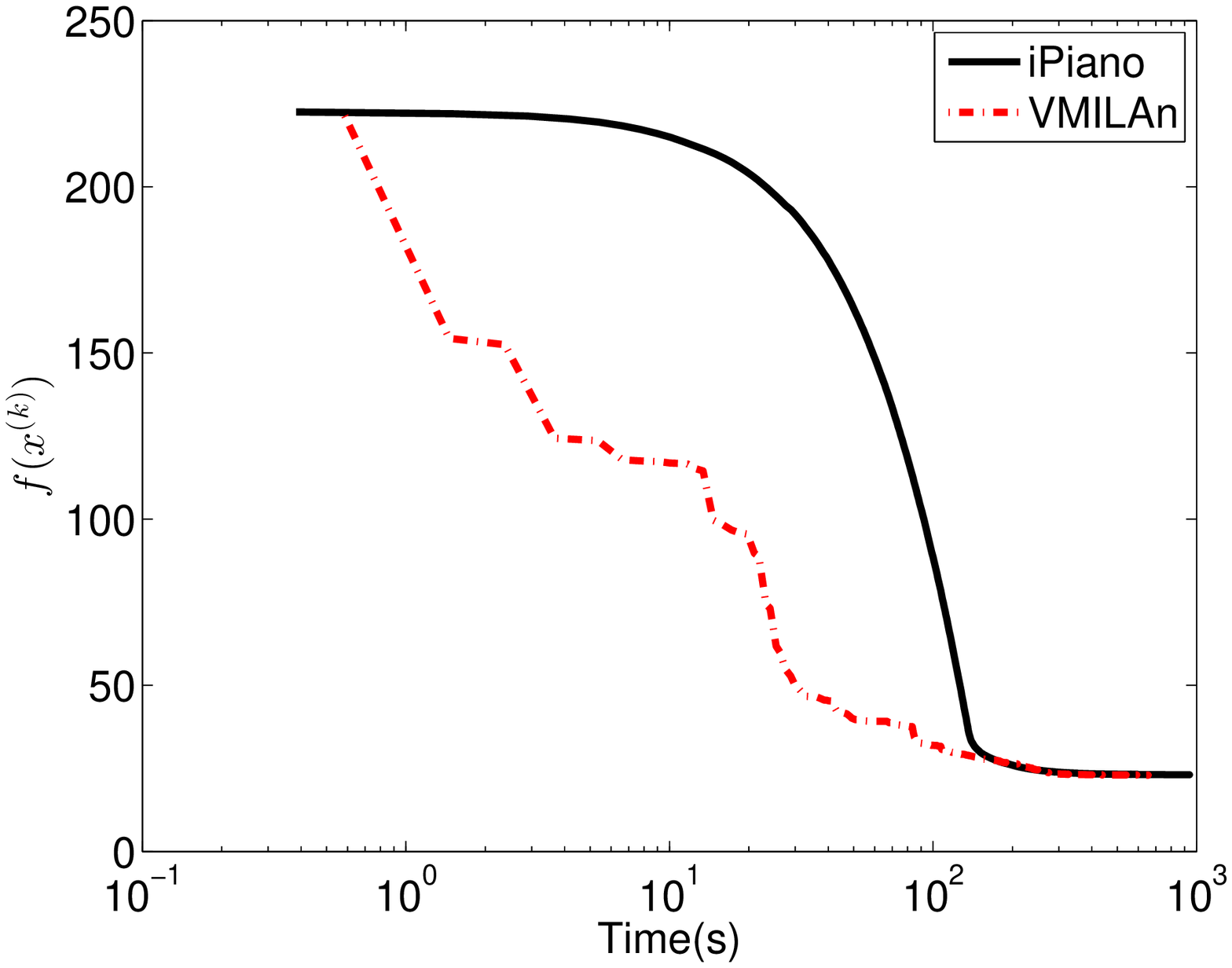}&\includegraphics[scale=0.41]{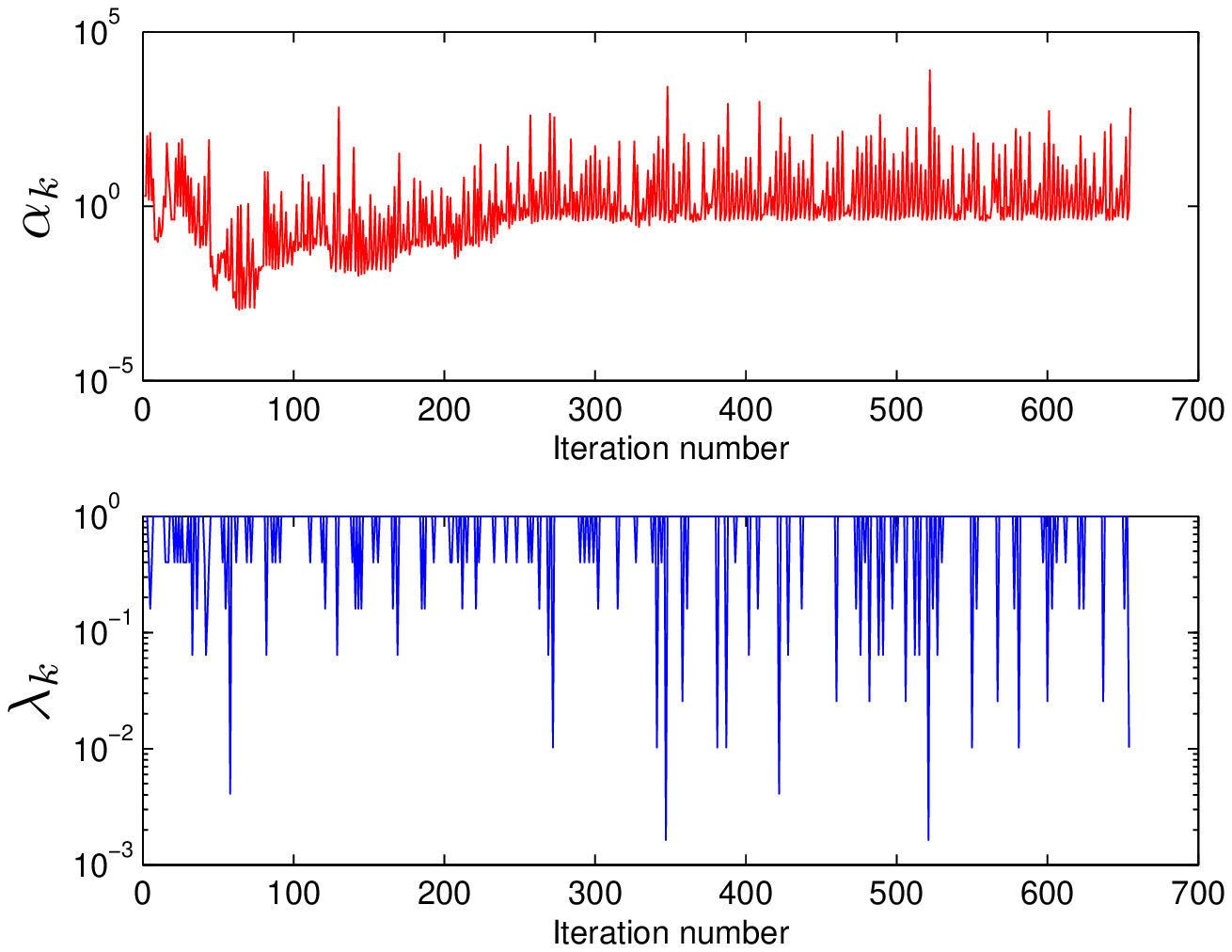}\\
\includegraphics[scale=\subFigScale]{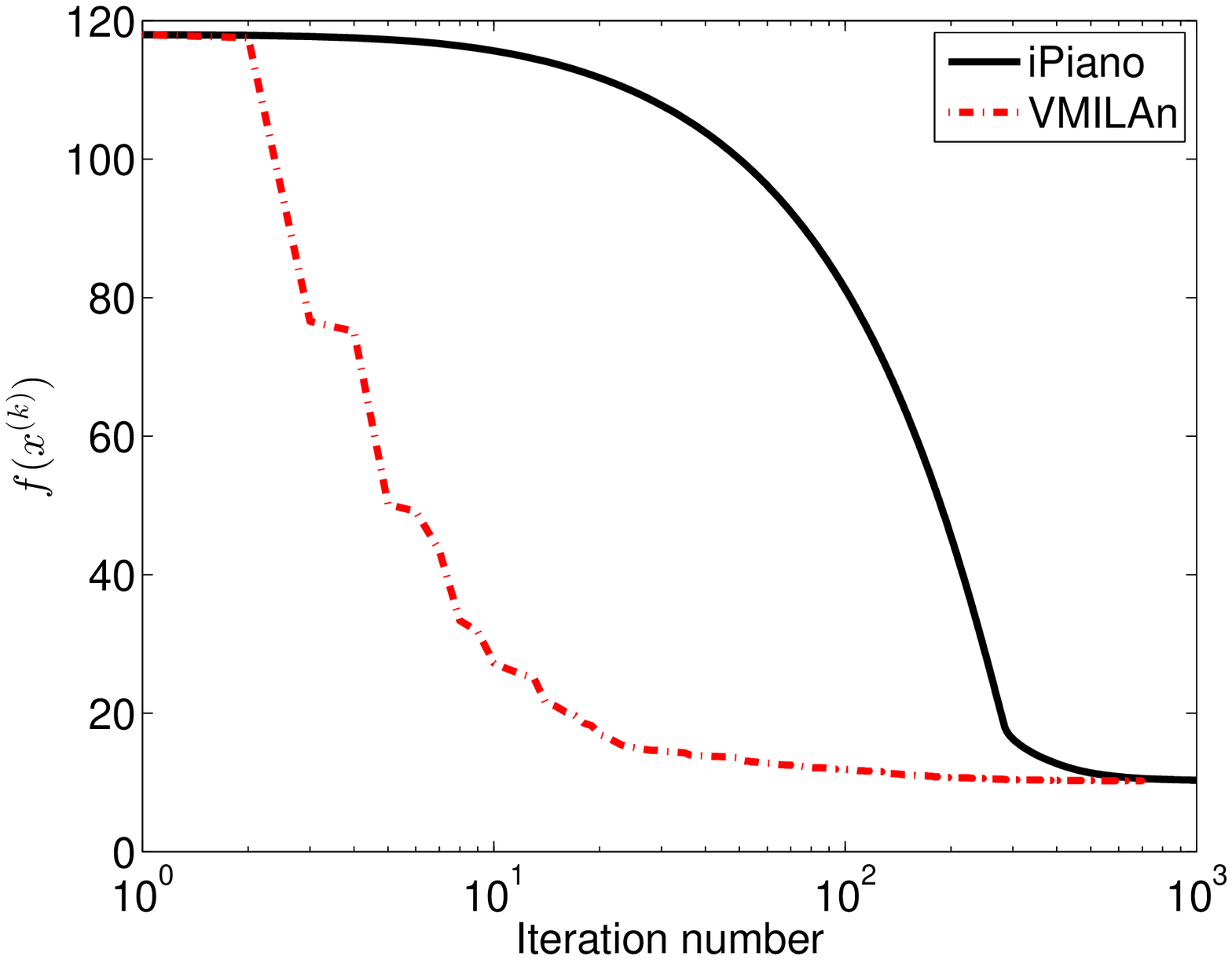}&\includegraphics[scale=\subFigScale]{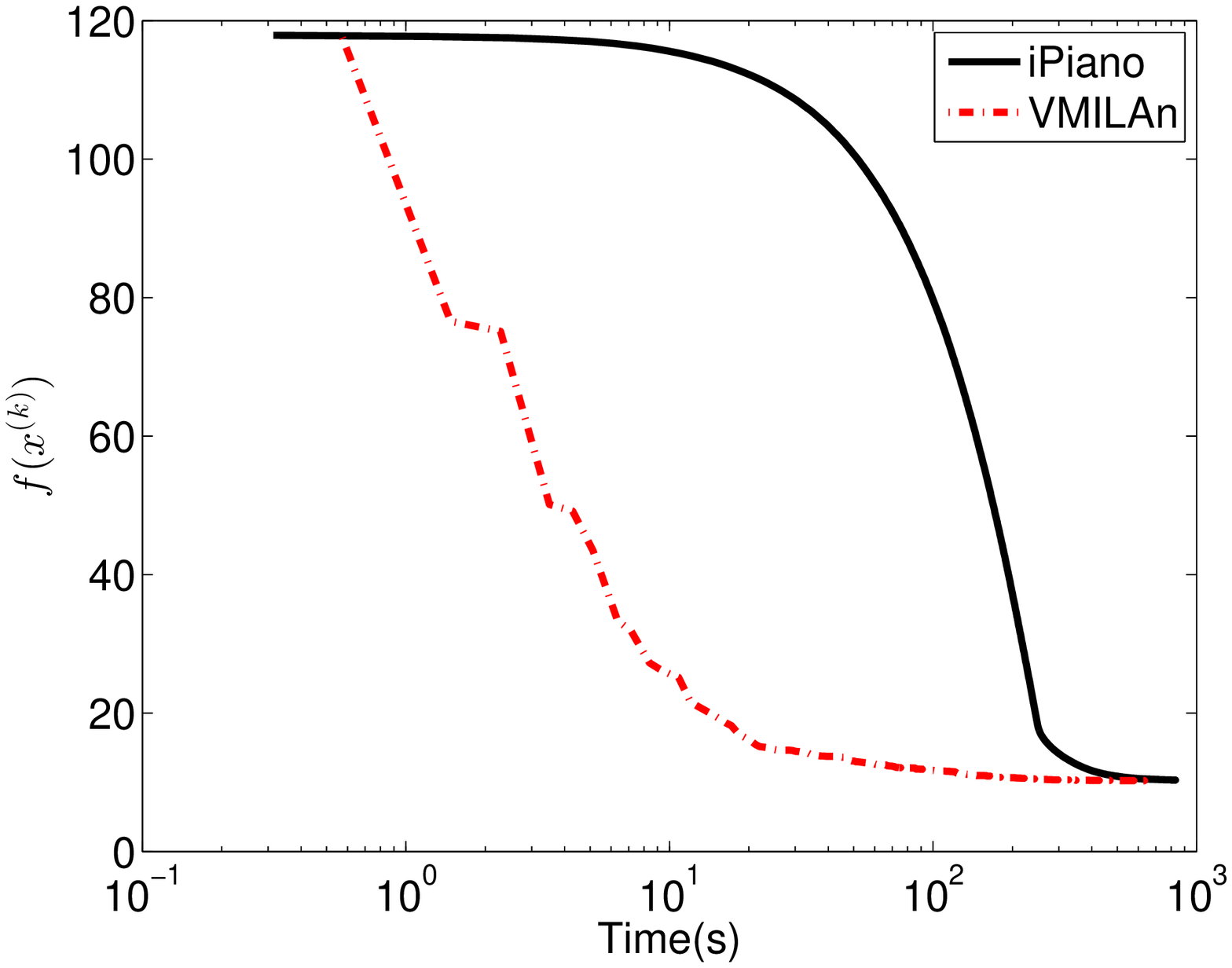}&\includegraphics[scale=0.41]{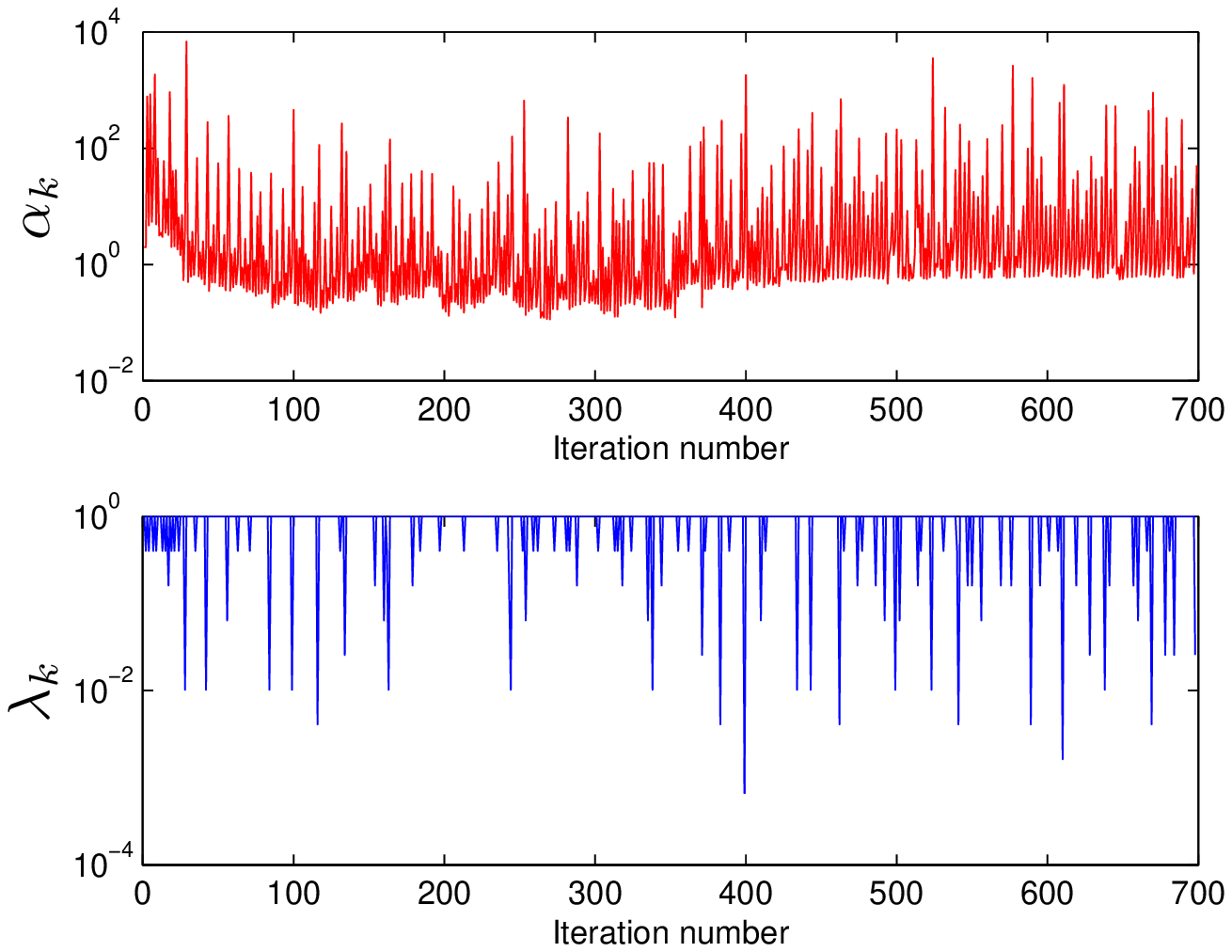}
\end{tabular}}
\caption{Linear diffusion based image compression for the trui (top row), peppers (central row) and walter (bottom row) datasets. Decrease of the objective function with respect to the iteration number (left) and computational time in seconds (center), and behaviour of the parameters $\alpha_k$ and $\lambda_k$ with respect to the iteration number (right) represented by the red and blue plots, respectively.}\label{fig:fig4}
\end{center}
\end{figure}

As remarked in the previous numerical test, also in this application VMILAn seems to be competitive if compared to other forward-backward approaches, since it is able to provide comparable reconstructions by performing a lower number of iterations and allowing a reduction of the computational time. \silviacorr{In all the experiments described in this section, the first option in \eqref{SGPmodified} never occurred.}

\subsection{Image deblurring in presence of Cauchy noise}

As a final test, we take into account the problem of recovering a blurred image corrupted by Cauchy noise. In \cite{Sciacchitano-etal-2015} the authors propose a novel variational model aimed to face Cauchy noise image restoration based on total variation regularization. More in detail, they suppose the degraded image $g\in \R^n$ can be written as $g = Hu + v$, where $u\in \R^n$ is the true object, $H\in\mathbb{R}^{n\times n}$ is the discretization of the blurring operator and $v\in \R^n$ represents the random noise which models a Cauchy distribution corresponding to a density of the form
\begin{equation*}
f(v) = \frac{1}{\pi}\frac{\gamma}{\gamma^2 + v^2}, \qquad \gamma >0.
\end{equation*}
The discrete version of the optimization problem they suggest can be formulated as follows
\begin{equation}\label{eq:CauchyNoise_minprob}
 \min_{x\in\mathbb{R}^n} \frac{\lambda}{2} \sum_{i=1}^n \log\bigl(\gamma^2 + ((Hx)_i - g_i)^2\bigr) + \sum_{i=1}^n \| \nabla_i x \|,
\end{equation}
where $\lambda$ is the regularization parameter. We decide to force the solution of being non-negative and therefore we add to the objective function in \eqref{eq:CauchyNoise_minprob} the indicator function of the non-negative orthant. In these settings, the nondifferentiable part of the function to minimize becomes as in \eqref{eq:TV+Ind} (with $\rho =1$), while $f_0$ reduces to the logarithmic discrepancy.\\
We consider two datasets borrowed by \cite[Section 5.2]{Sciacchitano-etal-2015}. In particular the operator $H$ is associated to a Gaussian blur with a window size $9\times 9$ and standard deviation equal to 1, while $\gamma$ has been set equal to $0.02$. We report the true images and the distorted ones in figure \ref{fig:fig7}.
\begin{figure}[htbp!]
\begin{center}
\begin{tabular}{cccc}
\includegraphics[scale=0.2]{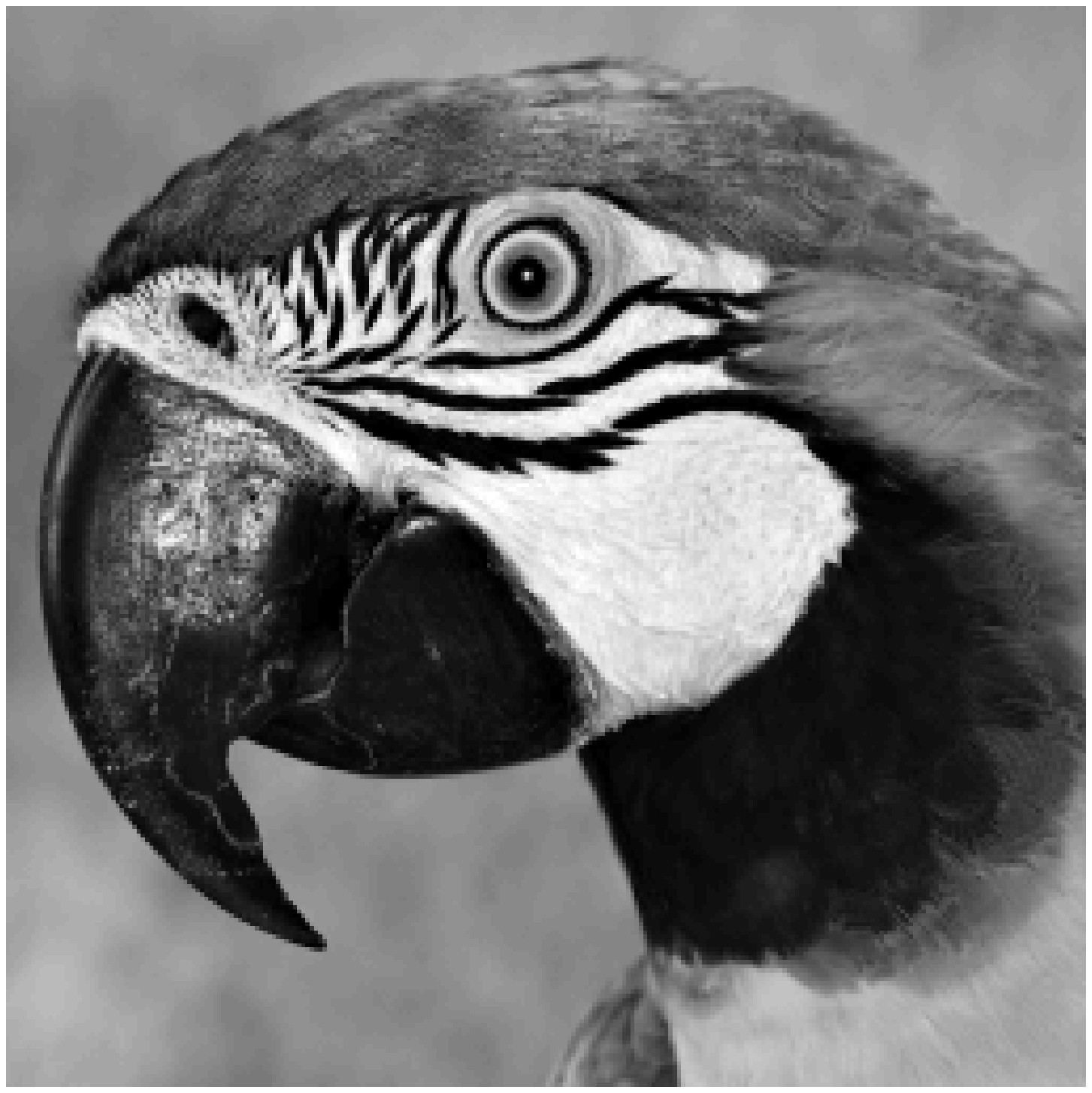}&\includegraphics[scale=0.2]{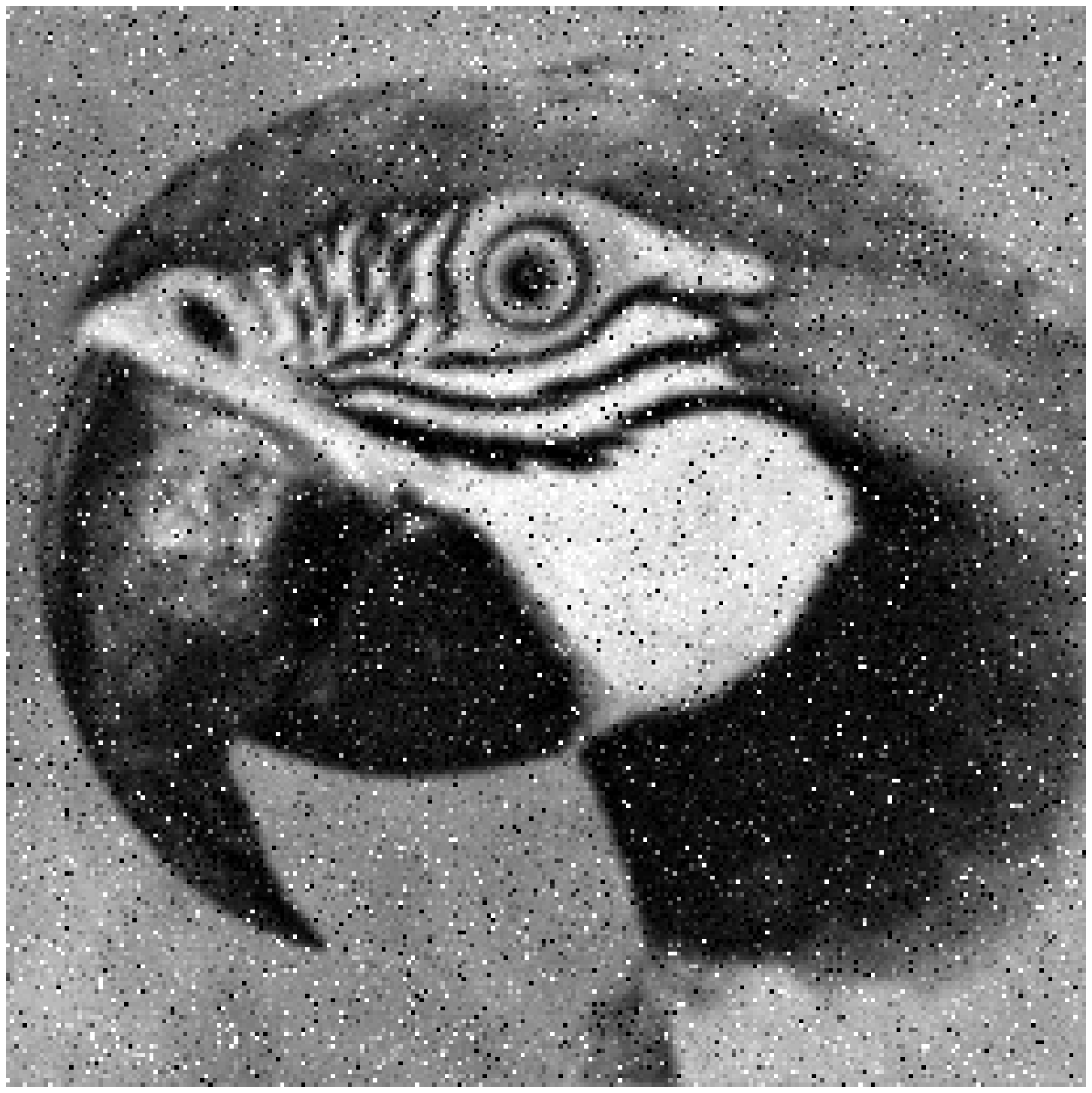}&\includegraphics[scale=0.25]{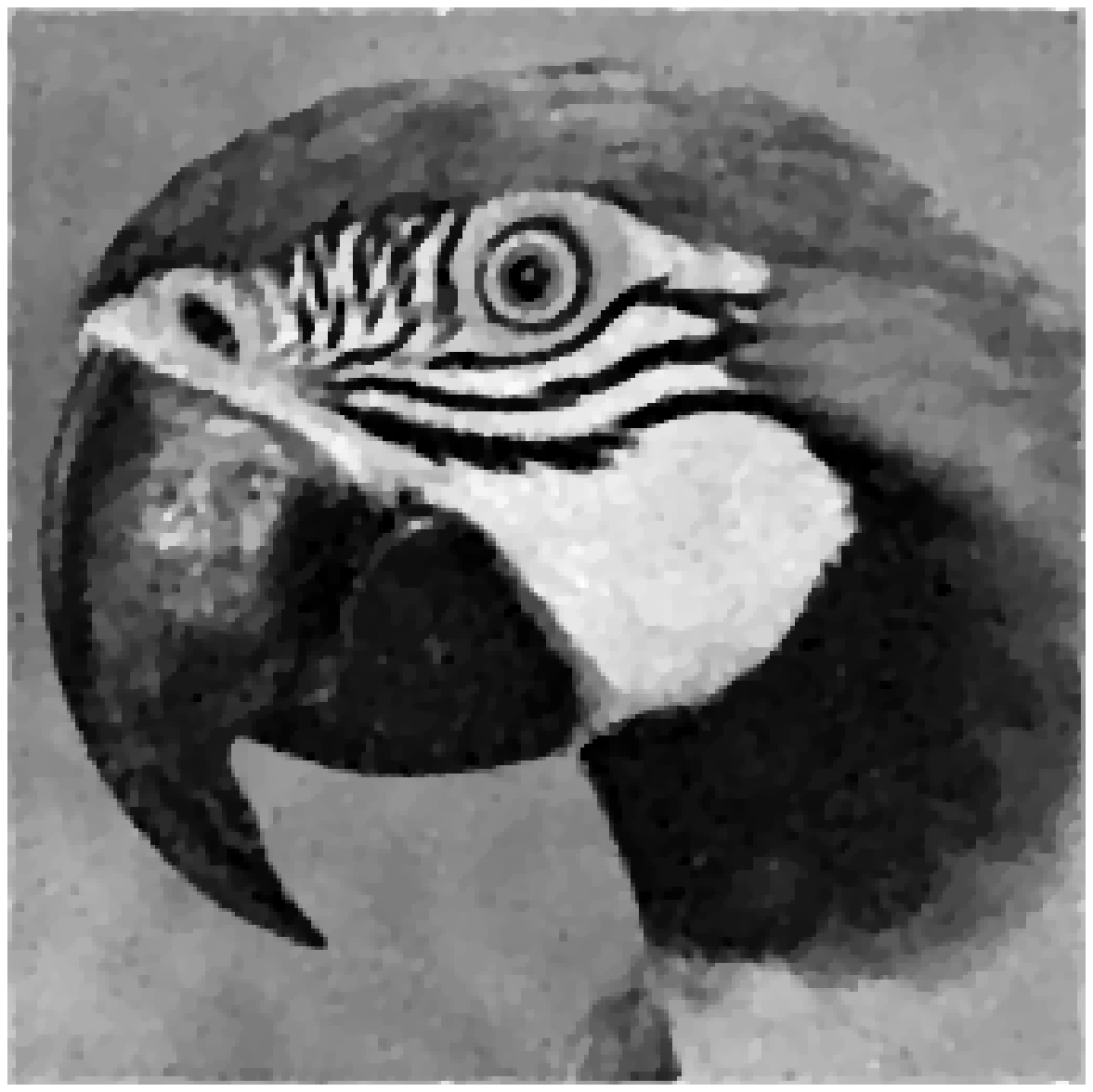}\\
\includegraphics[scale=0.2]{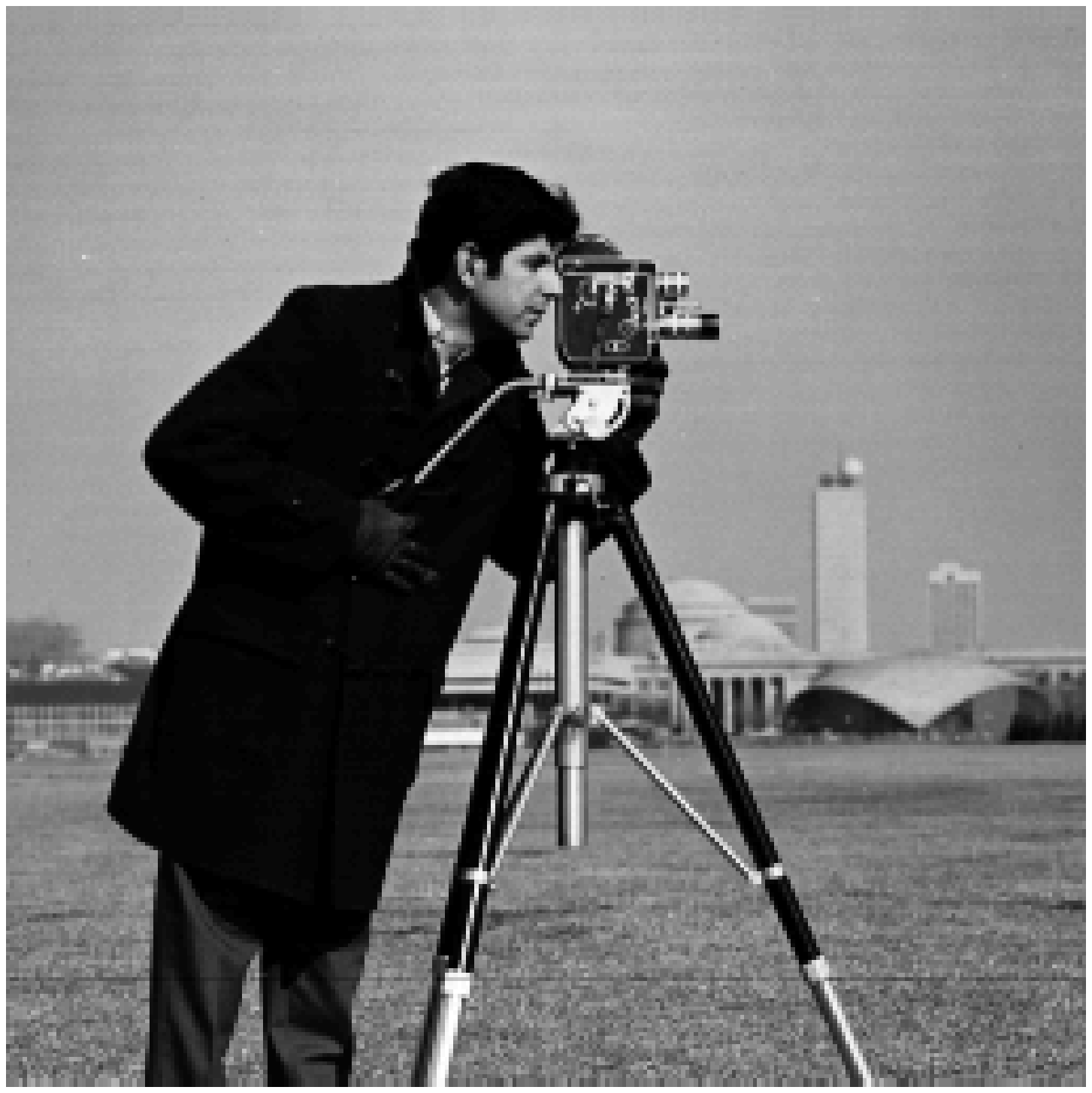}&\includegraphics[scale=0.2]{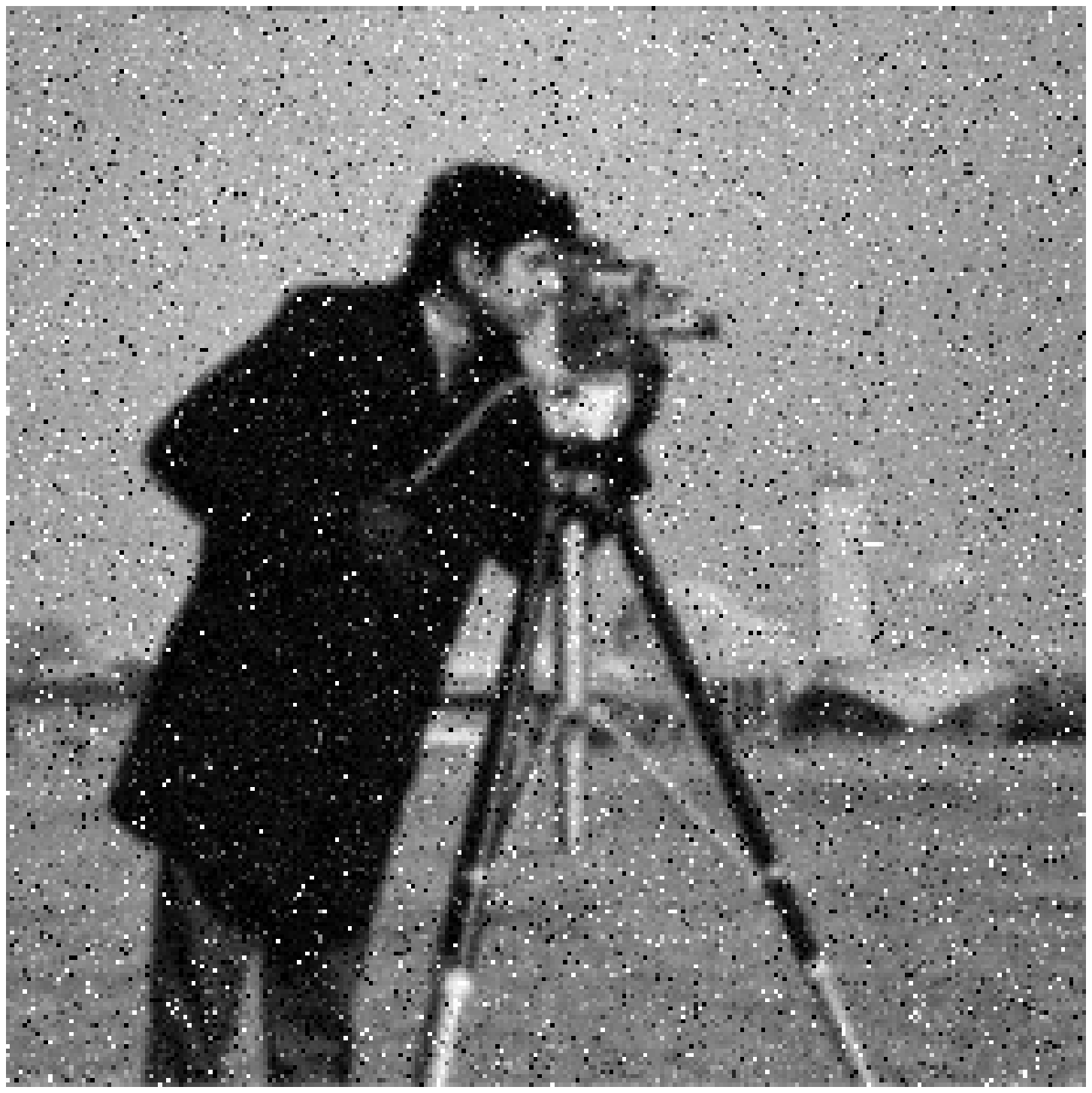}&\includegraphics[scale=0.25]{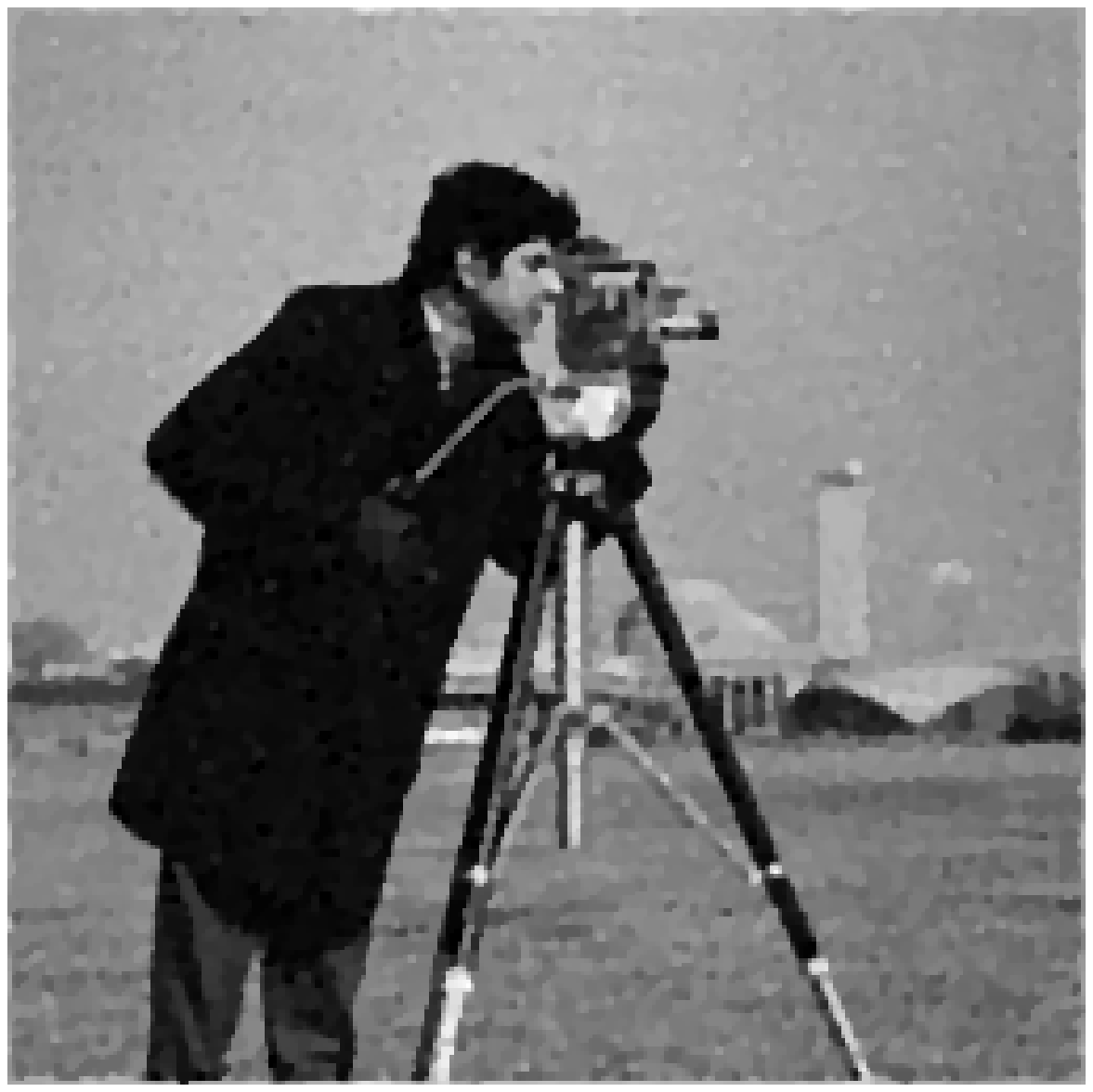}\\
\end{tabular}
\caption{Cauchy noise image deblurring datasets: original objects (left), blurred and noisy images (middle) and VMILAn reconstructions (right).}\label{fig:fig7}
\end{center}
\end{figure}
The regularization parameter $\lambda$ has been fixed equal to $0.35$. We applied VMILAn by computing the proximal point $\tilde{y}^{(k)}$ inexactly by means of the FISTA algorithm as in Section \ref{sec_exp1}. \silviacorr{Besides the Euclidean metric $D_k=I$ we consider two nontrivial choices of the scaling matrix.} In particular, we consider the diagonal scaling matrix whose generic $i-th$ element is defined as
\begin{equation}\label{SG_CauchyNoise}
 (D_k)^{-1}_{ii} = \max\left\{\min\left\{\frac{x^{(k)}_i}{V_i(x^{(k)})},\mu\right\}, \frac{1}{\mu}\right\}
\end{equation}
where $V(x^{(k)}) = \lambda H^Ts^{(k)}$ with $s^{(k)}_i = \frac{(Hx^{(k)})_i}{\gamma^2 + (Hx^{(k)} - g)^2_i}$. The scaling matrix defined in \eqref{SG_CauchyNoise} follows the gradient splitting idea already mentioned in Section \ref{sec_exp1}: the positivity of $V(x^{(k)})$ is ensured by the non-negative constraints and the properties of the blurring operator. \silviacorr{The other choice of the scaling matrix is $(D_k)_{ii}^{-1} = \max\{\min\{(A_k)_{ii},\mu\},\frac{1}{\mu}\}$ where the matrix $A_k$ is borrowed by the MM approach and it is given by formula (36) in \cite{Chouzenoux-etal-2014} where $\varepsilon = 0$ and the function $\omega$ is set equal to the function $\nu$ in the tenth row of Table 1 in \cite{Chouzenoux-Pesquet-2016}. In the following we will refer to the three scaling matrices described above as I, SG and MM respectively.}\\
The other parameters are set exactly as in section \ref{sec_exp1}.\\ 
In figure \ref{fig:fig8} we show the relative distance between the objective function values and the limit value $f^*$ computed by 5000 iterations of VMILAn with the MM metric. The benefits gained by using a variable metric are quite evident in terms of both number of iterations and computational time.\\
\silviacorr{As further benchmark we include in our comparison also the method VMFB where the majorant function is computed according to Lemma 5.1 in \cite{Chouzenoux-etal-2014} and \cite[Table 1]{Chouzenoux-Pesquet-2016}.}
\begin{figure}[htbp!]
\begin{center}
\begin{tabular}{cc}
\includegraphics[scale=0.25]{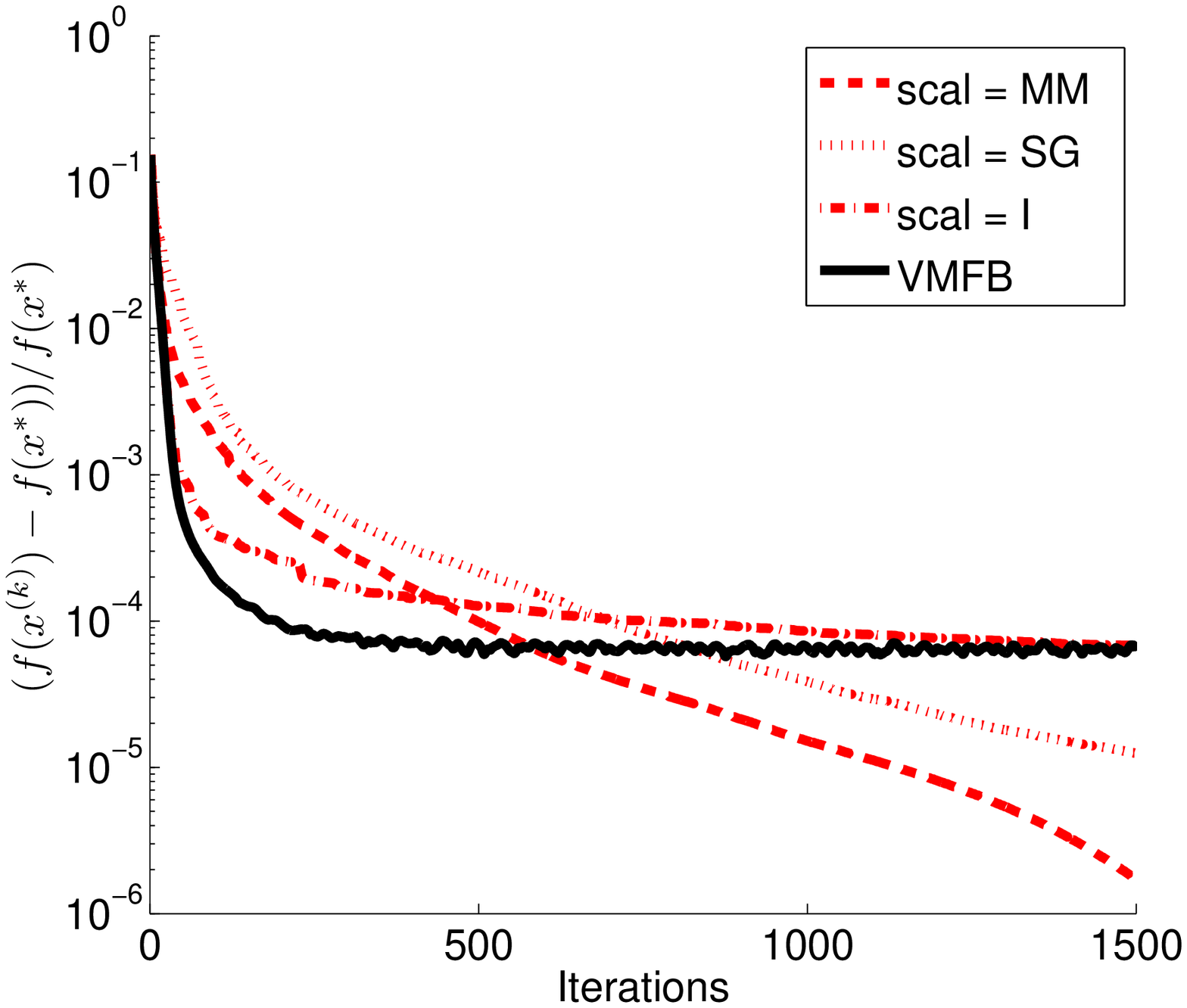}&\includegraphics[scale=0.25]{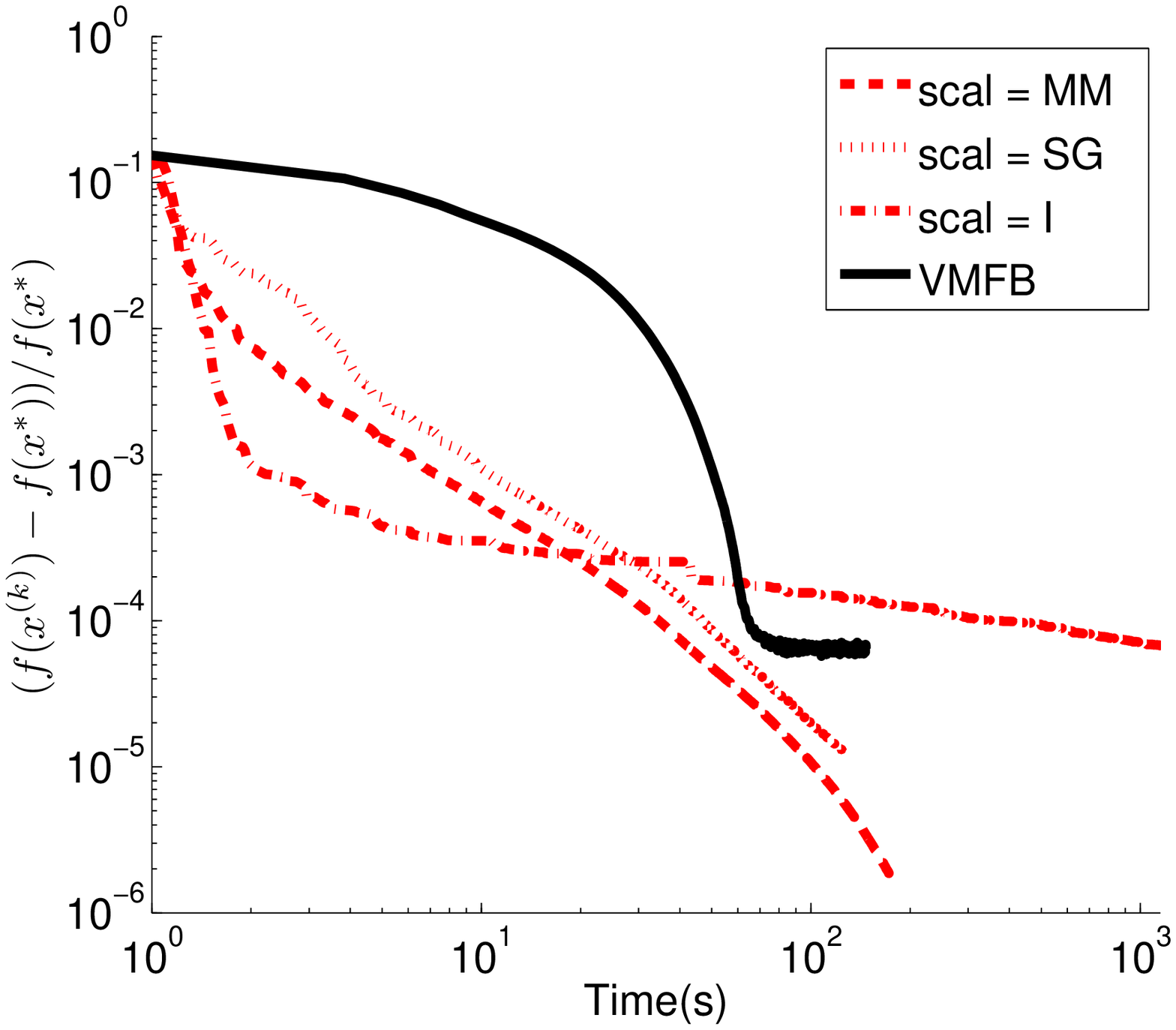}\\
\includegraphics[scale=0.25]{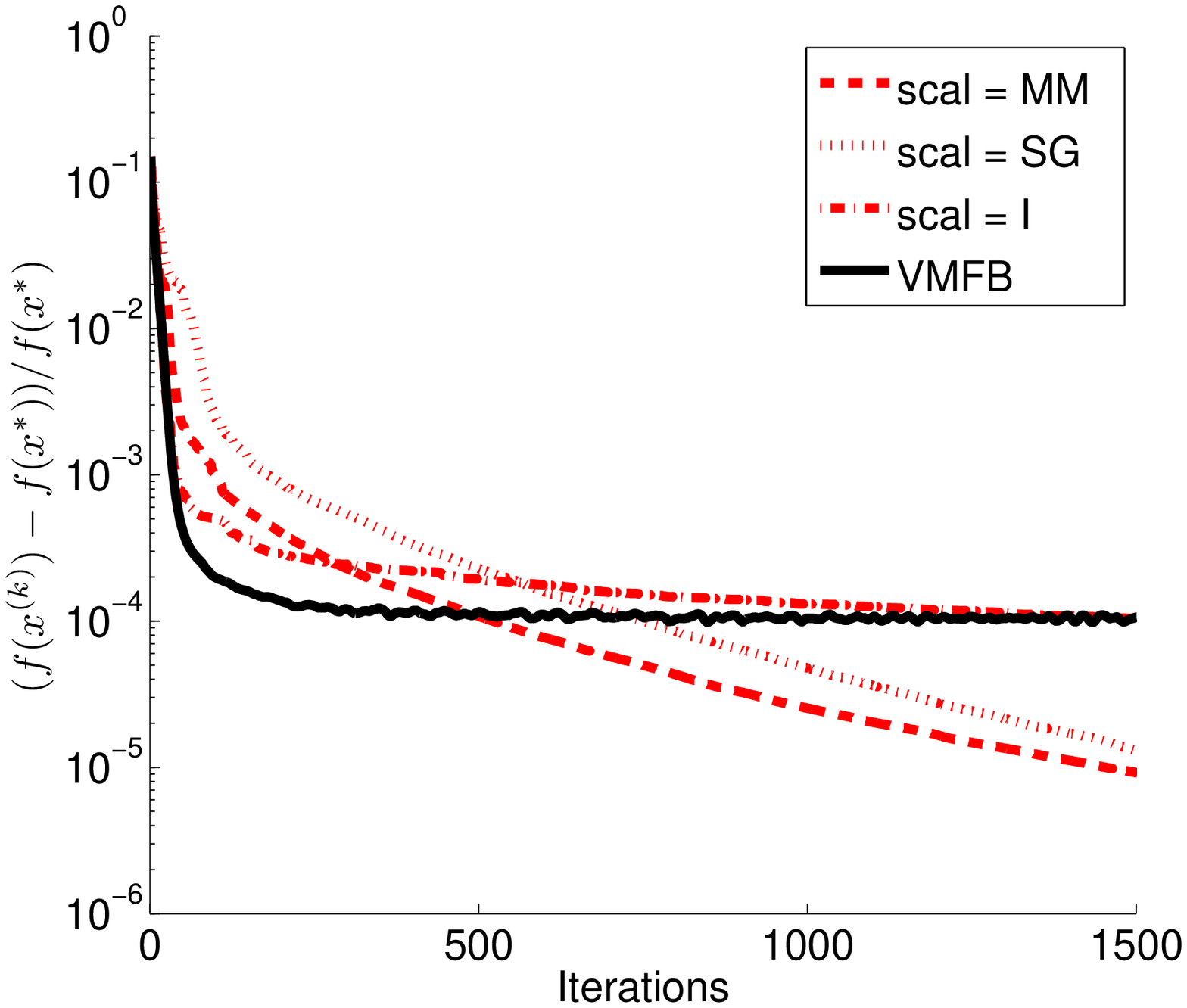}&\includegraphics[scale=0.25]{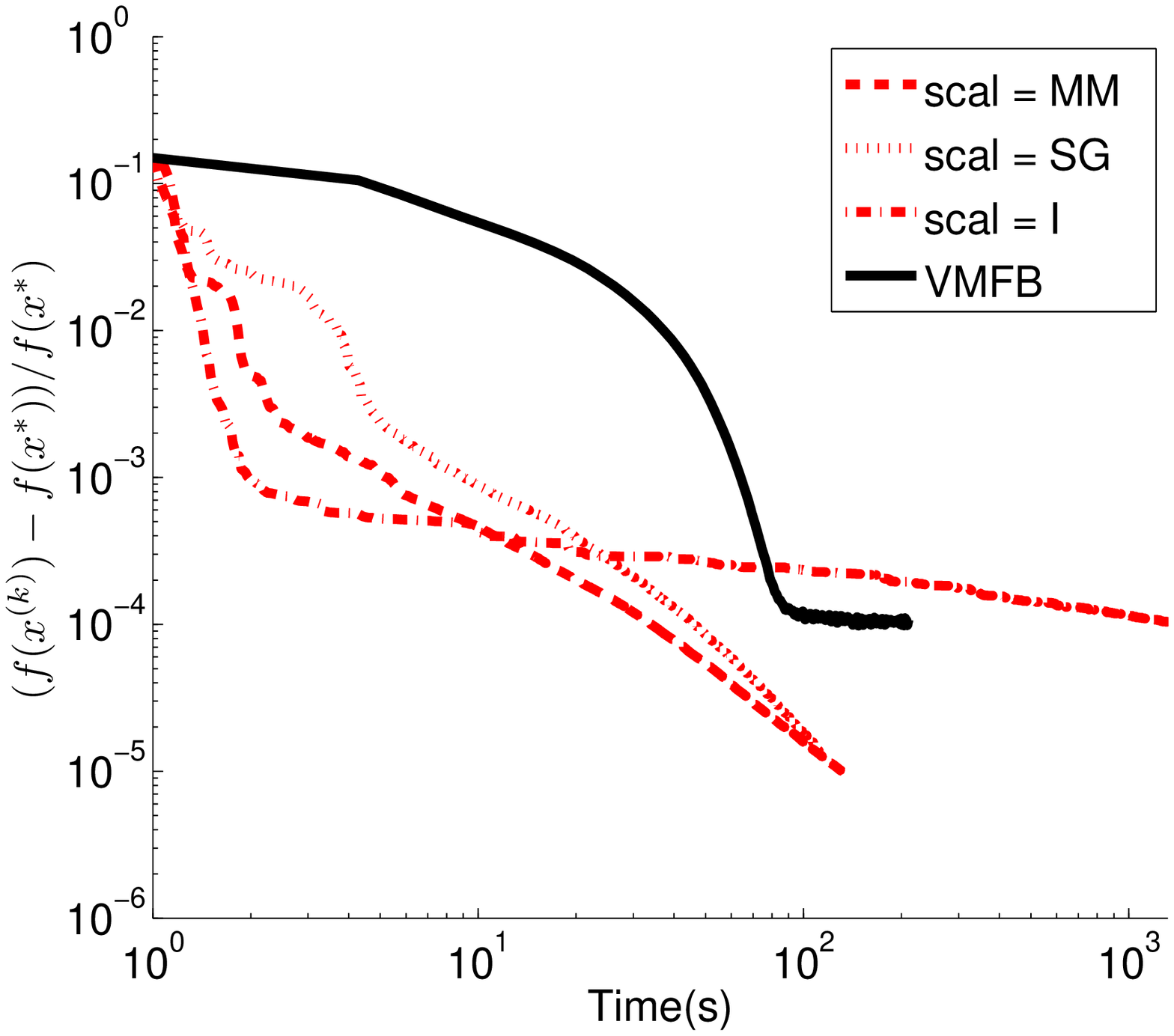}
\end{tabular}
\caption{Relative decrease of the objective function toward the minimum value with respect to the iteration number (left) and
computational time in seconds (right) for the Cauchy noise image restoration datasets: parrot (first row) and cameraman (second row).}\label{fig:fig8}
\end{center}
\end{figure}

However, to appreciate the validity of VMILAn as restoration method, in table \ref{tab:PSNR} we report the values of the peak signal-to-noise ratio (PSNR) related to the approximated solutions compared to the values shown in \cite{Sciacchitano-etal-2015} corresponding to the same two datasets. The PSNR is widely used in the literature to measure the image quality and is defined as
\begin{equation*}
{\rm{PSNR}}(x) = 10\log_{10}\frac{n|\max(x) - \min(x)|^2}{\|x_{\rm true} - x\|^2},
\end{equation*}
where $x_{\rm true}\in\mathbb{R}^{n}$ is the true object.
\begin{table}[htbp!]
\begin{center}
\begin{tabular}{lcccccc}
 & Data & VMILAn(I)&VMILAn(SG)&VMILAn(MM)&VMFB& \cite{Sciacchitano-etal-2015}\\
 \hline
Parrot & 18.23 & 26.67 & 26.70& 26.71&26.62& 26.79\\
Cameraman & 18.29 & 25.90 & 26.41 & 26.52 & 25.82 & 26.72 
\end{tabular}
\caption{PSNR values obtained by VMILAn in solving the Cauchy noise image restoration problems.}\label{tab:PSNR}
\end{center}
\end{table}
The PSNR values presented in table \ref{tab:PSNR} allow to say that the performances of VMILAn are comparable to those of the reference approach \cite{Sciacchitano-etal-2015}. The reconstructed images obtained with VMILAn (scal = MM) and related to the PSNR reported in table \ref{tab:PSNR} are shown in the right panel of figure \ref{fig:fig7}.

\section{Conclusions}\label{sec5}

In this paper we considered a variable metric linesearch based proximal-gradient algorithm recently proposed in \cite{Bonettini-Loris-Porta-Prato-2015} for the minimization of a class of nonconvex and nonsmooth functions. We revised the convergence analysis of this algorithm under the hypothesis of the objective function satisfying the KL property at each point of its domain, showing that any limit point is stationary and the sequence generated by the method converges to it. Since the KL requirements are quite general and are fulfilled by a large variety of functions, this result allows to generalize a similar one which was proved in \cite{Bonettini-Loris-Porta-Prato-2015} only for convex functions. In the second part of the paper we presented the results obtained by applying the considered algorithm in several numerical experiments dealing with nonconvex optimization problems in image processing. The comparison with other commonly used approaches demonstrated the efficiency of our method in terms of both speed of convergence and quality of the results.

\appendix
\section{Proof of Lemma \ref{keylemma}}\label{appendixA}
We first prove the following inequality
\begin{equation}\label{new_work2}
\frac{1}{4\alpha_{\max}\mu}\|\tyk-\xk\|^2 \leq  -(1+\tau)\thk(\tyk).
\end{equation}
We recall that $\hk$ is $\frac 1 {\alpha_k}$ strongly convex with respect to the norm induced by $D_k$, i.e.
\begin{equation}\label{strongly_convex}
\hk(\x)\geq \hk(\y) + \w^T(\x-y) +\frac 1{2\alpha_k}\|\x-\y\|^2_{D_k}, \ \ \ \forall w\in\partial \hk(\y).
\end{equation}
Since $\yk$ is the solution of \eqref{minhk} and, thus, $0\in\partial \hk(\yk)$, from the previous inequality with $\x=\tyk$ and $\y = \yk$ we have $\frac{1}{2\ak}\|\tyk-\yk\|^2_{D_k} \leq \hk(\tyk)-\hk(\yk)$ which, in view of \eqref{inexact}, gives
\begin{equation}\label{inexact_dist}
\frac{1}{2\ak}\|\tyk-\yk\|^2_{D_k}\leq -\frac \tau 2 \thk(\tyk).
\end{equation}
Exploiting again \eqref{strongly_convex} with $\x=\xk$ and $\y=\yk$, recalling that $\hk(\xk)=0$, we obtain
\begin{equation}\nonumber
\hk(\yk)\leq - \frac{1}{2\alpha_k}\|\xk-\yk\|_{D_k}^2.
\end{equation}
Combining the last inequality with \eqref{inexact} and using $\thk(\tyk)\leq\hk(\tyk)$ we obtain
\begin{equation*}
\frac{1}{2\ak}\|\xk-\yk\|^2_{D_k}\leq -\left(1+\frac \tau 2\right)\thk(\tyk).
\end{equation*}
By combining the triangle inequality with the previous one we obtain
\begin{eqnarray*}
\|\xk-\tyk\|_{D_k} &\leq& \|\xk-\yk\|_{D_k} +\|\yk-\tyk\|_{D_k}\\
&\leq& \sqrt{-(2+\tau)\ak\thk(\tyk)}+\|\yk-\tyk\|_{D_k}
\end{eqnarray*}
which yields
\begin{eqnarray*}
\|\xk-\tyk\|_{D_k}^2 &\leq&  {-(2+\tau)\ak\thk(\tyk)}+\|\yk-\tyk\|_{D_k}^2+\\
 & & \ + 2\|\yk-\tyk\|_{D_k}\sqrt{-(2+\tau)\ak\thk(\tyk)}\\
&\leq& -2(2+\tau)\ak\thk(\tyk) + 2\|\yk-\tyk\|_{D_k}^2,
\end{eqnarray*}
where the last inequality follows from $2\sqrt{uv}\leq u + v$. Combining it with \eqref{inexact_dist} gives
\begin{eqnarray*}
\frac{1}{4\ak}\|\xk-\tyk\|_{D_k}^2 &\leq& -(1+\frac \tau 2)\thk(\tyk)  + \frac{1}{2\ak}\|\yk-\tyk\|_{D_k}^2\\
&\leq& -(1+\tau)\thk(\tyk).
\end{eqnarray*}
Finally, \eqref{new_work2} follows from $\frac{1}{4\ak}\|\xk-\tyk\|_{D_k}^2\geq \frac{1}{4\alpha_{\max}\mu}\|\xk-\tyk\|^2$.\\
We are now ready for giving the proof of \eqref{lambdastar}--\eqref{HH3bis}.\\
Since $\nabla f_0$ is $L-$Lipschitz continuous, we can combine the descent lemma with \eqref{new_work2} exactly as in \cite[Lemma 3.3, Proposition 3.2]{Bonettini-Loris-Porta-Prato-2015}, and conclude that there exist $c\in \R_{>0}$ and $\lambda_{\min}\in (0,1]$ such that
\begin{equation}\label{new_work3}
 f(\xk+\lambda \dk) \leq  f(\xk) +\lambda\left(1-cL(1+\tau)\lambda\right) \thk(\tyk), \ \forall \ \lambda\in[0,1]
\end{equation}
and \eqref{lambdastar} hold.
Then, \cite[Lemma 3.1]{Bonettini-Loris-Porta-Prato-2015} and \eqref{lambdastar} directly yields \eqref{new_work5}.\\
Let us prove \eqref{HH1}. Combining \eqref{new_work2} with the backtracking rule \eqref{Armijo} immediately yields
\begin{equation}
f(x^{(k)}+\lambda_k d^{(k)})\leq f(x^{(k)})-\frac{\beta\lambda_k}{4\alpha_{\max}\mu(1+\tau)}\|\tyk-x^{(k)}\|^2.
\label{almostH1}
\end{equation}
Because of \eqref{SGPmodified}, it is either $x^{(k+1)}=\tyk$ or $x^{(k+1)}=x^{(k)}+\lambda_k d^{(k)}$. In both cases, since $\lamk\in[\lambda_{\min},1]$, we have
\begin{equation}\label{dk_ine}
\|\xkk-\xk\|\leq \|\tyk-\xk\| 
\end{equation}
which leads to
\begin{equation}\nonumber
f(x^{(k)}+\lambda_k d^{(k)})\leq f(x^{(k)})-\frac{\beta\lambda_{\min}}{4\alpha_{\max}\mu(1+\tau)}\|x^{(k+1)}-x^{(k)}\|^2.
\end{equation}
Then, \eqref{HH1} follows by taking $a=\frac{\beta\lambda_{\min}}{4\alpha_{\max}\mu(1+\tau)}$ and using \ref{step5} of Algorithm \ref{algo:nuSGP} which implies $f(x^{(k+1)})\leq f(x^{(k)}+\lambda_k d^{(k)})$.\\
In order to show that \eqref{HH2} holds, consider the right inequality in \eqref{new_work3} with $\lambda=1$. If $1-cL(1+\tau)\geq 0$, then the right inequality of condition \eqref{HH2} follows with $\eta_k\equiv 0$, while if $1-cL(1+\tau)< 0$, then the inequality is satisfied by setting $\eta_k=\left(1-cL(1+\tau)\right) \thk(\tyk)$ and observing that \eqref{new_work5} guarantees that $\lim_{k \to \infty}\eta_k = 0$.
The left inequality of \eqref{HH2} follows from the definition of $\xkk$ at \ref{step5} of Algorithm \ref{algo:nuSGP}.\\
In the following we prove \eqref{HH3bis}. By rewriting function $h^{(k)}$ as
\begin{equation*}
h^{(k)}(y)=f_1(y)+\frac{1}{2\alpha_k}\|y-z^{(k)}\|_{D_k}^2-\frac{\alpha_k}{2}\|\nabla f_0(\xk)\|_{D_k^{-1}}^2-f_1(\xk),
\end{equation*}
where $\zk = \xk-\ak D_k^{-1}\nabla f_0(\xk)$,
we can apply \cite[Theorem 2.8.7]{Zalinescu-2002} and \cite[Chapter XI, Equation 1.2.5]{Hiriart-1993} to compute the $\epsilon_k$-subdifferential of $h^{(k)}$:
\begin{eqnarray}
\partial_{\epsilon_k}h^{(k)}(y)&=\!\!\bigcup \limits_{0\leq \bar{\epsilon}_k+\hat{\epsilon}_k\leq \epsilon_k}\!\!\partial_{\bar{\epsilon}_k}f_1(y)+\partial_{\hat{\epsilon}_k}\left(\frac{1}{2\alpha_k}\|y-z^{(k)}\|_{D_k}^2\right)\nonumber\\
&=\!\!\bigcup \limits_{0\leq \bar{\epsilon}_k+\hat{\epsilon}_k\leq \epsilon_k}\!\!\partial_{\bar{\epsilon}_k}f_1(y)+\left\{\frac{1}{\alpha_k}D_k(y-z^{(k)}+e): \frac{\|e\|_{D_k}^2}{2\alpha_k}\leq \hat{\epsilon}_k\right\}.\label{epssubrule}
\end{eqnarray}
The point $\tyk$ satisfies condition \eqref{inexact} if and only if $0\in \partial_{\epsilon_k}h^{(k)}(\tyk)$, where $\epsilon_k = -\frac\tau 2 \thk(\tyk)$. Thanks to \eqref{epssubrule}, this ensures that there exist $\bar{\epsilon}_k,\hat{\epsilon}_k$ as above, $e^{(k)}\in \R^n$ satisfying $\frac{\|e^{(k)}\|_{D_k}^2}{2\alpha_k}\leq \hat{\epsilon}_k$ and $w^{(k)}\in \partial_{\bar{\epsilon}_k}f_1(\tyk)$ such that
\begin{equation}\label{perturbed_wk}
w^{(k)}=\frac{1}{\alpha_k}D_k(z^{(k)}-\tyk+e^{(k)}).
\end{equation}
Set $v^{(k)}=\nabla f_0(\tyk)+w^{(k)}$. By using the Lipschitz continuity of $\nabla f_0$, the fact that $\alpha_k\in [\alpha_{\min},\alpha_{\max}]$ and $D_k\in \mathcal{M}_{\mu}$, we have:
\begin{eqnarray*}
\|v^{(k)}\|&=\|\nabla f_0(\tyk)+\frac{1}{\alpha_k}D_k(\xk-\alpha_k D_k^{-1}\nabla f_0(\xk)-\tyk+e^{(k)})\|=\\
&=\|\nabla f_0(\tyk)-\nabla f_0(\xk)+\frac{1}{\alpha_k}D_k(\xk-\tyk+e^{(k)})\|\\
&\leq L\|\xk-\tyk\|+\frac{\mu}{\alpha_k}(\|\xk-\tyk\|+\|e^{(k)}\|)\\
&\leq \left(L+\frac{\mu}{\alpha_{\min}}\right)\|\xk-\tyk\|+\frac{\mu}{\alpha_{\min}}\sqrt{\mu}\|e^{(k)}\|_{D_k}\\
&\leq \frac{1}{\lambda_{\min}}\left(L+\frac{\mu}{\alpha_{\min}}\right)\|\xkk-\xk\|+\left(\frac{\sqrt{2\mu^3\alpha_{\max}}}{\alpha_{\min}}\right)\sqrt{\hat{\epsilon}_k}.
\end{eqnarray*}
The thesis follows by choosing $b=\frac{1}{\lambda_{\min}}\left(L+\frac{\mu}{\alpha_{\min}}\right)$, $\zeta_k=\left(\frac{\sqrt{2\mu^3\alpha_{\max}}}{\alpha_{\min}}\right)\sqrt{\hat{\epsilon}_k}$ for all $k\in \N$ and by observing that, since
\begin{equation*}
0\leq\zeta_k \leq \left(\frac{\sqrt{2\mu^3\alpha_{\max}}}{\alpha_{\min}}\right)\sqrt{\epsilon_k}= \left(\frac{\sqrt{2\mu^3\alpha_{\max}}}{\alpha_{\min}}\right)\sqrt{-\frac{\tau}{2}\thk(\tyk)}
\end{equation*}
and, because of \eqref{new_work5}, $\lim\limits_{k\rightarrow \infty} \thk(\tyk)=0$, then also $\lim\limits_{k\rightarrow \infty}\zeta_k=0$.

\section{Proof of Proposition \ref{real_lemma5}}\label{appendixB}

Since $f$ is lower semicontinuous and bounded from below, and $\{f(\xk)\}_\kinN$, from \eqref{HH1}, is monotone nonincreasing, we have that $ \lim_{k\to\infty} f(\xk)$ exists and $f(\bar\x) \leq \lim_{k\to\infty} f(\xk)$. Let us show that also the opposite inequality holds. By summing inequality \eqref{HH1} from $k=0$ to $N$ we obtain
\begin{eqnarray*}
a\sum_{k=0}^N \|\xkk-\xk\|^2 &\leq& \sum_{k=0}^N f(\xk)-f(\xkk)= f(\x^{(0)}) - f(\x^{(N+1)}).
\end{eqnarray*}
Taking limits for $N\rightarrow\infty$ on both sides gives
\begin{equation}\label{serie}
a\sum_{k=0}^\infty \|\xkk-\xk\|^2 \leq f(\x^{(0)})- f(\bar\x)<\infty \Rightarrow\lim_{k\to\infty}\|\xkk-\xk\|=0.
\end{equation}
\silviacorr{Let $\vk= \nabla f_0(\tyk) + \wk$, with $\wk\in\partial_{\bar\epsilon_k} f_1(\tyk)$, $\bar\epsilon_k\leq -\frac \tau 2 \thk(\tyk)$ satisfying inequality \eqref{HH3bis}. Then, by combining \eqref{HH3bis} and \eqref{serie} we obtain
\begin{equation}\label{limvk}
\lim_{k\to\infty}\nabla f_0(\tyk) +\wk = \lim_{k\to\infty} \vk = 0.
\end{equation}
Let $\{\x^{(k_j)}\}_{j\in\N}$ be a subsequence of $\{\xk\}_\kinN$ such that $\lim_{j\to\infty}\x^{(k_j)} = \bar\x$. Using \ref{step5} of Algorithm \ref{algo:nuSGP} and recalling that $\lambda_k\in[\lambda_{\min},1]$, we have
\begin{equation}\label{ykconv0}
\lambda_{\min}^2\|\tyk-\xk\|^2 \leq  \lamk^2\|\tyk-\xk\|^2 \leq \|\xkk-\xk\|^2.
\end{equation}
Inequality \eqref{ykconv0}, combined with \eqref{serie}, gives $\lim_{k\to\infty} \|\tyk-\xk\| = 0$. Then, we also have $\lim_{j\to\infty}\ty^{(k_j)} = \bar\x$. Thus, by \eqref{limvk} and by continuity of $\nabla f_0$, we can write
\begin{equation}\label{technical}
\lim_{j\to\infty} \w^{(k_j)} = -\nabla f_0(\bar\x).
\end{equation}
Since $\w^{(k_j)}\in\partial_{\bar\epsilon_k} f_1(\ty^{(k_j)})$, we have
\begin{eqnarray}
f_1(\bar\x)&\geq&f_1(\ty^{(k_j)}) + (\bar\x - \ty^{(k_j)})^T \w^{(k_j)}-{\bar\epsilon_{k_j}}\nonumber\\
&\geq& f(\x^{(k_{j}+1)})-f_0(\ty^{(k_j)})+ (\bar\x - \ty^{(k_j)})^T \w^{(k_j)}-{\bar\epsilon_{k_j}},\label{exploit_convexity}
\end{eqnarray}
where the second inequality follows from $ f(\x^{(k_j+1)}) \leq f(\ty^{(k_j)})=f_0(\ty^{(k_j)})+f_1(\ty^{(k_j)})$. Taking the limit of the right-hand-side for $j\to\infty$, and recalling \eqref{new_work5} which implies $\lim_{j\to\infty}\bar \epsilon_{k_j}=0$, we obtain
\begin{equation*}
f_1(\bar\x) \geq \lim_{j\to \infty} f(\x^{(k_j+1)}) -f_0(\bar\x) = \lim_{k\to \infty} f(\xk) -f_0(\bar\x)
\end{equation*}
which reads also as $f(\bar\x)\geq \lim_{k\to \infty} f(\xk)$ and completes the first part of the proof.\\
As for the second part, since $\lim_{j\to\infty}\ty^{(k_j)} = \bar\x$, $\lim_{j\to\infty}\bar \epsilon_{k_j}=0$ and \eqref{technical} holds, we can apply Remark \ref{rem3}(iii) and thus obtain
\begin{equation}
-\nabla f_0(\bar\x)\in \partial f_1(\bar\x)
\end{equation}
which is equivalent to $0\in\partial f(\bar\x)$.}

%

\section*{Acknowledgments}

We would like to thank the anonymous reviewers for their valuable remarks and suggestions. This work has been partially supported by MIUR under the two projects FIRB - Futuro in Ricerca 2012, contract RBFR12M3AC and PRIN 2012, contract 2012MTE38N. I. Loris is a Research Associate of the Fonds de la Recherche Scientifique - FNRS. The Italian GNCS - INdAM is also acknowledged, as well as a ULB ARC grant.

\section*{References}

\bibliographystyle{unsrt}
\bibliography{paper,nonconvex}

\end{document}